\documentclass[11pt, reqno]{article}
\usepackage{fullpage}
\usepackage{enumerate}
\usepackage{setspace}
\usepackage{graphicx}
\usepackage{float}

\usepackage{lipsum} 
\usepackage[normalem]{ulem}

\usepackage{upgreek, mathtools,bm}

\usepackage{amsmath,amsfonts,amssymb,amsthm, comment}
\usepackage{hyperref}
\hypersetup{
    colorlinks=true,
    linkcolor=blue,
    citecolor=red,
    urlcolor=blue,
    pdfborder={0 0 0}
}
\usepackage{yfonts}
\usepackage{enumitem}
\usepackage{xcolor}
\usepackage{bbm}
\usepackage{wrapfig}

\usepackage[font=sf, labelfont={sf,bf}, margin=1cm]{caption}

\usepackage{cleveref}
  \crefname{theorem}{Theorem}{Theorems}
  \crefname{thm}{Theorem}{Theorems}
  \crefname{thm*}{Theorem*}{Theorems}
  \crefname{lemma}{Lemma}{Lemmas}
  \crefname{lem}{Lemma}{Lemmas}
  \crefname{remark}{Remark}{Remarks}
  \crefname{prop}{Proposition}{Propositions}
\crefname{notation}{Notation}{Notations}
\crefname{claim}{Claim}{Claims}
  \crefname{defn}{Definition}{Definitions}
  \crefname{corollary}{Corollary}{Corollaries}
  \crefname{section}{Section}{Sections}
  \crefname{figure}{Figure}{Figures}
    \crefname{assumption}{Assumption}{Assumptions}

\newtheorem{thm}{Theorem}[section]
\newtheorem{thm*}{Theorem*}[section]

\newtheorem{lemma}[thm]{Lemma}
\newtheorem{corollary}[thm]{Corollary}
\newtheorem{prop}[thm]{Proposition}
\newtheorem{defn}[thm]{Definition}

\newtheorem{con}[thm]{Conjecture}

\numberwithin{equation}{section}

\theoremstyle{definition}
\newtheorem{remark}[thm]{Remark}

\def\cT{\mathcal{T}}

\def\cP{\mathcal{P}}

\def\cN{\mathcal{N}}
\def\cM{\mathcal{M}}
\def\cL{\mathcal{L}}

\def\cI{\mathcal{I}}
\def\cH{\mathcal{H}}
\def\cG{\mathcal{G}}
\def\cF{\mathcal{F}}
\def\cE{\mathcal{E}}

\def\cC{\mathcal{C}}
\def\cB{\mathcal{B}}
\def\cA{\mathcal{A}}

\def \bb {\beta}

\def\I{\mathbf{I}}

\def\btt{B(\cT_2)}
\def\bt3{B(\cT_3)}
\def\btc{B(\cT^c)}

\def\P{\mathbb{P}}

\def\E{\mathbb{E}}
\def\C{\mathbb{C}}
\def\R{\mathbb{R}}

\def\S{\mathbb{S}}

\def  \p- {p\textunderscore}

\def\ep{\varepsilon}

\newcommand{\note}[1]{{\color{red}{[note: #1]}}}

\newcommand{\abs}[1]{ \left \lvert #1 \right \rvert}

\def\Bnp{\cB_{np}}

\def\Bn{\cB_{n}}

\newcommand{\be}{\beta}

\newcommand{\de}{\delta}
\newcommand{\si}{\sigma}

\newcommand{\Om}{\Omega}

\newcommand{\lam}{\lambda}

\newcommand{\sub}{\subset}

\newcommand{\lr}[1]{\left( {#1}\right)}

\newcommand{\set}[1]{ \left\{ #1\right\}}

\newcommand{\inn}[1]{ \left\langle #1\right\rangle}

\newcommand{\1}{\mathbf{1}}
\newcommand{\e}{\mathbf{e}}

\newcommand{\Ls}{\mathcal{L}}

\newcommand{\sR}{\mbox{Supp } R}

\newcommand{\Tc}{\textbf{Col}}

\newcommand{\tnrm}[1]{{\left\vert\kern-0.25ex\left\vert\kern-0.25ex\left\vert #1 
    \right\vert\kern-0.25ex\right\vert\kern-0.25ex\right\vert}}

\begin{document}
 \title{Least singular values of shifted sparse random combinatorial matrices
}

\author{\textsc{Dongbin Li, Alexander E. Litvak,  and Tingzhou Yu}}
\date{}

\maketitle
\abstract{Let $M_n$ be an $n\times n$ random matrix with entries in $\{0, 1\}$, where each row is independently and uniformly sampled from the set of all vectors in $\{0, 1\}^n$ containing exactly $d$ ones.  we establish quantitative lower bounds on the smallest singular value of the shifted matrices $M_n-z \mathbf{I}_n$ whenever  $|z| \leq \sqrt{d}\, \log\log d$ and $ C \log n  \leq d \leq n/2$ for some absolute positive constant $C$. 
As an application, we show that the empirical spectral distribution of the appropriately rescaled matrix $M_n$ converges in probability to the circular law provided that $\log^{2+\ep} n \leq d=o(n)$ for some fixed $\ep \in (0,1)$.
}

\bigskip

{\small

\tableofcontents


\section{Introduction}\label{sec:intro}
Let $\cM_{n,d}$ consist of all $n\times n$ matrices with entries in $\{0,1\}$ where each row has a fixed sum of  $d$. Our random matrix $M_n$ is uniformly drawn from the set $\cM_{n,d}$. In some works such  matrices are called {\it random combinatorial matrices}, so we follow this terminology. Note that the rows of $M_n$ are independent random vectors. An element $M_n \in \cM_{n,d}$ can be interpreted as the adjacency matrix of a directed graph (digraph) on $n$ labeled vertices, where each vertex has exactly $d$ outgoing edges. Alternatively, $M_n$ can be viewed as the adjacency matrix of a bipartite graph between two disjoint sets of $n$ vertices each, with each vertex on the left having degree $d$. In this paper, we investigate the quantitative bounds on the least singular value of $M_n$ in the sparse regime $d= \Omega(\log n)$. Our main result is the following.

\begin{thm} \label{thm: lsvlowertails}
There exist absolute positive constants $C_1$, $C_2$, $C_3$, $C_{\ref{thm: lsvlowertails}}\geq 1$
such that the following holds. Let $n$ be a large enough integer and 
$C_{\ref{thm: lsvlowertails}}\log n \leq d \leq n/2$.  
    Let $M_n$ be an $n \times n$ random matrix uniformly drawn from the set $\cM_{n,d}$.
    Let $z\in \C$ satisfy $|z| \leq \sqrt{d}\, \log \log d$. Then 
    \[
    \mathbb{P}(s_n(M_n-z\I_{n}) \leq n^{-\bb}) 
   \leq C_1/\sqrt{d},
    \]
    for some $\bb =\bb(n, d) \in   (2,  C_2 \log \log n]$. Moreover, if  $d\geq(\log n)^{2}$ then $\beta <9$, 
    if $d\geq C_3 (\log n)^{21}$ then $\beta <5$. 
\end{thm}
\begin{remark}
    The lower bound on $d$ of the order of $\log n$ is optimal up to constants at $z=0$ in the sense that, with probability close to 1, $M_n$ contains a zero column if $d \leq (1-\ep)\log n$ (see Section~\ref{appx} for a proof).
\end{remark}
\begin{remark}\label{sing-comp-bb}
  In fact, our proof gives 
  $\mathbb{P}(s_n(M_n-z\I_{n}) \leq\alpha(n, d))    \leq C_1/\sqrt{d}$, where 
$$
   \mbox{(1)  if } \, \, d\geq c_0 n^{2/3} (\log n)^{1/3} \quad \mbox{ then } \quad  
    \alpha(n, d)\geq \frac{1}{2n \, d^{3/2}}\geq  \frac{1}{n^{5/2}};\quad\quad\quad 
    \quad\quad\quad\quad\quad \quad\quad\quad\quad\quad\quad\quad\quad\quad\quad
$$
$$
   \mbox{(2)  if } \, \,d\leq c_0 n^{2/3} (\log n)^{1/3} \quad \mbox{ then } \quad  
    \alpha(n, d)\geq \frac{ d^{3\alpha-5.5}}{n^{2 \alpha} (\log n)^{\alpha -1/2}}\geq
    n^{-2 \alpha};\quad\quad\quad \quad \quad\quad\quad\quad\quad\quad\quad\quad\quad\quad
$$
 for  
$$
  \alpha = \frac{2 \log (6d)}{\log (d/(C_2 \log n))} \in (2,  C_4 \log \log n] . 
$$
In particular, for large enough $n$, 
$$
     \mbox{(3)  if } \, \,  d\geq C_3 (\log n)^{21} 
    \quad \mbox{ then } \quad  \alpha < 2.1  \quad \mbox{ and } \quad  
    \alpha(n, d)\geq \frac{ d^{0.8}}{n^{4.2} (\log n)^{1.6}}\geq \frac{1}{n^{4.2} };
    \quad\quad\quad \quad  \quad\quad\quad\quad\quad\quad\quad\quad\quad\quad
$$
$$
     \mbox{(4)  if } \, \,  d\geq  (\log n)^{2} 
    \quad \mbox{ then } \quad  \alpha < 4.1  \quad \mbox{ and } \quad  
    \alpha(n, d)\geq \frac{ d^{6.8}}{n^{8.2} (\log n)^{3.6}}\geq \frac{1}{n^{8.2} }.
    \quad\quad\quad \quad  \quad\quad\quad\quad\quad\quad\quad\quad\quad\quad
$$
Here, $c_0$, $C_1$, $C_3$, $C_4$ are appropriate positive absolute constants. 
\end{remark}

Our result, on one hand, is intimately linked to the invertibility of the random combinatorial matrices, the study of which was initiated by Nguyen. In \cite{nguyen2013singularity}, he proved that if $d=n/2$, then $\mathbb{P} (s_{n}(M_n)=0) = O(n^{-C})$ for any $C>0$; moreover, in the same paper, it was conjectured that the sharp singularity probability is $(\frac{1}{2} +o(1))^{n}$. Nguyen's result has been strengthened by several authors \cite{ferber2021counting, jain2021approximate, tran2020smallest},  and finally, the sharp singularity probability was obtained by Jain, Sah and Sawhney \cite{jain2020sharpsing}. 
In the sparse case, it was Aigner-Horev and Person \cite{APsparseRCM} who first showed that $\mathbb{P}(s_n(M_n)=0)=o(1)$ under the assumption that $\lim_{n \rightarrow \infty} d/(n^{1/2}\log ^{3/2} n) = \infty$. Shortly after, it was shown by Ferber, Kwan and Sauermann \cite{ferber2022singularity} that $M_n$ is asymptotically almost surely non-singular when $\min (d,n-d) \geq (1+\ep)\log n$ for a fixed $\ep >0$. As for the quantitative side, in the dense regime $d=pn$ with $p\in (0, 1/2]$ fixed, Tran \cite{tran2020smallest} established a quantitative lower bound for $s_n(M_n)$. Later, a matching upper bound of order $s_n(M_n)=O(\sqrt{d}/n)$ was obtained in \cite{LLY-upper}, therefore identifying the correct scale $n^{-1/2}$ for the least singular value in the dense regime.
However, obtaining a quantitative lower tail bound on the smallest singular value in the sparse regime remains largely unsolved. In this paper, by setting $z=0$, we have that 
    $$
    \mathbb{P}(s_n(M_n) \leq n^{-\bb}) 
    \leq C_1/\sqrt{d}.
    $$
   To the best of our knowledge, this is the first time a quantitative lower bound for $s_n(M_n)$ has been obtained in the sparse regime $d\geq C \log n$.  

On the other hand, the quantitative estimate for $s_n(M_n-z\I_{n})$ in Theorem \ref{thm: lsvlowertails} enables us to establish the following circular law for random combinatorial matrices in the sparse regime $\min\{d, n-d\} \geq \log^{2+\ep} n $, which complements an earlier result of Nguyen and Vu \cite{NVcircgivensum} in the dense regime.

\begin{corollary}\label{thm:CLRCB}
Fix $\ep\in (0,1)$. For each $n\geq 1$ Let $d = d(n)$ satisfy  $\min\{d, n-d\} \geq \log^{2+\ep} n $ and 
$\min\{d, n-d\}/n \to 0$, and let $M_n$ be drawn uniformly from $\cM_{n,d}$.  Denote  
$$
  \overline{M_{n}}=\frac{M_n}{\sqrt{d(1-d/n)}}. 
$$ 
 Then the sequence of empirical spectral distributions $\left(\mu_{\overline{M_{n}}}\right)_{n\geq 1}$ 
  converges in probability to $\mu_{circ}$ as $n \rightarrow \infty$.
\end{corollary}

\begin{figure}[h!]
    \centering
   \includegraphics[width=0.8\textwidth, clip]{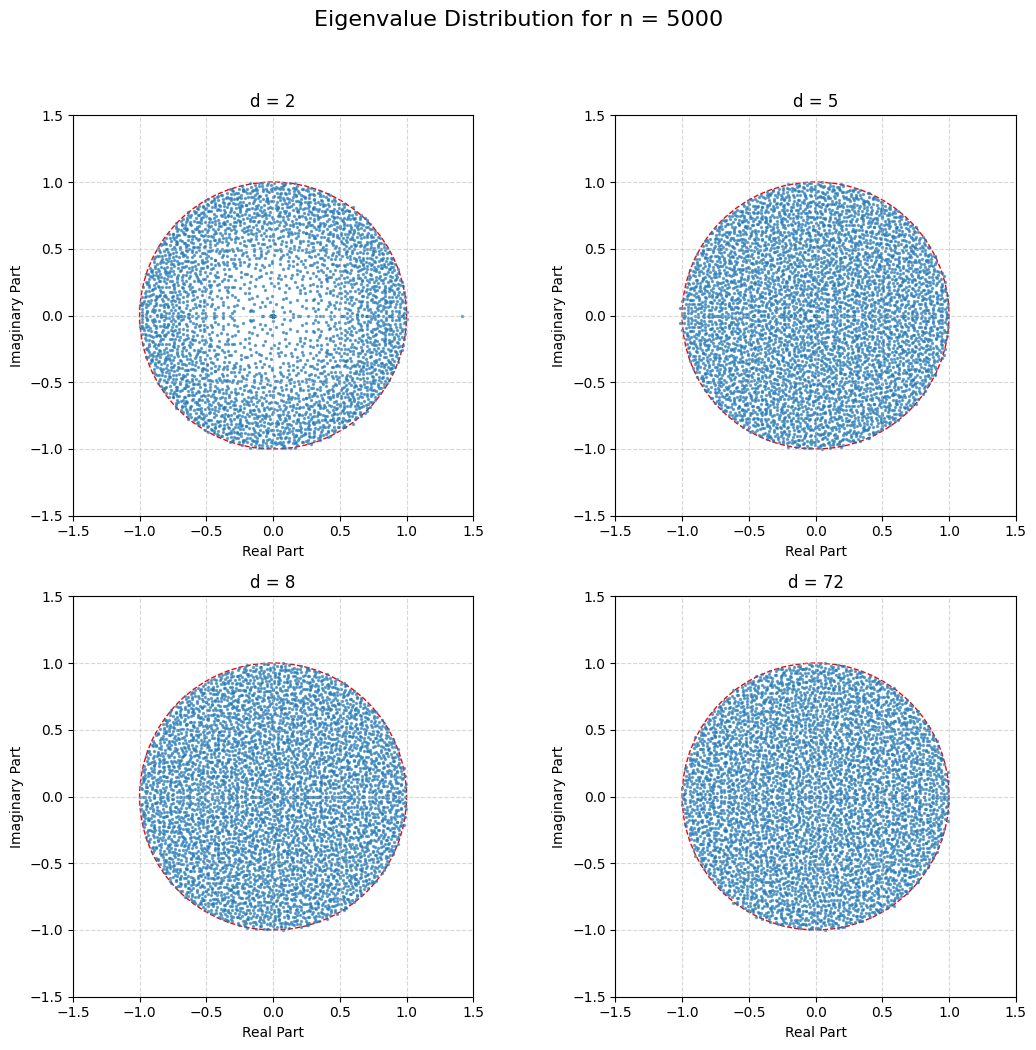}
    \caption{Empirical spectra of four independently generated $5000 \times 5000$ random combinatorial matrices where each row has exactly $d\in \{2, 5, 8,72\}$ ones. In each panel, the plotted points are the eigenvalues of the normalized matrix $\overline{M_{n}}=M_n/\sqrt{d(1-d/n)}$. The red dashed circle is the unit circle $\{z\in \mathbb{C}: |z|=1\}$. For $d=2$, the empirical spectral distribution is visibly inconsistent with the circular law. However, as $d$ increases, the spectrum becomes progressively closer to uniform on the unit disk.}
    \label{fig:enter-label}
\end{figure}

\begin{remark} 
Denote by $E_n$ the matrix with all entries equal to 1.  Similarly to \cite[Remark~1.5]{Cookcircdigraph}) 
(cf., \cite[Corollary~1.12]{TVuniversality}), we observe that $M_n$ satisfies circular law iff $E_n-M_n$ does.
Thus, we may assume that $d \leq n/2$.  
\end{remark}

\begin{remark}
The result for the dense regime $\min{(d, n-d)}=cn$  for some $c \in (0,1/2]$ was obtained by Nguyen and Vu \cite{NVcircgivensum}. In their work, random $\pm 1$ matrices $A_n$ with prescribed row sums $s$ satisfying $|s| \leq (1-\ep)n$ for some fixed $\ep \in (0,1]$ are considered. Since $(A_n+E_n)/2$ has the same distribution as our random combinatorial matrices $M_n$ with $d=(n +s)/2 \in [\ep n/2, (1-\ep/2)n]$. Our proof also works for the case $\min{(d, n-d)}=cn$.
\end{remark}

\begin{remark}\label{remark-conj}
   Denote by $X$ the number of  zero columns in the random matrix $M_n$. Clearly, $\mathbb{E}X=n(1-d/n)^n$, which implies that for any fixed $d$, the multiplicity of the zero eigenvalue  is bounded from below by a constant proportion of $n$ with a positive probability (see Section~\ref{appx} for a proof). Therefore, one cannot expect that the circular law holds when $d$ does not grow with $n$. This is also consistent with the numerical evidence shown in Figure~\ref{fig:enter-label}, where for $d=2, n=5000$ the empirical spectral distribution is visibly inconsistent with the circular law. On the other hand, as $d$ increases, the empirical spectrum becomes progressively closer to the unit disk. This suggests Corollary~\ref{thm:CLRCB} may remain valid as long as $d \rightarrow \infty$ with $n$ at any speed as indicated by the graphs above, where relatively large $d$ are chosen.  We formally state this as Conjecture~\ref{conjdtoinfty}. 
\end{remark}

The remainder of the introduction focuses on the main ideas behind the proof of Theorem~\ref{thm: lsvlowertails}, and for ease of exposition, we take $z=0$.
 An often employed approach to estimating the least singular value, introduced in \cite{LPRTsingular} (see also \cite{LPRTVRandom}), is to partition the complex unit sphere  and to work separately with different types of vectors. In this context, one usually splits the sphere  into vectors of small complexity (close to sparse vectors) and ``spread" vectors with a relatively small $\ell_{\infty}$-norm.  Such splitting was further formalized in \cite{rudelson2008littlewood} into a concept of ``compressible" and ``incompressible" vectors to obtain sharp lower tail bounds for the smallest singular value of random matrices with i.i.d. subgaussian entries. For non-Hermitian random matrices with dependent entries such as $d$-regular random matrices and random combinatorial matrices, the concept of compressible and incompressible vectors is not directly applicable. We use a different splitting by partitioning the complex unit sphere $\mathbb{S}^{2n-1}\subset \C^n$ into two parts. The first part consists of \textit{almost constant vectors}, defined as follows:
\[
\mbox{Cons}(\delta, \rho)=\left\{x \in \mathbb{S}^{2n-1}: \exists \lambda \in \C \ \text{such that} \ \left|\left\{i \leq n: |x_{i}-\lambda| \leq \frac{\rho}{\sqrt{n}} \right\} \right| > (1-\delta)n \right\},
\]
where $\delta, \rho \in (0,1)$, and the second part is its complement. The prototype of this structural dichotomy dates back to Cook's work  \cite{Cooksingularitydreg} and was also used and refined in various papers (see, e.g., \cite{Cookcircdigraph,litvak2017adjacency,litvak2019smallest,litvak2019structure,tran2020smallest}). 

The set $\mbox{Cons}(\delta, \rho)$, roughly speaking, consists of vectors obtained by shifting very sparse vectors by constant vectors. Given its low complexity, it is tempting to invoke a standard net argument to establish that the images of such vectors are bounded away from zero. However, due to the sparsity of the random matrices, the net argument deployed in \cite{tran2020smallest} or \cite{jain2020sharp} is not directly applicable in the present context. 

To elucidate the underlying obstruction, we first note that since our random matrices have neither i.i.d. entries nor $d$-regularity, $M_n$ and $M_n^{T}$ are not equally distributed. We therefore have to show that for all $x \in \mbox{Cons}(\delta, \rho)$, $\|\bar{x}M_n\|_2$ and $\|M_n x\|_2$ are bounded away from zero separately. The former estimation allows us to reduce the invertibility problem for non-almost constant vectors to a lower bound on the distance between a row vector to a random hyperplane spanned by the remaining $n-1$ rows; while the latter estimate reveals the structure of a unit normal vector to this hyperplane. Different difficulties arise in establishing these bounds.

To bound $\|M_nx\|_2$, one may use the independence between rows; however, in the sparse regime, the individual probability for a fixed vector $x \in \mathbb{S}^{2n-1}$, as computed using the method for the dense regime (see, e.g., \cite{tran2020smallest, jain2020sharp}), is too large to beat the metric entropy of $\mbox{Cons}(\delta, \rho)$ when employing the $\ep$-net argument. If one tries to further reduce the complexity of almost constant vectors by shrinking $\delta$ and $\rho$, it consequently impairs the arithmetic structure of a unit normal vector, rendering the small ball estimation ineffective.
Meanwhile, bounding $\|\bar{x} M_n\|_2$ is considerably more difficult because the columns of $M_n$ (or the rows of $M_n^T$) are not independent. In \cite{jain2020sharp,tran2020smallest}, the authors addressed the left-multiplication problem for the dense case $d=n/2$ via a conditioning argument: conditioning on the first $k-1$ columns of $M$ with $1 \leq k \leq n/4$,  the entries in the $k$-th column are independent Bernoulli variables with success probability $p \in [1/3,2/3]$. Hence, the Paley-Zygmund inequality together with a tensorization argument yields a good individual probability. However, in the sparse regime, the sparsity leads to a lack of sufficient randomness, preventing the use of the conditioning argument to establish the desired individual probability necessary for the $\varepsilon$-net argument. 

To overcome the difficulty due to the sparsity, we split vectors into four (overlapping) classes on $\C^n$: steep vectors (possessing a significant jump), gradual vectors (that are not steep), almost constant vectors, and non-almost constant vectors. A similar splitting also appears in \cite{litvak2019smallest,litvak2019structure,litvak2022singularity}.
 To be more precise, given a vector $x=(x_{i})_{1\leq i \leq n} \in \C^n$, denote by $(x_{i}^{\ast})_{1 \leq i \leq n}$ the non-increasing rearrangement of the sequence $(|x_{i}|)_{1\leq i \leq n}$. The vector $x$ is said to be steep, if $x_k^{\ast} \gg x_m^{\ast}$ for some $k \ll m$.

To deal with steep vectors, we first explore the expansion properties of our random matrices. Unlike the purely combinatorial approach used in \cite{litvak2017adjacency} for the 
$d$-regular case, which relies heavily on the deterministic structure of regular graphs, we introduce a probabilistic framework based on the negative association of suitably constructed random variables. This negative association method is new to this context and appears essential here, as we do not see how the simple switching argument from \cite{litvak2017adjacency} would directly extend to our row-constrained model with free columns.


The expansion properties guarantee that there are relatively many rows that have exactly one entry  equal to 1 within the set of columns corresponding to large coordinates of a steep vector $x$. 
Then the inner product of each such row with the steep vector $x$ is dominated by the absolute  value of the unique large coordinate due to the large jump and sparsity of the matrices. Therefore, to fully take advantage of the expansion properties, we need to choose the jump size and its location carefully. In addition, the selection of the size of the jump and its location also directly impacts the construction of the associated nets. Furthermore, the effective deployment of the net argument depends critically on a sharp estimate of the operator norm $\|M_n-\mathbb{E}M_n\|$ (which is equal to the operator norm of $M_n$ restricted to the subspace of vectors having zero sum of coordinates). Although this estimate may be of interest in its own right, it is not a mere technical refinement but an indispensable prerequisite for the entire argument. In particular, its optimality is essential,  for it is precisely this sharp bound that dictates the optimal order $\log n$ obtained in our main theorem.

It should be emphasized, however, that such optimality for the operator norm is far from immediate—indeed, it is considerably less transparent than in the classical i.i.d. setting in \cite[Lemma 3.6]{litvak2022singularity}, where Talagrand's inequality is directly applicable.  By contrast, in the present dependent framework, standard matrix concentration tools also seem limited: a naive application of the matrix Bernstein inequality (see, e.g., \cite[Corollary 6.12]{tropp2015introduction}) yields only a suboptimal bound of order $\sqrt{d} \log n$, which is insufficient for our purpose. The required sharp estimate is instead obtained through an adaptation of the Kahn-Szemer\'{e}di argument. The argument was first introduced by Kahn and Szemer\'{e}di in \cite{friedman1989second} to bound the second eigenvalue of a random graph from the permutation model, and has since been adapted to establish bounds on the spectral gap for several other random graph models (see, e.g., \cite{CookGJsizebiased,zhu2023second} and the references therein).

Bounding the magnitude of the matrix-vector product for almost constant gradual vectors is straightforward. Loosely speaking, without big jumps, any nonzero almost constant gradual vector $x$ is essentially very close to a constant nonzero vector $\lambda_{0} \mathbf{1}$, so the $\ell_2$-norms are also comparable. 
Moreover, by the concentration, each column contains approximately $d$ ones, therefore the inner product of each row or column with an almost constant gradual vector is roughly $d\lambda_0$. Hence, we can  deduce that both $\|\bar{x}M_n\|_2$ and $\|M_nx\|_2$ are bounded below.

After  bounds for the above two classes are obtained, it remains to deal with non-almost constant vectors. The infimum over the non-almost constant vectors can be tackled by relating it to the average distance of a row of the matrix $M_n$ to the subspace spanned by the rest of the rows (similar reductions were used in various papers, see, e.g.,   \cite{rudelson2008littlewood,litvak2019smallest,TVuniversality}). 
Note that the distance is further lower bounded by the absolute value of the inner product between a row and a unit normal vector to the subspace. Moreover, since the rows of our random matrices are independent, via conditioning, we may assume the row and the unit normal vector are independent. Now the bounds for the first two classes imply that, with  high probability, the normal vector belongs to the class of non-almost constant vectors. This enables us to apply an anti-concentration inequality in \cite[Lemma 8.5]{sah2023limiting} to complete the proof of Theorem \ref{thm: lsvlowertails}.

We finally discuss the organization of the paper. In the next section we introduce notation and some preliminary results. 
In \cref{sec:clrcb}, we will prove Corollary \ref{thm:CLRCB} using Theorem \ref{thm: lsvlowertails} and the \textit{replacement principle} first introduced by Tao and Vu in \cite{TVuniversality}.
Then the rest of the paper will be mostly devoted to the proof of Theorem~\ref{thm: lsvlowertails}. In \cref{sec:res_nrm}, we obtain a sharp upper bound on 
 the operator norm $\|M_n-\mathbb{E}M_n\|$ using the Kahn-Szemer\'{e}di argument. In \cref{sec:concen_expan}, we focus on the expansion properties of our 
random matrices, which also crucially determine the definition of the steep vectors. In 
\cref{sec: inverb_almost_const}, we show that for any almost constant non-zero vector $x$ and any complex number 
$z$ with $|z| \leq \sqrt{d} \log \log d$, with high probability, both $\|\bar{x}(M_n-z\I_n)\|_2/\|\bar{x}\|_2$
and $\|(M_n-z\I_n)x\|_2/\|x\|_2$ are bounded from below by a constant depending on $n$ and $d$. The argument is technically 
involved since we are required to distinguish several types of jumps as well as different jump locations and combine these with a 
very delicate construction of the $\ep$-nets. More details will be discussed there. We prove Theorem \ref{thm: lsvlowertails} in 
\cref{sec:proof_lsv}. The key idea is to show that with large probability the distance of a row of $M_n$ from the subspace spanned by the rest of the rows cannot be too small. This requires combining the results in \cref{sec: inverb_almost_const} with an application of anti-concentration inequality in \cref{lem: weak_smallball_notalmost}. Finally, in the last section, we briefly 
discuss the singularity threshold.

\section{Notations and preliminaries}

In the rest of the paper,
we simply write our random matrix $M_n$ as $M$  omitting the index $n$.

\subsection{Notations}
Given a natural number $n$,  
we denote $[n]:=\{1, 2,\dots, n\}$. For $x\in \C^n$, let $\sigma_x: [n]\to [n]$ be a permutation such that $|x_{\sigma_x(1)}|\ge |x_{\sigma_x(2)}|\ge \cdots \ge |x_{\sigma_x(n)}|$.  Let $x^*\in \C^n$ be the non-increasing rearrangement of $x$ defined in the following way: $x_i^*=|x_{\sigma_x(i)}|$ for $i\in [n]$. 

For a vector $x \in \C^n$ (or $x \in \R^n$), its canonical Euclidean norm is denoted by $\|x\|_2$. 
With a slight abuse of notation we denote the unit sphere in $\C^n$ by  
$$\S^{2n-1}:=\{z\in \C^n: \|z\|_2=1\}$$
and the unit sphere in $\R^n$ by  
$$\S^{n-1}:=\{x \in \R^n: \|x\|_2=1\}.$$ 
We also denote 
$$
\S_{0}^{2n-1}:=\Big\{v\in \S^{2n-1}: \sum_{i=1}^n v_i=0\Big\} \quad \quad \mbox{ and } \quad \quad 
\S_{0}^{n-1}:=\Big\{v \in \S^{n-1}: \sum_{i=1}^n v_i=0\Big\}.
$$

Given a matrix $A$, we denote its $i$-th row by $R_i=R_i(A)$.
 and $i$-th column by $\Tc_i(A)$. 
For $n\times n$ zero-one matrices we denote
\[
\mbox{Supp }R_i=\mbox{Supp}(R_i(A)) := \set{j \in  [n] : A_{ij} = 1}.
\]

Denote $\mathbf{1}=(1,1,...,1)\in\C^n$ and let $\I_n$ denote the $n\times n$ identity matrix. Given a matrix $A$, $A^{T}$ is the transpose of $A$. Given a (column) vector $x \in \C^n$, the transpose of $x$ is denoted by $x^T$.

Given an $m \times n$ matrix $A$, the operator norm of $A$ acting from $\ell_{2}^{n}$ to $\ell_{2}^{m}$ is denoted by $\|A\|$, that is, 
\[
\|  A \| = \sup_{\|x\|_2 =1}\| Ax \|_2.
\]
Given two real numbers $a$ and $b$ we use the standard notation $a\vee b=\max\{a, b\}$. 
For a complex number $z\in \C$, we denote by $\Re (z)$ its real part.

\subsection{Concentration}

Below we will use the following particular case of the Bennett inequality 
(see, e.g., \cite[Lemma~3.4]{litvak2022singularity}).

\begin{lemma} \label{Benet} Let $p\in (0, 1/2]$,  $m\geq 1$. Let $\de_1, ..., \de_m$ 
be independent Bernoulli($p$) random variables. Then for $\tau> e$,
$$
   \P\left( \sum_{i=1}^{m} \de_{i}> (1+\tau) pm \right)\le 
   \exp(-\tau \log(\tau/e) pm),
$$   and for $0<\tau <1$,
$$
   \P\left( \sum_{i=1}^{m} \de_{i}> (1+\tau) pm \right)\le 
   \exp\left(-\frac{mp\tau^2}{2(1-p)}(1-\tau/3)\right) \leq \exp\left(-\frac{mp\tau^2}{3(1-p)}\right).
$$
\end{lemma}

The following immediate consequence of Bennett's lemma controls the number of ones in each column, 
which will be used later in several places, in particular in Theorem~\ref{thm:CLRCB} and  
in Lemmas~\ref{lem: Net} and \ref{lem:t1}.
 
  \begin{lemma} \label{lem: colsum} \label{lem: rowsum}
  Let $n\geq 2$, $m \geq 1$, $\beta\geq 4$, and $1\leq d\le n/2$.  
      Let $M$ be an $m \times n$ random matrix with entries in $\{0, 1\}$, where each row is independently and uniformly sampled from the set of all vectors in $\{0, 1\}^n$ containing exactly $d$ ones. Given $\tau>0$, consider the event
    \[
    \cE_{\ref{lem: colsum}}=\cE_{\ref{lem: colsum}}(\tau):=\set{M=\{M_{ij}\}_{i \in [m],j \in [n]}: \sum_{i=1}^{m} M_{ij}\le \frac{(1+\tau) md}{n} \ \mbox{ for every } j\le n}.
    \]
    Then for $\tau=\max{\big\{e^2, \frac{\beta n\log m}{md}\big\}}$,
    \[
    \P( \cE_{\ref{lem: colsum}})\ge 1-n\cdot m^{-\be},
    \]   
and for $0 < \tau < 1$,
    \[
    \P( \cE_{\ref{lem: colsum}})\ge 1-n\exp\left(-\frac{md\tau^2}{3n}\right).
    \]
\end{lemma}

\begin{proof}
    Fix $j\in [n]$. Note that $M_{1j},\dots, M_{mj}$ are independent Bernoulli($d/n$) random variables. By Lemma~\ref{Benet}, for $\tau=\max{\big\{e^2, \frac{\beta n\log m}{md}\big\}}$ , 
    \[
    \P\lr{\sum_{i=1}^{m} M_{ij}>\frac{(1+\tau)md}{n}}\le \exp(-\tau \log (\tau/e)(dm/n))
    \le m^{-\be},
    \]
 and for $0 < \tau < 1$,
    \[
    \P\left( \sum_{i=1}^{m} M_{ij}>\frac{(1+\tau)md}{n} \right)\leq \exp\left(-\frac{md\tau^2}{3n(1-d/n)}\right) \leq  \exp\left(-\frac{md\tau^2}{3n}\right).
    \]
 Then the union bound over $j\leq n$ completes the proof . 
\end{proof}

\subsection{Negative association} \label{ssna}

In Sections~\ref{sec:res_nrm} and \ref{sec:concen_expan} we will use the notion of negatively associated (NA) 
random variables, which allows us to work with lightly dependent random variables (such as entries of $M$) 
as if they were independent.

\begin{defn}(Negative association)\label{def:na}
Random variables $X_i, i \in [n]$, are said to be negatively associated (NA) if for all disjoint subsets $I, J \subset [n]$ and all coordinate-wise non-decreasing functions $f: \R^I\to \R$ and $g: \R^J\to \R$ one has 
\begin{equation}\label{NA}
\mathbb{E} f((X_i)_{i \in I})\, g((X_j)_{j \in J}) \leq 
\mathbb{E} f((X_i)_{i \in I})\, \mathbb{E} g((X_j)_{j \in J})  .
\end{equation}
\end{defn}

\begin{remark}\label{rem: non-inc}
 Not that equivalently (passing to functions $-f$ and $-g$) we can require that $f$ and $g$  
 are coordinate-wise non-increasing in the definition.  (cf, \cite{joag1983negative}).
\end{remark}

The following two properties of negative association are very useful in the applications (see, e.g., \cite[Properties~4 and 6]{joag1983negative}).
\begin{enumerate}
    \item \textbf{Closure under products.} If $X_1, \dots, X_n$ and $Y_1, \dots, Y_m$ are two independent  families of negatively associated random variables, then the family \[X_1, \dots, X_n, \, Y_1, \dots, Y_m\] is also negatively associated.
    \item \textbf{Disjoint monotone aggregation.} If $X_i, i \in [n]$, are negatively associated random variables, $\mathcal{A}$ is a family of disjoint subsets of $[n]$, and for each $A\in \mathcal{A}$, $f_{A}: \R^A\to \R$ is a coordinate-wise non-decreasing function, then the random variables 
    \[
    f_{A}((X_i)_{i \in A}), \quad A \in \mathcal{A},
    \]
    are also negatively associated.
    \end{enumerate}

\begin{lemma}\cite[Theorem 2.6]{joag1983negative}\label{lem:nega_joa}  
    Let $X_1,\dots, X_n$ be a sequence of independent random variables. Suppose that the conditional expectation $\E\lr{f((X_i)_{i \in A})\, |\, \sum_{i\in A}X_i}$ is increasing in $\sum_{i\in A}X_i$ for every proper subset $A$ of $[n]$ and  every coordinate-wise increasing function $f: \{0, 1\}^{|A|}\to \R$. Then the conditional distributions of $X_1,\dots, X_n$ given $\sum_{i\in [n]}X_i$ is NA almost surely.
\end{lemma}

\begin{lemma} \label{lemma: Negas}  Let $n, m\geq 1$ and 
  let $M$ be an $m \times n$ random matrix with entries in $\{0, 1\}$, where each row is independently and uniformly sampled from the set of all vectors in $\{0, 1\}^n$ containing exactly $d$ ones. Then the entries $(M_{ij})_{ 1 \leq i \leq m,  1 \leq j \leq n}$ are negatively associated. 
\end{lemma}

\begin{proof}
    Let $\delta_{1}, \cdots, \delta_{n}$ be i.i.d. Bernoulli($p$) random variables with $p=d/n$, 
    note that for each $i\in[m]$, $(M_{i1}, \dots, M_{in})$ has the same distribution as the 
    conditional distribution of $(\delta_{1}, \cdots, \delta_{n})$ given $\sum_{i=1}^{n} 
    \de_{i}=d$. Therefore, by \cref{lem:nega_joa}, it suffices to show that the conditional 
    expectation 
    $$
     \mathbb{E}\set{f((\de_i)_{i \in A})\,\, |\,\,  \sum_{i\in A} \de_{i} }
     $$ 
     is increasing in 
    $\sum_{i\in A} \de_{i}$, for every $\ell\leq n$, every proper subset 
    $A$ of $[n]$ with $|A|=\ell$ and  every coordinate-wise increasing function 
    $f: \{0, 1\}^{|A|}\to \R$.  
    
     Fix $\ell$, $A\sub [n]$ with $|A|=\ell$ and a coordinate-wise increasing function 
    $f: \{0, 1\}^{|A|}\to \R$. Consider the function $g$ defined by 
    $$
     g(t):=\mathbb{E}\set{f((\delta_i)_{i\in A})\,\, |\,\,  \sum_{i\in A} \de_{i} =t}
    $$ 
    for every $0\le t\le \min(\ell, d)$. We first show that $g$ is increasing.
Indeed, given a set $B\sub A$, let $f(B)$ denote $f((w_i)_{i\in A})$, where 
$w_i=1$ if $i\in B$ and $w_i=0$ otherwise. 
Then 
\begin{align*}
    g(t)=\frac{1}{\binom{\ell}{t}}\sum_{S’\subset A, |S’|=t}  f(S')
\end{align*}
and
\begin{align*}
    g(t+1)&=\frac{1}{\binom{\ell}{t+1}}\sum_{S\subset A, |S|=t+1} f(S)
    = \frac{1}{\binom{\ell}{t+1}} \frac{1}{t+1} \sum_{\substack{S\subset A\\|S|=t+1}} \sum_{\substack{S' \subset S\\ |S'|=t}} f(S)\\
    &= \frac{1}{\binom{\ell}{t+1}}\frac{1}{t+1} \sum_{\substack{S'\subset A\\|S'|=t}} \sum_{\substack{S \supset S'\\ |S|=t+1}}f(S)
    \ge \frac{1}{\binom{\ell}{t+1}}\frac{1}{t+1} \sum_{\substack{S'\subset A\\|S'|=t}}
    (\ell-t)f(S')=g(t). 
\end{align*}

Thus, for every fixed row with index $i\in [m]$, the family $\cC_i:=\{M_{ij}: j\in [n]\}$ is NA. Meanwhile, since the rows of $M$ are independent, the families $\{\cC_i\}_{i\in [m]}$ are independent NA families. Therefore, by the first property (``closure under products") of NA, 
$$
 \bigcup_{i\in [m]}\cC_i=\set{M_{ij}: i \in [m], j \in [n]}
 $$ 
 is an NA family.
\end{proof}

Let $Q$ be an $m \times n$ deterministic matrix, and let $M$ be as in \cref{lemma: Negas}. Denote  
\begin{equation} \label{fsubq}
 f_{Q}(M):=Q\circ  M=\sum_{ i \in [m],  j \in [n]} Q_{ij}M_{ij},
\end{equation}
$$
  \mu:=\mathbb{E} f_{Q}(M)=\frac{d}{n} \sum_{i \in [m],  j \in [n]} Q_{ij}, \quad\quad
  \mbox{ and }
 \quad \quad \sigma^{2}:=\frac{d}{n} \sum_{i \in [m],  j \in [n]} Q_{ij}^2.
$$
 We show Bennett's inequality for  $f_{Q}(M)$ using negative association.
\begin{thm} \label{thm:Bennettinequality} Let $n, m\geq 1$.
Let $M$ be as in \cref{lemma: Negas}.
 Let $Q$ be an $m \times n$ deterministic matrix with all entries in $[0,K]$.
 Then for all $t \geq 0$,
     \[
     \max\{ \mathbb{P}(f_Q(M)-\mu \geq t), \, \mathbb{P}(f_Q(M)-\mu \leq -t)\}
     \leq \exp \left(-\frac{\sigma^2}{K^2} h \left( \frac{Kt}{\sigma^2} \right)\right) 
     \leq  \exp \left(-\frac{3t^2}{6\sigma^2 + 2Kt} \right),
     \]
     where $h(u)=(1+u) \log (1+u) -u $ for $ u \geq 0$.
\end{thm}

\begin{proof}
    The second inequality follows from the elementary inequality  $h(u)\geq \frac{u^2}{2(1+u/3)}$ for $ u \geq 0$. We proceed to proving the first inequality. By homogeneity, we may assume that $K=1$. Since $Q_{ij}$ are nonnegative, by disjoint monotone aggregation, $(Q_{ij}M_{ij})_{i \in [m],j \in [n]}$ is again an NA family. We begin by proving a bound for the upper tail. For $\lambda \geq 0$, by Chebyshev's inequality, one has 
    \[
    \mathbb{P}(f_Q(M)-\mu \geq t) \leq e^{-\lambda(t+\mu)} \mathbb{E}e^{\lambda f_Q(M)} \leq 
    e^{-\lambda(t+\mu)} \prod_{ i \in [m],j \in [n]}  \mathbb{E}e^{\lambda Q_{ij} M_{ij}},
    \]
    where for the second inequality, we used that $(Q_{ij}M_{ij})_{i \in [m],j \in [n]}$ is an NA family.
    
    Consider the function $\phi(u):=e^u-u-1$ on $\R$. By taking a continuous extension at $u=0$, one can readily verify that $\phi(u)/u^2$ is increasing on $[0,\infty)$.
    Applying this to $u=\lam  Q_{ij}M_{ij}\leq \lam$, 
    we observe 
    \[ \frac{\phi(\lambda Q_{ij}M_{ij})}{ (Q_{ij}M_{ij})^2 } \leq \phi(\lambda).
    \]
    This implies that 
\begin{align*}
   \mathbb{E}e^{\lambda Q_{ij} M_{ij}} &= \mathbb{E}\left(1+\lambda Q_{ij} M_{ij}
   + \phi(\lambda Q_{ij}M_{ij})\right)
   \leq 1+\lambda \mathbb{E} (Q_{ij} M_{ij})+ \phi(\lambda) \mathbb{E} (Q_{ij} M_{ij})^2
   \\ &=
   1+\frac{\lambda d Q_{ij}}{n} + \frac{\phi(\lambda) d Q_{ij}^2}{n} \leq \exp \left(\frac{\lambda d Q_{ij}}{n} + \frac{\phi(\lambda) d Q_{ij}^2}{n} \right).
 \end{align*}
     Thus,
    \[
     e^{-\lambda(t+\mu)} \prod_{i \in [m],j \in [n]}  \mathbb{E}e^{\lambda Q_{ij} M_{ij}} \leq e^{-\lambda t + \sigma^2 \phi(\lambda)}.
    \]
    Optimizing over $\lambda \geq 0$, we obtain the desired bound on $\mathbb{P}(f_Q(M)-\mu \geq t)$.

    Next we show the lower bound. For $\lambda \leq 0$, by Chebyshev's inequality, one has 
    \[
    \mathbb{P}(f_Q(M)-\mu \leq -t) \leq e^{\lambda(t-\mu)} \mathbb{E}e^{\lambda f_Q(M)} .
    \]
Since $\lambda\le 0$, the map $u\mapsto e^{\lambda u}$ is coordinate-wise non-increasing on $[0,\infty)$. 
By Remark~\ref{rem: non-inc} we can define NA family using coordinate-wise non-increasing functions. Thus,
using this, the fact that $(Q_{ij}M_{ij})_{i \in [m],j \in [n]}$ is an NA family, and the disjoint monotone 
aggregation property of an NA family, we obtain 
  \[
    \mathbb{P}(f_Q(M)-\mu \leq -t) \leq 
    e^{\lambda(t-\mu)} \prod_{i \in [m],j \in [n]}  \mathbb{E}e^{\lambda Q_{ij} M_{ij}}. 
    \]
    Note that for $u \leq 0$ one has $e^u \leq 1+ u + u^2/2$ and $e^{-u}\geq 1-u+u^2/2$, 
    therefore, for $\lambda \leq 0$,
    \[
    e^{\lambda Q_{ij} M_{ij}} \leq 1 + \lambda Q_{ij} M_{ij} + \frac{\lambda^2  (Q_{ij} M_{ij})^2}{2} \leq  1 + \lambda Q_{ij} M_{ij} + \phi(-\lambda) (Q_{ij} M_{ij})^2.
    \]
    This implies 
    \[
     \mathbb{E}e^{\lambda Q_{ij} M_{ij}} \leq \exp \left(\frac{\lambda d Q_{ij}}{n} + \frac{\phi(-\lambda) d Q_{ij}^2}{n} \right).
    \]
    Thus,
    \[
    \mathbb{P}(f_Q(M)-\mu \leq -t) \leq e^{\lambda t + \sigma^2 \phi(-\lambda)}.
    \]
    Optimizing over $\lambda \leq 0$, we complete the proof.
\end{proof}

\section{Proof of Corollary \ref{thm:CLRCB}}
\label{sec:clrcb}

The main idea to prove Corollary \ref{thm:CLRCB} is to use the replacement principle, which is stated in Theorem~\ref{thm: replaceprinciple} below. It enables us to compare our random combinatorial matrix with a random matrix having i.i.d. Bernoulli($d/n$) entries. We then invoke the circular law for random matrices with i.i.d. Bernoulli($d/n$) entries (see, e.g., Theorem 1.4 in \cite{AJMcirclawrev}) to conclude the proof. Our exposition is similar to \cite{TVuniversality, NVcircgivensum}. The replacement principle requires two conditions to be checked. The first one is relatively straightforward in our model. To verify the second one, we need a bound on the smallest singular value similar to the one  given in Theorem~\ref{thm: lsvlowertails}, but for random matrices with i.i.d. Bernoulli($d/n$) entries.  Such an estimate was first obtained by Basak and Rudelson \cite[Theorem~2.2]{BRudcirclawsparse} in a much more general setting (when each entry of a matrix was the product of a Bernoulli random variable with a sub-Gaussian random variable).
In fact, following similar lines to those of the proof of Theorem~\ref{thm: lsvlowertails}, we can 
obtain the following.

\begin{thm}\label{thm: Berlowertails}
 There exist absolute positive constants $C$, $C_1$, $C_2\geq 1$
such that the following holds.   Let $n$ be a large enough integer and 
$C\log n \leq d \leq n/2$.  Let 
     $\Bn$ be an $n \times n$ random matrix having i.i.d. Bernoulli$(d/n)$ entries. Let $z\in \C$ 
     satisfy  $|z| \leq \sqrt{d}\, \log \log d$. Then 
    \[
    \mathbb{P}(s_n(\Bn-z\I_{n}) \leq n^{-\bb}) 
    \leq \frac{C_1}{\sqrt{d}},
    \]
     for some $\bb \in   (2,  C_2 \log \log n]$.
\end{thm}

\begin{remark}\label{rem:Berlowertails}
 Similarly to Theorem~\ref{thm: lsvlowertails}, the bound $n^{-\bb}$ in Theorem~\ref{thm: Berlowertails} 
 can be substituted with $\alpha(n, d)$ introduced in Remark~\ref{sing-comp-bb}.  
\end{remark}

Recall that a sequence of random variables $(X_n)_{n\geq 1}$ is said to be bounded in probability if we have 
\[
\lim_{K \rightarrow \infty } \liminf_{n \rightarrow \infty} \mathbb{P}(|X_n| \leq K) =1
\]

\begin{thm}[Replacement principle]\cite[Theorem 2.1]{TVuniversality} \label{thm: replaceprinciple}
 For every  $n\geq 1$  let $A_n, B_n$ be two random $n \times n$  matrices. 
 Assume that the following two conditions hold. 
 \begin{enumerate}
     \item The sequence 
     \[\frac{1}{n} \left(\|A_n\|_{HS}^2 + \|B_n\|_{HS}^2\right)
     \]
     is bounded in probability, where $\|\cdot\|_{HS}$ denotes the Hilbert-Schmidt norm of a matrix.
     \item For almost all complex number $z$,
     \[
     \frac{1}{n}\log \left|\det\left(A_n-zI_n\right)\right|-\frac{1}{n}\log \left|\det\left(B_n-zI_n\right)\right|
     \]
     converges in probability to $0$.
 \end{enumerate}
 Then $\mu_{A_n}-\mu_{B_n}$ converges in probability to $0$.
\end{thm}

As we mentioned above, the main difficulty is to verify the second condition. 
We need a few auxiliary results.

\begin{thm} \label{thm:distance} 
Let  $1\leq k<n$ and $p\in (0,1)$ satisfy $p(1-p)(n-k)\geq 1$. 
Let $V \subset \C^n$ be a (fixed) subspace of dimension $k$ 
and  $u\in \C^n$ be a (fixed) vector. 
Let $\delta=(\delta_1,\cdots, \delta_n)$ be a random vector consisting 
of i.i.d. Bernoulli($p$) random variables, and denote the distance from 
$\de + u$ to $V$ by $r$. Furthermore, denote 
$$
   E(p) = \E r \quad \quad \mbox{ and } \quad \quad D(p)= \sqrt{p(1-p)(n-k) + d^2_{u'}},
$$
where $u'=u+p\, \1=u+\E \delta$ and $d_{u'}$ is the distance from $u'$ to $V$. 
Then  
$$
          \tfrac{1}{2}  D(p)\leq     E(p)\leq D(p)
$$
and for any $t>0$, 
\[
 \mathbb{P}_{\delta}\left( \left|r- E(p) \right|\geq t\right) \leq Ce^{-ct^2},
\]
where $C,c$ are positive absolute constants.
\end{thm}

\begin{proof}
 Denote by $P=\{p_{ij}\}_{ij}$ the orthogonal projection matrix from $\C^n$ to the subspace $V^{\perp}$. 
 Let $x=\delta-\mathbb{E}\delta$. Then
 \begin{equation} \label{r-dist}
 r^2=\|P(\delta+u)\|^2_2=\|P(x +\E\delta+u)\|^2_2=\|P(x+u')\|_2^2=\|Px\|_2^2+2 \Re(\langle x,Pu' \rangle) +\|Pu'\|_2^2.
 \end{equation}
 Since random vector $x$ consists of i.i.d. mean-zero coordinates, we have 
 \begin{equation} \label{0-inner}
 \mathbb{E}\, \Re(\langle x,Pu'\rangle)=0
 \end{equation}
 and 
\begin{equation} \label{norm-2} 
  \mathbb{E}\|Px\|_2^2=
 \E \langle P x, x\rangle=\mathbb{E}\sum_{i,j}p_{ij}x_ix_j = 
 p(1-p)\sum_{i=1}^{n}P_{ii}= p(1-p)\mbox{Tr}(P) = p(1-p)(n- k).
 \end{equation} 
 Therefore, 
 \[\mathbb{E}\, r^2=p(1-p)(n-k)+d^2_{u'} = (D(p))^2.
 \]
 Using that $x$ is a $\R^n$-valued random vector consisting of i.i.d. 
 mean zero coordinates again we obtain that 
 \begin{align}\label{inner-b}
  &\mathbb{E} |\langle x, Pu'\rangle|^2=\mathbb{E} \left|\sum_{i=1}^{n} x_i 
  \sum_{j=1}^{n} \overline{p_{ij}u'_{j}} \right|^2= \sum_{i=1}^{n}\mathbb{E} x_i^2 
  \left| \sum_{j=1}^{n} \overline{p_{ij}u'_{j}} \right|^2\\
  &=p(1-p)\sum_{i=1}^{n} \left| \sum_{j=1}^{n} \overline{p_{ij}u'_{j}} \right|^2=p(1-p)\|Pu'\|^2=p(1-p)d_{u'}^2  
 \end{align}
and 
\[
 \mathbb{E} \left| \sum_{i\neq j} p_{ij}x_i x_j \right|^2=\sum_{i\neq j}\sum_{k\neq \ell}p_{ij}\overline{p_{k\ell}}\, \mathbb{E}(x_ix_jx_kx_\ell)=p^2(1-p)^2\left(\sum_{i\neq j}|p_{ij}|^2+\sum_{i\neq j}p_{ij}\overline{p_{ji}}\right).
\]
Hence, 
\[
 \mathbb{E} \left| \sum_{i\neq j} p_{ij}x_i x_j \right|^2\le p^2(1-p)^2\left(\sum_{i\neq j}|p_{ij}|^2+\sum_{i\neq j}|p_{ij}||p_{ji}|\right)
\]
Using $|ab|\le (|a|^2+|b|^2)/2$ for $\forall a, b \in \C$, we observe 
$\sum_{i\neq j}p_{ij}\overline{p_{ji}}\leq \sum_{i\neq j}|p_{ij}||p_{ji}|\le \sum_{i\neq j}|p_{ij}|^2$. Therefore,
\begin{equation} \label{off-diag} 
 \mathbb{E} \left| \sum_{i\neq j} p_{ij}x_i x_j \right|^2\le2p^2(1-p)^2\sum_{i\neq j}|p_{ij}|^2.
\end{equation}
Since $x_i^2-p(1-p)$ are also i.i.d.  mean zero,  
\begin{equation} \label{4thpow} 
\mathbb{E}\left|\sum_{i=1}^{n}p_{ii}[x_i^2-p(1-p)]\right|^2=\sum_{i=1}^{n}|p_{ii}|^2 \mathbb{E}[x_i^2-p(1-p)]^2=p(1-p)(1-2p)^2 \sum_{i=1}^{n}|p_{ii}|^2.
\end{equation} 
Next, we evaluate the variance of $r^2$, which yields a lower bound on $\E r$.
Using (\ref{r-dist}), (\ref{0-inner}), and (\ref{norm-2}),  we obtain 
 \begin{align*}
 \mbox{Var}(r^2)&=\mathbb{E}(r^2-\mathbb{E}(r^2))^2
 = \E \left(\|P(x+u')\|_2^2 - \E \|P(x+u')\|_2^2 \right)^2
 \\&=
 \E \left(\|Px\|_2^2 + 2\Re(\langle x, Pu'\rangle) - \E \|Px\|_2^2 \right)^2
 \\ &= \mathbb{E}\left( \sum_{i=1}^{n}p_{ii}\left(|x_i|^2-p(1-p)\right) 
  +\sum_{i\neq j} p_{ij}x_i x_j + 2\Re(\langle x, Pu'\rangle) \right)^2
 \\ &\leq 3 \mathbb{E}\left| \sum_{i=1}^{n}p_{ii}[|x_i|^2-p(1-p)]\right|^2+3 \mathbb{E} \left|\sum_{i\neq j} p_{ij}x_i x_j \right|^2 + 12 \mathbb{E} \left|\Re(\langle x, Pu'\rangle )\right|^2.
 \end{align*}
 Since for every $z\in \C$, $|\Re(z)|\le |z|$, the last term is bounded by $12 \mathbb{E} \left|\langle x, Pu'\rangle \right|^2$.
 Therefore, using (\ref{inner-b}), (\ref{off-diag}), and (\ref{4thpow}),  and $4p(1-p)\leq 1$, 
 \begin{align*}
 \mbox{Var}(r^2)&\leq 
  3p(1-p)\Big( (1-2p)^2\sum_{i=1}^{n}|p_{ii}|^2 + 2 p(1-p)\sum_{i\neq j} |p_{ij}|^2
  + 4 d^2_{u'} \Big) \\
 &\leq 3 p(1-p)\sum_{1\leq i,j \leq n} |p_{ij}|^2 + 3 d^2_{u'}=  3 \left(p(1-p) \mbox{Tr}(PP^{*}) + d^2_{u'}\right)\\
 &=3 \left(p(1-p) \mbox{Tr}(P)+d^2_{u'}\right)= 3\left( p(1-p)(n-k) + d^2_{u'} \right)
 = 3 (D(p))^2.
\end{align*}
This implies 
$$
  \E r^4 \leq 3 (D(p))^2 + (\E r^2)^2 = (D(p))^2 ((D(p))^2 +3) .
$$
 Using the H{\"o}lder inequality, we observe
 $$
   (D(p))^2 = \E\, r^2 = \E \, r^{2/3} r^{4/3} \leq 
   \left(\E\, r\right)^{2/3}\left(\E r^{4}\right)^{1/3} \leq 
   \left(\E\, r\right)^{2/3} \, (D(p))^{2/3} ((D(p))^2 +3)^{1/3}.
 $$
 Using that under the assumptions of the theorem $(D(p))\geq 1$, we obtain
$$
   \E \, r \geq \frac{(D(p))^2}{\sqrt{(D(p))^2 +3}} 
   \geq  \frac{D(p)}{2}. 
$$
Since $\E \, r\leq \left(\E\, r^2\right)^{1/2} = D(p)$, this shows the first 
inequality of Theorem~\ref{thm:distance}. The second one follows from 
Talagrand's concentration inequality (see, e.g., \cite[Theorem 2.1.13]{tao2012topics}), 
as $r$ is a $1$-Lipschitz convex function of $\de$.
%
\end{proof}

\begin{corollary} \label{Cor:distance1}
Let $1\leq d\leq n/2$ and $1\leq k<n$. Denote $p=d/n$ and assume 
$p(1-p)(n-k)\geq 1$.  Let $V \subset \C^n$ be a (fixed) subspace 
of dimension $k$ and  $u\in \C^n$ be a (fixed) vector. 
Let $\de$ be a random vector drawn uniformly from the set of all vectors in $\{0, 1\}^n$ containing exactly $d$ ones  and denote the distance from $\de + u$ to $V$ by $r$. 
Let $E(p)=E(d/n)$ 
be the parameter from Theorem~~\ref{thm:distance} 
(note that to define it we need to deal with a Bernoulli random vector). 
Then for any $t>0$ we have
\[
 \mathbb{P}_{\delta}\left(\left|r- E(p) \right|\geq t\right) \leq
C\sqrt{d}\, e^{-ct^2},
\]
where $C, c$ are positive absolute constants.
\end{corollary}

\begin{proof}
Let $\delta'=(\delta_1',\cdots, \delta_n')$ be a random vector consisting of i.i.d. Bernoulli($d/n$) random variables,  and denote the distance from 
$\de' + u$ to $V$ by $r'$. Denote also 
 $u'=u+p\, \1=u+\E \delta= u+\E \delta'$ and let $d_{u'}$ be the distance from 
 $u'$ to $V$.    
By Theorem \ref{thm:distance}, we have
\[
\mathbb{P}_{\delta'}\left( \left|r'- E(p) \right|\geq t\right) \leq
Ce^{-ct^2}
\]
for some absolute positive constants $C$ and $c$. 
Using Stirling's formula $m!=\sqrt{2\pi m} (m/e)^m e^{\theta_m /(12 m)}$ for some $\theta_m \in (0,1)$, we observe 
\[
\mathbb{P}_{\delta'}\Big(\sum_{i=1}^{n}\delta_i'=d\Big)=\binom{n}{d}\left(\frac{d}{n}\right)^d\left(1-\frac{d}{n}\right)^{n-d} \geq \frac{1}{2}\, 
\sqrt{\frac{n}{2\pi d(n-d)}} \geq \frac{1}{4}\, \sqrt{\frac{1}{\pi d}}.
\]
Therefore,
\[
    \mathbb{P}_{\delta} \left(\left|r-E(p)\right|\geq t\right)
=\mathbb{P}_{\delta'}\left(\left|r'- E(p)  \right|\geq t\,\, \middle\vert \,\,  \sum_{i=1}^{n}\delta_i'=d\right)
\leq \frac{\mathbb{P}_{\delta'}\left(\left|r'-E(p) \right|\geq t\right)}{\mathbb{P}_{\delta'}(\sum_{i=1}^{n}\delta_i'=d)} \leq 8 C \sqrt{d}\, e^{-ct^2},
\]
which completes the proof. 
\end{proof}

We need two more well-known facts about singular values.

\begin{lemma}[Cauchy interlacing law] \cite [Lemma A.1]{TVuniversality} \label{lemma: CIL} Let $1\leq m\leq n$. 
Let $A$ be an $n\times n$ matrix and $A'$ be the sub-matrix formed by the first $n-m$ rows of $A$. Let $s_1(A) \geq \cdots \geq s_n(A) \geq 0$ and $s_1(A') \geq \cdots \geq s_{n}(A') \geq 0$ be the singular values of $A$ and $A'$, respectively. Then we have 
\[
s_{i+m}(A) \leq   s_i(A') \leq  s_i(A),
\]
 for $1 \leq i \leq n-m$.   
\end{lemma}

\begin{lemma}[Negative second moment]\cite[Lemma A.4]{TVuniversality} \label{lemma: nsm}
Let $1\leq k \leq n$, and let $A'$ be  a $k \times n$ matrix of rank $k$ with singular values $s_1(A') \geq \cdots \geq s_{k}(A') \geq 0$ and rows $R_1,\cdots, R_{k}$. For each $i \leq k$, let $W_i$ be the subspace generated by the $k-1$ rows $R_1, \cdots, R_{i-1}, R_{i+1}, \cdots, R_{k}$. Then we have
\[
\sum_{i=1}^{k} s_{i}^{-2}(A') = \sum_{i=1}^{k} \mbox{\rm dist}^{-2}(R_i, W_i).
\]
\end{lemma}

We are now ready to verify the second condition in 
Theorem~\ref{thm: replaceprinciple}.

\begin{thm}\label{thm:secondcondition}
Let $\ep \in(0, 1)$ be fixed, $\log^{2+\ep} n \leq  d \leq n/2$, and  $p=d/n$. 
Let $M_n$ be uniformly drawn from $\cM_{n,d}$ and $\Bnp$ be a random $n\times n$ 
matrix with i.i.d. Bernoulli($p$) entries that is independent of $M_n$. 
Then  for every fixed $z \in \mathbb{C}$,
\[
\frac{1}{n} \log |\det (M_n-z \sqrt{d(1-p)}I_n)|-\frac{1}{n} \log |\det (\Bnp-z \sqrt{d(1-p)}I_n)|
\]
converges to zero in probability as $n\to \infty$.
\end{thm}

\begin{proof}
Let $u_1, \cdots, u_n$ denote the rows of $-z \sqrt{d(1-p)}I_n$, and let 
$R_1, \cdots, R_n$ and $R_1', \cdots, R_n'$ denote the rows of $M_n$ and 
$\Bnp$, respectively. Set $V_0=\{0\}$, and for $i \geq 1$, let $V_{i}$ 
be the subspace spanned by $R_1+u_1, \cdots, R_{i}+u_{i}$.  Denote the 
distance from $R_i+u_i$ to $V_{i-1}$ by $d_i$. Similarly, define $V_{i}'$ and ${d_i'}$ 
using the vectors $R_j'+u_j$, $j\leq n$. Recall that the absolute value of the 
determinant of a matrix represents the volume of the $n$-dimensional parallelepiped 
generated by the row vectors. Therefore 
\[
\log |\det (M_n-z \sqrt{d(1-p)}I_n)|=\sum_{i=1}^{n} \log d_i \quad \mbox{ and } \quad  \log |\det (\Bnp-z \sqrt{d(1-p)}I_n)
=\sum_{i=1}^{n} \log d_i'.
\]
Let $m=\lceil n-d^{-1/2+\ep/6}n\rceil$. Then 
\[
\frac{1}{n} \log |\det (M_n-z \sqrt{d(1-p)}I_n)|-\frac{1}{n} \log |\det (\Bnp-z \sqrt{d(1-p)}I_n)|
\]
\[
=\underbrace{\frac{1}{n} \Big(\log \prod_{i=1}^{m} d_i-\log  \prod_{i=1}^{m} d_i'\Big)}_{T_1} + \underbrace{\frac{1}{n} \sum_{i=m+1}^{n} (\log d_i-\log d_i')}_{T_2}.
\]
 
   We first evaluate  $T_1$. Fix $1\leq i\leq m$.
Denote by $\mbox{Vol}_{i}$ the volume of the $m$-dimensional parallelepiped 
generated by $R_{1}+u_1, \cdots, R_{i}+u_i, R'_{i+1}+u_{i+1}, \cdots, R'_{m}+u_m$ 
and by $\mbox{Vol}_{0}$ the one generated by $R_{1}'+u_1, \cdots,  R'_{m}+u_m$. 
  Let 
  $$
   \mathcal{F}_i=\{R_{1}+u_1, \cdots, R_{i-1}+u_{i-1}, R'_{i+1}+u_{i+1}, \cdots, R'_{m}+u_m\}
   $$
   and $U_i$ be the subspace generated by $\mathcal{F}_i.$ 
Note that parallelepipeds corresponding to $\mbox{Vol}_{i}$ and $\mbox{Vol}_{i-1}$ 
are generated by $\mathcal{F}_i \cup \{R_{i}+u_i\}$ and $\mathcal{F}_i \cup \{R_{i}'+u_i\}$ respectively. Therefore, 
 %
\[ 
\log \mbox{Vol}_{i} - \log \mbox{Vol}_{i-1}= \log r_i - \log r_i',
\]
where $r_i$ and $r_i'$ denote the distance from $R_{i}+u_i$ and  $R_{i}'+u_i$ to $U_i$  respectively.

\smallskip

For each $i\leq m$ we apply Theorem~\ref{thm:distance} and Corollary~\ref{Cor:distance1} with $t=d^{1/4}$ and $V=U_i$ to vectors $R_i'$ and $R_i$ respectively (note that conditioning on $\cF_i$, the subspace $U_i$ becomes deterministic, and 
$R_i$, $R_i'$ are independent of $\cF_i$).  Note that 
 $$p(1-p)=(d/n)(1-d/n) \geq d/(2n)\quad \quad \mbox{  and } \quad \quad k := \mbox{dim}\, U_i\leq m-1.$$ 
 Thus
$$
 (D(p))^2\geq p(1-p)(n-k)\geq \frac{d}{2n} \, d^{-1/2+\ep/6} n = \tfrac{1}{2}  d^{1/2+\ep/6}. 
$$ 
As we are looking for a limit result and $d\to \infty$ as $n\to \infty$, without 
loss of generality we assume that $d>3^{12/\ep}$. 
Theorem~\ref{thm:distance} and Corollary~\ref{Cor:distance1} imply that 
with probability at least $1-C m\, \sqrt{d}\exp(-c\sqrt{d})$ one has for every $i\leq m$, 
$$
  \frac{r_i}{r_i'}\leq \frac{E(p) + d^{1/4}}{E(p) - d^{1/4}} = 
  1+ \frac{2 d^{1/4}}{E(p) - d^{1/4}} \leq 1+ \frac{2 d^{1/4}}{D(p)/2 - d^{1/4}} 
  =1+ \frac{2}{(1/3)d^{\ep/12} -1},
$$
hence
\[
|\log r_i -\log r_i' | \leq \frac{2}{(1/3)d^{\ep/12} -1}.
\]
  Note that $\mbox{Vol}_{m}=\prod_{i}^{m} d_i$ and $\mbox{Vol}_{0}= \prod_{i}^{m} d_i'$. 
  Therefore
\begin{align*} 
|T_1| &\leq \frac{1}{n} \Big|\log \mbox{Vol}_{m} - \log \mbox{Vol}_{0} \Big|
\leq \frac{1}{n} \sum_{i=1}^{m} |\log \mbox{Vol}_{i} - \log \mbox{Vol}_{i-1}| 
\\ &\leq \frac{1}{n} \sum_{i=1}^{m} |\log r_i - \log r_i'| \leq 
 \frac{2}{(1/3)d^{\ep/12} -1}.
\end{align*}
with probability  at least 
  $$ 1-C n\, \sqrt{d} \exp(-c\sqrt{d})  \longrightarrow 1 \quad \mbox{ as } \quad n\to \infty$$
(we used that $d\geq  \log^{2+\ep} n$). 
 This implies that 
$T_1$ converges to $0$ as $n\to \infty$ in probability. 

\smallskip 

Now we deal with $T_2$. Fix  $z \in \C$ and assume that $|z| \leq  \log \log d$ 
(we can do that as $d\to \infty$ as $n\to \infty$).  
By Theorem~\ref{thm: lsvlowertails} (applied with $z'=z \sqrt{d(1-p)}$ instead of $z$)  
for some $2 <\bb\leq 9$ 
one has $s_n(A)>  n^{-\bb}$ with probability at least 
$1-C_1/\sqrt{d}$, where 
$A=M_n-z \sqrt{d(1-p)}I_n$, $C_1$ and $C_2$ are absolute positive constants. 
Note that the rows of $A$ are $R_i+u_i$, $i\leq n$ and that $A$ is not singular whenever  $s_n(A)>  n^{-\bb}$.
Therefore, we may assume that $V_i$ has full rank for all $i\leq n$. Let $k\leq n$. 
Applying Lemmas~\ref{lemma: CIL} and \ref{lemma: nsm} to matrix $A'$ generated by the first $k$ rows $R_1+u_1,\cdots, R_k+u_k$ we obtain 
\[
d_k^{-2}={\mbox{dist}}^{-2}(R_k+u_k, V_{k-1}) \leq \sum_{i=1}^{k} s_{i}^{-2}(A') \leq
\sum_{i=1}^{k} s_{i+n-k}^{-2}(A) \leq k\, s_{n}^{-2}(A)\leq  n^{1+2\bb}.
\]
On the other hand, clearly $\|R_k\|_2=\sqrt{d}$ and $\|u_k\|_2\leq |z|\sqrt{d}\leq \sqrt{d}\log \log d$, 
therefore 
$$
 d_k={\mbox{dist}}(R_k+u_k, V_{k-1})\leq \|R_k+u_k\|_2\leq 2\sqrt{d}\log \log d.
$$
Thus, for any $k > n-d^{-1/3}n$,
\[
  n^{-1/2-\bb} \leq d_{k} \leq 2 \sqrt{d} \log \log d.
\]

For $M=\Bnp$ we use a similar argument. For the lower bound on $d_k'$, we 
 combine Theorem~\ref{thm: Berlowertails}, \cref{rem:Berlowertails} (applied with $M=\Bnp$ and with 
$z \sqrt{d(1-p)}$ instead of $z$) and Lemmas~\ref{lemma: CIL} and \ref{lemma: nsm}.
For the upper bound, we note that by Lemmas~\ref{Benet} with very 
large probability we have $\|R_k'\|_2 \leq \sqrt{d}\log \log d$ for all $k\le n$. Therefore, 
 with probability at least $1-C'_{1}/\sqrt{d}$, 
\[
   n^{-1/2-\bb} \leq d_{k}' \leq 2 \sqrt{d}\log \log d.
\]

 These bounds on $d_k$ and $d_k'$ together with $d\geq (\log n)^{2+\ep}$, $m\geq  n-d^{-1/2+\ep/6}n$,
  and $\bb<9$ imply 
\[
|T_2| \leq \frac{n-m}{n} \max_{k\leq n} \Big| \log\frac{d_k}{d_k'}\Big| 
\leq  \frac{ \log (2n^{1+\bb} \log \log n) }{d^{1/2-\ep/6}} 
\leq  \frac{ 2(1+\bb)  \log n }{(\log n)^{1+(\ep-\ep^2)/6}} 
\leq  \frac{20}{(\log n)^{(\ep-\ep^2)/6}} 
\]
with probability at least $1-(C_1+C^{\prime}_{1})/\sqrt{d}$. 
Thus, $T_2$ converges to $0$ in probability as $n \rightarrow \infty$.
\end{proof}

\begin{proof}[Proof of Theorem~\ref{thm:CLRCB}] Let 
$p=d/n$, $M_n$ is uniformly drawn from $\cM_{n,d}$, and $\Bnp$ is a random matrix with i.i.d. Bernoulli($p$) entries that is independent of $M_n$. 
Denote 
\[
A_n:=\frac{M_n}{\sqrt{n p(1-p)}} \quad \quad \mbox{ and } \quad \quad 
B_n:=\frac{\Bnp}{\sqrt{n p(1-p)}}.
\]
Then
\[\mathbb{E} \|A_n\|_{HS}^2=\mathbb{E}\|B_n\|_{HS}^2=\frac{n}{1-p} \leq 2n.
\]
Therefore, the first condition of Theorem~\ref{thm: replaceprinciple} is satisfied by Markov's inequality. The second condition for matrices $A_n$, $B_n$ follows by Theorem~\ref{thm:secondcondition}. Therefore, Theorem~\ref{thm: replaceprinciple} 
implies that $\mu_{A_n}-\mu_{B_n}$ converges in probability to $0$. Finally,  applying the circular law for i.i.d. Bernoulli($d/n$)-entry random matrices (see Theorem 1.4 in \cite{AJMcirclawrev}), we complete the proof.
\end{proof}
While Corollary \ref{thm:CLRCB} establishes the circular law under the condition $\min\{d, n-d\} \geq \log^{2+\ep} n $, as is mentioned in Remark~\ref{remark-conj}, we suspect that this condition is not optimal. Indeed, in both the sparse i.i.d. Bernoulli setting \cite{RTcirclawmin} and the $d$-regular case \cite{LLTTYcircdigraph}, the circular law holds whenever $d \rightarrow \infty$ as $n \rightarrow \infty$. We close this section with the following conjecture.
 \begin{con}\label{conjdtoinfty}
    Let $M_n$ be drawn uniformly from $\cM_{n,d}$, and define the normalized matrix
 $$
   \overline{M_{n}}=\frac{M_n}{\sqrt{d(1-d/n)}}. 
 $$ 
 Assume that $d=d(n) \rightarrow \infty$ as $n \rightarrow \infty$, then the empirical spectral distributions  $\left(\mu_{\overline{M_{n}}}\right)_{n\geq 1}$  converges in probability to $\mu_{circ}$ as $n \rightarrow \infty$.
 \end{con}

\section{Restricted operator norm}\label{sec:res_nrm}
The upper bound on the operator norm   $\|M-\mathbb{E}M \|$ plays a crucial role in the construction of nets in \cref{sec: inverb_almost_const}. It is not difficult to see that  $\|M-\mathbb{E}M \|$ equals 
the operator  norm of $M$ restricted to the subspace $\big\{v\in \S^{n-1}: \sum_{i=1}^n v_i=0\big\}$ (see (\ref{restnorm}) below), 
so we call it {\it restricted operator norm}. For $d=n/2$, Tran \cite[Proposition 2.8]{tran2020smallest} employed Bernstein's inequality combined with a net argument to show that this norm  is typically of order $O(\sqrt{n})$ (see also \cite[Lemma 5.1]{jain2021approximate}). The proof also works for $d=pn$, where $p\in (0, 1)$ is a fixed constant. However, their approach does not extend to the sparse regime, particularly when $d$ is of the order of $\log n$. In contrast, we apply the Kahn-Szemer\'{e}di argument to deduce that the restricted operator norm is $O_\lambda(\sqrt{\min\{d,n-d\}+\log n})$ for $1\le d \le n$.
The main result of this section is the following theorem.

\begin{thm} \label{thm: operator_bdd} There are absolute positive constants $C_1$ and $C_2$ 
such that the following holds. 
Let $1\leq m\leq n$, $\gamma\geq 4$, $\beta \geq 1$, and let $d$ be an integer satisfying $1\le d \le n$.  
Let $M$ be an $m \times n$ random matrix with entries in $\{0, 1\}$, where each row is independently and uniformly 
sampled from the set of all vectors in $\{0, 1\}^n$ containing exactly $d$ ones. Let 
$$N= \max\big\{ 9 \min\{d,n-d\}, 2\gamma\log m \big\}.$$
Then, with probability 
    at least $1-n\cdot m^{-\gamma} -2n^{-\beta}$, one has 
    \[
     \|M-\mathbb{E}M\| \leq \frac{2 \min \{d,n-d\}}{\sqrt{N}} + C_1 \beta \sqrt{N}
     \leq C_2 \beta \left( \sqrt{\min\{d,n-d\}+ \gamma \log m} \right).
    \]
In particular, taking $m=n$, $\gamma=4$, $\beta =2$, and $\log n \leq d \leq n- \log n$, we have 
\[
 \|M-\mathbb{E}M\| =\|M^{T}-\mathbb{E}M^{T}\|\leq C_{\ref{thm: operator_bdd}}\sqrt{d}.
\]
with probability at least $1 -3 n^{-1}$. 
\end{thm}

\begin{remark}\label{onlysmalld}
    Denote by $E_{m \times n}$ the matrix with all entries equal to 1, by considering the matrix $A=E_{m \times n}-M$ and noting that $M-\mathbb{E}M=-(A-\mathbb{E}A)$, it is enough to prove Theorem~\ref{thm: operator_bdd} for $d \leq n/2$. 
\end{remark}

\begin{remark}
Note that for $d \leq n/2$ one has 
$$
    \|M-EM\| = \|(M-EM)^{T}\| \geq  \|(M-EM)^{T}e_1\|_2=\sqrt{\frac{d(n-d)}{n}} \geq \sqrt{\frac{d}{2}},
$$
     where $e_1$ is the first element of the standard basis  of $\R^{n}$. 
     Thus, for $m=n$ and $\log n \leq d \leq n/2$, the bound in Theorem~\ref{thm: operator_bdd} 
     is sharp up to a constant.  
\end{remark}

The proof of Theorem~\ref{thm: operator_bdd} is based on properties of negatively associated
random variables discussed in Section~\ref{ssna}  and on the  Kahn-Szemer\'{e}di argument, 
which is based on splitting the  sum under consideration into two sums depending on how large 
the corresponding summands are. In the rest of this section, we will assume that $d \leq n/2$.

\subsection{Kahn-Szemer\'{e}di argument}
\label{ssks}

We follow the exposition in \cite{CookGJsizebiased} (see also \cite{zhu2023second}) with the necessary modifications. 
In view of Remark~\ref{onlysmalld} we assume that $d\leq n/2$, so $\min\{ d, n-d\}=d$.
We also note that  $(M-\E M) \1=0$, hence 
\begin{equation}\label{restnorm}
\| M-\mathbb{E}M \|= \sup_{y \in \S_{0}^{n-1}} \|(M-\mathbb{E}M)y\|_2=\sup_{y \in \S_{0}^{n-1}} \|My\|_2=\sup_{x \in \S^{m-1}, y \in \S_{0}^{n-1}} \langle x,My\rangle.
\end{equation}

 For a fixed $x \in \S^{m-1}, y \in \S_{0}^{n-1}$, consider the matrix $Q=Q(x, y):=xy^{T}$. Then 
\[
f_{Q}(M)=\langle x, My \rangle=\sum_{i \in [m],j \in [n]} M_{ij}x_iy_j,
\]
where the function $f_{Q}$ was defined in (\ref{fsubq}). 
The key point in Kahn-Szemer\'{e}di's argument is to split the sum 
$\sum_{1\leq i,j \leq n} M_{ij}x_iy_j$ into two pieces according 
to absolute values of $x_i y_j$'s. Fix 
\begin{equation} \label{newN}
     N= \max{ \Big\{d, \frac{(1+\tau)md}{n} \Big\}}
\end{equation} 
 for some positive $\tau$ to be determined later
 (in fact, $\tau$ comes from Lemma~\ref{lem: colsum}). 
We define the light and heavy couples of coordinates by
\[
\mathcal{L}(x,y)=\Big\{(i,j) \in [m] \times [n]: |x_iy_j| \leq \frac{\sqrt{N}}{n}\Big\}
\quad \mbox{and } \quad 
\mathcal{H}(x,y)=\Big\{(i,j) \in [m] \times [n] : |x_iy_j| > \frac{\sqrt{N}}{n}\Big\}.
\]
We decompose $f_{Q}(M)$ as 
\[
f_{Q}(M)=\sum_{(i,j) \in \mathcal{L}(x,y)}M_{ij}x_iy_j + \sum_{(i,j) \in \mathcal{H}(x,y)}M_{ij}x_iy_j.
\]
We first estimate the mean of the light part.

\begin{lemma}\label{lem: meanlight}
 Let $M$ be as in Theorem~\ref{thm: operator_bdd} and $N$ as in (\ref{newN}). 
 Then for every $x \in \S^{m-1}, y \in \S_{0}^{n-1}$, we have 
 \[
 \left|\mathbb{E}\sum_{(i,j) \in \mathcal{L}(x,y)}M_{ij}x_iy_j \right| \leq \frac{d}{\sqrt{N}}.
 \]
\end{lemma}
\begin{proof}
    Since $\mathbb{E} M_{ij}=d/n$, we have 
    \[\left|\mathbb{E}\sum_{(i,j) \in \mathcal{L}(x,y)}M_{ij}x_iy_j \right|=\frac{d}{n} \left|\sum_{(i,j) \in \mathcal{L}(x,y)}x_iy_j \right| \leq \frac{d}{n}\sum_{i=1}^{m} \left|\sum_{\{j:(i,j) \in \mathcal{L}(x,y)\}}x_iy_j \right|.
    \]
    For any $y \in \S_{0}^{n-1}$, we have $\sum_{j=1}^n y_j =0$. Hence for every (fixed) 
    $i \in [m]$, 
    \[
    \sum_{\{j:(i,j) \in \mathcal{L}(x,y)\}}x_iy_j =-\sum_{\{j:(i,j) \in \mathcal{H}(x,y)\}}x_iy_j .
    \]
    Therefore, using that $|x_iy_j|>\sqrt{N}/n$ on $\mathcal{H}(x,y)$ and that 
    $x, y$ are unit vectors,
    \[
    \left|\mathbb{E}\sum_{(i,j) \in \mathcal{L}(x,y)}M_{ij}x_iy_j \right| \leq \frac{d}{n}\sum_{i=1}^{m} \left|\sum_{\{j:(i,j) \in \mathcal{H}(x,y)\}}x_iy_j \right| \leq \frac{d}{\sqrt N} \sum_{(i,j) \in \mathcal{H}(x,y)} |x_iy_j|^2 \leq \frac{d}{\sqrt{N}}.
    \]
\end{proof}

\begin{lemma}\label{lem: meanlightsec}
    Let $M$ be as in Theorem~\ref{thm: operator_bdd} and $N$ as in (\ref{newN}). Then for every $\beta>0$ and every 
    $x \in \S^{m-1}, y \in \S_{0}^{n-1}$, 
    \[
    \mathbb{P} \left( \left|\sum_{(i,j) \in \mathcal{L}(x,y)}M_{ij}x_iy_j\right| \geq \frac{d}{\sqrt{N}} + \beta \sqrt{N} \right) \leq 4 \exp \left(-\frac{3\beta^2 n}{4(6+\beta)} \right).
    \]
\end{lemma}
\begin{proof}
    We further split $\mathcal{L}(x,y)$ into ``positive" and ``negative" 
    as $\mathcal{L}(x,y)=\mathcal{L}_{+}(x,y) \cup \mathcal{L}_{-}(x,y)$, where
    \[
    \mathcal{L}_{+}(x,y)=\{(i,j): 0 \leq x_iy_j \leq \sqrt{N}/n\},  \quad \quad \mathcal{L}_{-}(x,y) = \mathcal{L}(x,y) \setminus \mathcal{L}_{+}(x,y).
    \]
     Applying Theorem \ref{thm:Bennettinequality} with $K=\sqrt{N}/n$ and 
\begin{align*}
  Q_{ij}=\begin{cases} 
           x_i y_j,\,\, \mbox{ if } (i,j)  \in \mathcal{L}_{+}(x,y),\\
           0, \,\, \mbox{ otherwise},
         \end{cases} 
\end{align*} 
and using $N \geq d$ together with 
$$
 \sigma^2=\frac{d}{n} \sum_{(i,j) \in \mathcal{L}_{+}(x,y)}
(x_iy_j)^2 \leq d/n,
$$
we obtain
    \[
    \mathbb{P} \left( \left|\sum_{(i,j) \in \mathcal{L}_{+}(x,y)}M_{ij}x_iy_j - \mathbb{E}  \sum_{(i,j) \in \mathcal{L}_{+}(x,y)}M_{ij}x_iy_j \right| \geq \beta \sqrt{N} /2\right) \leq 2 \exp \left(-\frac{3\beta^2 n}{4(6+\beta)} \right),
    \]
    Similarly applying Theorem \ref{thm:Bennettinequality} with 
    \begin{align*}
  Q_{ij}=\begin{cases} 
           -x_i y_j,\,\, \mbox{ if } (i,j)  \in \mathcal{L}_{-}(x,y),\\
           0, \,\, \mbox{ otherwise},
         \end{cases} 
\end{align*} 
    we obtain 
    \[
    \mathbb{P} \left( \left|\sum_{(i,j) \in \mathcal{L}_{-}(x,y)}M_{ij}x_iy_j - \mathbb{E}  \sum_{(i,j) \in \mathcal{L}_{-}(x,y)}M_{ij}x_iy_j \right| \geq \beta \sqrt{N} /2\right) \leq 2 \exp \left(-\frac{3\beta^2 n}{4(6+\beta)} \right).
    \]
     Note that  Lemma~\ref{lem: meanlight} implies 
\begin{align*}
     \mathbb{P} &\left( \left|\sum_{(i,j) \in \mathcal{L}(x,y)}M_{ij}x_iy_j\right| \geq \frac{d}{\sqrt{N}} + \beta \sqrt{N} \right) 
     \\ &\leq  \mathbb{P} \left( \left|\sum_{(i,j) \in \mathcal{L}(x,y)}M_{ij}x_iy_j - \mathbb{E}  \sum_{(i,j) \in \mathcal{L}(x,y)}M_{ij}x_iy_j \right| \geq \beta \sqrt{N} \right).
\end{align*}
    Therefore, 
   the desired result follows by the triangle inequality.
\end{proof}

To bound the heavy part, we use the following version of \textit{discrepancy property} for a matrix $M$, introduced in \cite{CookGJsizebiased}. This concept originally dates back to the work of Kahn and Szemer\'edi (see \cite[Lemma 2.5]{friedman1989second}), where they proved that for a fixed $d$ and sufficiently large $n$, the second eigenvalue of a random matrix drawn uniformly from the permutation model is $O(\sqrt{d})$ with high probability.
For an adjacency matrix of a random graph, it ensures
that the number of edges between any two sets of vertices is comparable to the expected
number.

\begin{defn} \label{def: DP}
Let $n, m\geq 1$ and 
    let $M$ be an $m \times n$ matrix with non-negative entries. 
    Given $S\subset [m], T \subset [n]$, define the edge count as  
    \[
    e_{M}(S,T):=\sum_{i \in S}\sum_{j \in T} M_{ij}.
    \]
    Furthermore, we say that the matrix $M$ has the discrepancy property $DP(\de,\kappa_1, \kappa_2)$ with parameters 
    $\delta \in (0,1)$, $\kappa_1>0$, $\kappa_2 \geq 0$ 
    if for all nonempty $S \subset [m]$ and $T \subset [n]$  either 
    \begin{itemize}
        \item[(1)] $\frac{e_{M}(S,T)}{\delta|S||T|} \leq \kappa_1$, or 
        \item[(2)] $e_{M}(S,T) \log \frac{e_{M}(S,T)}{\delta|S||T|} \leq \kappa_2 (|S| \vee |T|) \log \frac{en}{|S| \vee |T|}$.
    \end{itemize}  
\end{defn}

The following lemma asserts that our random matrices enjoy the discrepancy property with high probability.
\begin{lemma} \label{lem: DPRCM}
Let $1\leq m \leq n$. 
    Let $M$ be the random matrix as in Theorem~\ref{thm: operator_bdd}. Then for any $\beta>0$, with probability at least $1-n^{-\beta}$, the discrepancy property  DP$(\delta, \kappa_1, \kappa_2)$ holds for $M$ with $\delta=d/n$, $\kappa_1=e^2$, $\kappa_2=2(4+\beta)$.
\end{lemma}
\begin{proof} Fix $\beta>0$. 
    Let $Q=\mathbf{1}_S \mathbf{1}_T^{\top}$, where $\mathbf{1}_S \in \{0,1\}^{m}$ and $\mathbf{1}_T \in \{0,1\}^{n}$ are indicator vectors of the set $S$ and $T$, respectively. We have 
    \[
    f_Q(M)=\sum_{i \in [m], j \in [n]} (\mathbf{1}_S)_i(\mathbf{1}_T)_j M_{ij}=e_{M}(S,T).
    \]
    Denote 
    $$
     \mu(S,T):= \mathbb{E}f_Q(M)=\frac{d}{n}|S||T|=\delta |S||T|.
    $$
    Let $\gamma_1=e^2-1$.  Recall that $h(u)=(1+u)\log(1+u) -u$ is strictly increasing 
    and let $\gamma^{\ast}(S,T)$  be the (unique) solution of 
    \[
       h(u)\mu(S,T)=(\beta+4)(|S| \vee |T|) \log \frac{en}{|S| \vee |T|}.
    \]
    Denote  $\gamma(S,T)=\max{\{\gamma^{\ast}(S,T),\gamma_1\}}$. 
    Note that 
    $$
     \sigma^{2}:=\frac{d}{n} \sum_{i,j} Q_{ij}^2=\frac{d}{n}|S||T|=\mu(S,T).
     $$
     Applying Theorem \ref{thm:Bennettinequality} with $K=1$ and $t=\gamma(S,T) \mu(S,T)$, 
     we have
    \[
    \mathbb{P}(e_M(S,T) \geq (1+\gamma(S,T)) \mu(S,T)) \leq \exp(-\mu(S,T) h(\gamma(S,T)))
    \leq \exp(-\mu(S,T) h(\gamma^{\ast}(S,T))).
    \]
    Therefore, 
  \begin{align*}
     &\mathbb{P}(\exists S \subset [m],\,  \exists T \subset [n], \, |S|=k,\,  |T|=\ell:\quad  e_{M}(S,T) \geq (1+\gamma(S,T)) \mu(S,T))\\
    &\leq  \sum_{ S \subset [m], |S|=k} \sum_{ T \subset [n], |T|=\ell} \exp(-\mu(S,T) h(\gamma^{\ast}(S,T)))\\
    &\leq \binom{m}{k}\binom{n}{\ell} \exp \left(-(\beta+4)(k \vee \ell) \log \frac{en}{k \vee \ell}\right)\\
    &\leq \left(\frac{em}{k}\right)^{k} \left(\frac{en}{\ell}\right)^{\ell} \exp \left(-(\beta+4)(k \vee \ell) \log \frac{en}{k \vee \ell}\right)\\
    &\leq \exp \left(-(\beta+2)(k \vee \ell) \log \frac{en}{k \vee \ell}\right) \leq \exp (-(\beta+2) \log (en)),
  \end{align*}
  where  we used  that functions $h(u)$ and  $u \mapsto u \log (e/u)$ are increasing. Now, by the union bound over the $mn$ choices of $k \in [m],$ $\ell\in [n]$, we have that with probability at least $1-n^{-\beta}$,
  \begin{equation} \label{concen}
     \forall S \subset [m]\,\,  T \subset [n] \quad \quad e_{M}(S,T) \leq (1+\gamma(S,T)) \mu(S,T). 
  \end{equation}
  If $S,T$ are such that $\gamma(S,T)=\gamma_1=e^2-1$, then from  \eqref{concen}, we have 
  \[
   e_{M}(S,T) \leq e^2 \mu(S,T)=e^2\delta|S||T|,
  \]
  which implies that  (1) in Definition \ref{def: DP} holds with $\kappa_1=e^2$. 
  Otherwise, we have 
  \begin{equation}\label{condST} 
  e_{M}(S,T) \leq (1+\gamma^{\ast}(S,T)) \mu(S,T).
  \end{equation} 
  Consequently,
  \[
  \frac{h(\gamma^{\ast}(S,T))}{1+\gamma^{\ast}(S,T)} e_{M}(S,T) \leq (\be+4)(|S| \vee |T|) \log \frac{en}{|S| \vee |T|}
  \]
  by the definition of $\gamma^{\ast}(S,T)$. Note that $\gamma^{\ast}(S,T) \geq \gamma_1 =e^2 -1$, 
  hence $\log(1+\gamma^{\ast}(S,T)) \geq 2.$
  Therefore, using (\ref{condST}), we obtain 
  \begin{align*}
      \frac{h(\gamma^{\ast}(S,T))}{1+\gamma^{\ast}(S,T)}& \geq \log(1+\gamma^{\ast}(S,T))-1
      \geq \frac{\log(1+\gamma^{\ast}(S,T))}{2} \geq \frac{1}{2} \log \frac{e_{M}(S,T)}{\mu(S,T)}.
  \end{align*}
  This implies  that for $\gamma^{\ast} \geq \gamma_1$,
  \[
  e_{M}(S,T) \log \frac{e_{M}(S,T)}{\mu(S,T)} \leq 2(\be+4) (|S| \vee |T|) \log \frac{en}{|S| \vee |T|}.
  \]
  As $\mu(S,T)=\delta |S||T|$, the second case in Definition \ref{def: DP} holds with $\kappa_2=2(\be+4)$.
  This completes the proof.
  \end{proof}
  The following deterministic lemma shows that when the discrepancy property holds, the heavy part $\sum_{(i,j) \in \mathcal{H}(x,y)}M_{ij}x_iy_j$ is of order $O(\sqrt{N})$. We omit the proof, since it is very similar to the proof of Lemma 6.6 in \cite{CookGJsizebiased}, see also \cite[Lemma 5.5]{zhu2023second}.

  \begin{lemma} \label{lem: DPtobound} \label{rem: absolute}
  Let $1\leq m \leq n$, $\lam>0$, and $N$ as in (\ref{newN}). 
      Let $M$ be a nonnegative $m \times n$ matrix with all row and column sums bounded by $N$. Suppose that $M$ has  $DP(\delta, \kappa_1, \kappa_2)$ with $\delta=\lam N/n$, $\kappa_1>1$, $\kappa_2 \geq 0$. Then for any $x \in \S^{m-1}, y \in \S_{0}^{n-1}$, we have 
      \[
     \left| \sum_{(i,j) \in \mathcal{H}(x,y)}M_{ij}x_iy_j\right| \leq \alpha \sqrt{N},
      \]
      where $\alpha=16+32\lam (1+\kappa_1)+64 \kappa_2 \left( 1+ \frac{2}{\kappa_1 \log \kappa_1} \right)$.
  \end{lemma}

  \subsection{Completing the proof.}
  \label{sscp}
  
  We will finish the proof of Theorem \ref{thm: operator_bdd} with the standard $\ep$-net argument (as it is classical and used 
  in many works, we omit the proof). 
  
  \begin{lemma} \label{lem: Net}
      For $\ep \in (0,1/2)$. Let $\mathcal{N}_{\ep}$ and $\mathcal{N}^{0}_{\ep}$ be nets for $\S^{m-1}$ and $\S_{0}^{n-1}$, respectively. Let $A$ be an $m \times n$ matrix. Then
      \[
      \sup_{x \in \S^{m-1}\atop y \in \S_{0}^{n-1}} \langle x,Ay\rangle \leq \frac{1}{1-2\ep} \sup_{x \in \mathcal{N}_{\ep}, y \in \mathcal{N}^{0}_{\ep}} | \langle x,Ay\rangle|.
      \]
  \end{lemma}
  
  \begin{proof}[Proof of Theorem \ref{thm: operator_bdd}]
   In view of Remark~\ref{onlysmalld} we may assume that $d\leq n/2$. 
  Fix  $\beta\geq 1$, $\beta_1\geq 16$, and $\beta_2 \geq 4$. 
  Let $\tau=\max \big\{e^2, \frac{\beta_2 n\log m}{md}\big\}$, and recall that $N= \max{ \big\{d, \frac{(1+\tau)md}{n} \big\}}$.
  Note that
  $$
   \mbox{ if }\quad  \frac{\beta_2 n\log m}{md}\geq e^2 \quad \quad \mbox{ then } \quad  \quad \frac{(1+\tau) md}{n} <2 \beta_2 \log m
$$ and otherwise, $(1+\tau) md/n \leq (1+\tau)d < 9d$. Therefore,  
  $N \leq  \max\{9 d, \, 2 \beta_2 \log m\}$, and by 
  Lemma~\ref{lem: colsum}, column sums of $M$ do not exceed $N$ with probability at least 
  $1-n\cdot m^{-\beta_2}$. Applying Lemma~\ref{lem: DPRCM} we obtain that 
  with probability at least $1-n\cdot  m^{-\beta_2}-  n^{-\beta}$, the matrix 
  $M$ has  DP$(\delta, \kappa_1, \kappa_2)$ with $\delta=d/n$, $\kappa_1=e^2$, 
  $\kappa_2=2(4+\beta)$. Then Lemma~\ref{rem: absolute} implies that with probability at least $1-n\cdot  m^{-\beta_2}-  n^{-\beta}$,
    \begin{equation} \label{forth41}
     \sup_{x \in \S^{m-1}\atop  y \in \S_{0}^{n-1}}  \left| \sum_{(i,j) \in \mathcal{H}(x,y)}M_{ij}x_iy_j  \right| \leq \alpha \sqrt{N},
   \end{equation}
      where $\alpha=16+32 \frac{(1+e^2)d}{N}+128 (4+\beta) ( 1+ e^{-2})$.
      Let $\mathcal{G}$ denote this event. As we mentioned above, 
      \[\| M-\mathbb{E}M \|= \sup_{x \in \S^{m-1}\atop y \in \S_{0}^{n-1}} \langle x,My\rangle.
      \]
      Fix $\ep=1/4$. By a standard volume argument, on $\S^{m-1}$ and $\S_{0}^{n-1}$ there exist $\ep$-nets $\mathcal{N}_{\ep}$ and $\mathcal{N}^{0}_{\ep}$ of size at most $9^{m}$ and $9^{n}$, respectively. 
   Combining (\ref{forth41}), Lemma \ref{lem: meanlightsec}, and the union 
   bound, we obtain 
      \begin{align*}
          \mathbb{P}\Biggl( \mathcal{G} &\cap \left\{\| M-\mathbb{E}M \|\geq 2 \alpha \sqrt{N } +\frac{2d}{\sqrt{N}} + 2\beta_1 \sqrt{N} \right\} \Biggr) \\
          &\leq \mathbb{P}\left( \mathcal{G} \cap \left\{ \sup_{x \in \mathcal{N}_{\ep}\atop y \in \mathcal{N}^{0}_{\ep}} 
          | \langle x,My\rangle| \geq  \alpha \sqrt{N } +\frac{d}{\sqrt{N}} + \beta_1 \sqrt{N} \right\} \right)\\
          &\leq \mathbb{P}\left( \mathcal{G} \cap \left\{\sup_{x \in \mathcal{N}_{\ep}\atop y \in \mathcal{N}^{0}_{\ep}} 
          \left| \sum_{(i,j) \in \mathcal{L}(x,y)}M_{ij}x_iy_j \right| \geq  \alpha \sqrt{N } +\frac{d}{\sqrt{N}} + 
          \beta_1 \sqrt{N} -\sup_{x \in \mathcal{N}_{\ep}\atop y \in \mathcal{N}^{0}_{\ep}} \left|\sum_{(i,j) \in \mathcal{H}(x,y)}M_{ij}x_iy_j \right|\right\} \right)\\
          &\leq \mathbb{P}\left( \mathcal{G} \cap \left\{\sup_{x \in \mathcal{N}_{\ep}\atop y \in \mathcal{N}^{0}_{\ep}} \left| \sum_{(i,j) \in \mathcal{L}(x,y)}M_{ij}x_iy_j \right| \geq  \frac{d}{\sqrt{N}} + \beta_1 \sqrt{N} \right\} \right) \\
          &\leq 4 \cdot 9^{2n} \exp \left(-\frac{3\beta_1^2 n}{4(6+\beta_1)} \right) \leq \exp \left(-\beta_1 n/8 \right).
      \end{align*}
      Finally, to obtain the desired result, 
          we choose $\gamma=\beta _2$ and $\beta_1 = 16 \beta$. 
  \end{proof}

\subsection{Tall rectangular matrices}
\cref{thm: operator_bdd} is stated (and used below) for $1\le m\le n$. For the sake of completeness, we generalize it  
to the complementary regime $m> n$. This almost immediately follows by partitioning the rows of the matrix into blocks 
of size comparable to $n$, applying \cref{thm: operator_bdd} to each block, and then combing the estimates. 

\begin{thm}\label{cor:oper_bdd}
   Let $m, n\ge 1$, $\gamma\geq 4$, $\beta \geq 1$, and let $d$ be an integer satisfying $1\le d \le n$.  Let $M$ be an $m \times n$ random matrix with entries in $\{0, 1\}$, where each row is independently and uniformly sampled from the set of all vectors in $\{0, 1\}^n$ containing exactly $d$ ones. Denote $k:=\lceil m/n \rceil$. Then, with probability at least $1-k(2^\gamma n^{1-\gamma}+2n^{-\beta})$, one has
   \[
   \|M-\E M\|\le C_{\ref{thm: operator_bdd}}\beta\sqrt{k}\sqrt{\min\{d, n-d\}+\gamma \log n}
   \]
   where $C_{\ref{thm: operator_bdd}}$ denotes the absolute constant from \cref{thm: operator_bdd}.
\end{thm}

\begin{remark}
   Taking $\gamma=4$ in \cref{cor:oper_bdd}. If $\log n\le d \le n/2$ and $m\geq n$, then 
   \[
   \|M- \E M\|\le C \beta\sqrt{\frac{md}{n}}
   \]
   for some absolute constant $C>0$. 

   On the other hand, assuming that  $d\leq n/2$, note that 
   each row of $M-\E M$ has the squared Euclidean norm $d(n-d)/n\geq d/2$. 
   Thus, $\|M-\E M\|_{HS}^2\geq md/2$. Moreover, $(M-\E M)\1=0$, and hence, 
   $r:=\text{rank}(M-\E M)\le n-1$. 
   Denoting by $s_1\ge \cdots \ge s_r>0$ 
   the non-zero singular values of $M-\E M$,  we observe 
   \[
   \|M- \E M\|_{HS}^2=\sum_{i=1}^r s_i^2\le r s_1^2=r \|M-\E M\|^2.
   \]
   This implies,
   \[
   \|M-\E M\|^2\ge \frac{\|M- \E M\|_{HS}^2}{r}\ge  \frac{md}{2(n-1)}.
   \]
   Thus, $\|M- \E M\|\ge \sqrt{md/(2(n-1))}$. 
   Therefore, for $\log n \le d \le n/2$ and  $m\geq n$, the bound in \cref{cor:oper_bdd} is sharp up to an absolute constant.
\end{remark}

\begin{proof}
First note that the case $m\leq n$ is covered by \cref{thm: operator_bdd}, so we 
assume that $m>n$ and hence $k\ge 2$. Denote $d_*:=\min\{d, n-d\}$. The case $n=1$ is trivial, so 
we also assume $n\ge 2$. 

Note that for positive integers $m_1,\dots, m_k$ satisfying $n/2\le m_i\le n$ their sum 
$\sum_{i=1}^k m_i$ can take any value between $k\lceil n/2 \rceil$ and $kn$. 
Since $k(n+1)/2\leq m \leq kn$ for $m>n>1$,  there exist positive integers $m_1,\dots, m_k$ such that
    \[
    \sum_{i=1}^k m_i=m   \quad \text{ and } \quad \frac{n}{2} \le m_i\le n \, \text{ for all } i=1,\dots, k.
    \]

Partition the set of row indices into pairwise disjoint consecutive blocks $I_1, \dots, I_k$ with $|I_j|=m_j$ for $j=1,\dots, k$, and let $M^{(j)}$ be the $m_j\times n$ submatrix of $M$ formed by rows indexed by $I_j$. Then
\[
M-\E M =\begin{bmatrix}
    M^{(1)}-\E M^{(1)} \\ \vdots \\ M^{(k)}-\E M^{(k)}
\end{bmatrix}.
\]
For each $j=1,\dots, k$, the rows of $M^{(j)}$ are still independent. Since $m_j\le n$, we apply \cref{thm: operator_bdd} to $M^{(j)}$. Therefore, for $j=1,\dots, k$,
\[
\P\left(\left\|M^{(j)}-\E M^{(j)}\right\|>C_{\ref{thm: operator_bdd}}\beta\sqrt{d_*+\gamma \log m_j}\right) \le nm_j^{-\gamma}+2n^{-\beta}.
\]
Because $m_j\le n$, we have $\log m_j\le \log n$, and because $m_j\ge n/2$, we have $nm_j^{-\gamma}\le 2^{\gamma } n^{1-\gamma}$. Thus, for $j=1,\dots, k$,
\[
\P\left(\left\|M^{(j)}-\E M^{(j)}\right\|>C_{\ref{thm: operator_bdd}}\beta\sqrt{d_*+\gamma \log n}\right) \le 2^\gamma n^{1-\gamma}+2n^{-\beta}.
\]
By the union bound, with probability at least $1-k( 2^\gamma n^{1-\gamma}+2n^{-\beta})$, one has 
\[
\left\|M^{(j)}-\E M^{(j)}\right\|\le C_{\ref{thm: operator_bdd}}\beta\sqrt{d_*+\gamma \log n}   \quad \text{ for all } \,  j=1,\dots, k.
\]
Then, on this event we have for  any $y\in \S^{n-1}$,
\begin{align*}
    \|(M-\E M)y\|_2^2&=\sum_{j=1}^k\left\|(M^{(j)}-\E M^{(j)})y\right\|_2^2
    \\ &\le \sum_{j=1}^k\left\|(M^{(j)}-\E M^{(j)})\right\|_2^2\|y\|_2^2
    \le kC_{\ref{thm: operator_bdd}}^2\beta^2(d_*+\gamma \log n).
\end{align*}
This implies the desired result.
\end{proof}

\section{Concentration results and graph expansion properties}
\label{sec:concen_expan}

\subsection{Auxiliary lemmas}
\label{sec-five}


In this section we present a pair of lemmas that control the connectivity between the rows or columns of $M$ and an arbitrary subset of indices. The first lemma provides a probabilistic bound on the number of rows of $M$ that have an atypically large number of non-zero entries within a given set of columns $J$. The second simple lemma  offers a known deterministic statement for the columns of $M$. It is essentially  \cite[Claim 3.6]{litvak2019smallest} (where a different model was used). We provide its proof for the sake of completeness. 

\begin{lemma}\label{prop: supprow}
Let $1\le d, k \le n$ and $r\geq 2$. Let $J \subset [n]$, $|J|=k$. 
Let $M$ be uniformly drawn from $\cM_{n,d}$. Then with probability at 
least $1- \exp(-(8n/r))$,
\[
\left|\Big\{i \leq n: |\text{Supp}(R_{i}(M)) \cap J| \geq \frac{r kd}{n} \Big\}\right| 
\leq \frac{9 n}{r}.
\] 
\end{lemma}

\begin{proof}
Fix $J\subset [n]$ with $|J|=k$. Denote by $\mathcal{E}_i$ the event that $|\text{Supp}(R_{i}(M)) \cap J| \geq \frac{r kd}{n}$. 
Note that the random variable $|\text{Supp }R_{i}(M) \cap J|$ has a hypergeometric distribution with mean $kd/n$. Therefore, applying Markov's inequality, we have 
\[
   q:=\mathbb{P}(\mathcal{E}_i) \le \frac{\E\lr{|\text{Supp} (R_{i}(M)) \cap J|}}{rkd/n} \leq \frac{1}{r}.
\]
Let $\delta_{i}$ be the indicator function of the event $\mathcal{E}_i$, $i\leq n$. Note that events 
$\mathcal{E}_i$'s (hence random variables $\de_i$'s) are independent. Since $q\leq 1/r$, 
$\tau:= 9/(rq) -1 \geq 8\geq e^2$.  
Then, writing 
$9n/r= (1+ \tau)qn$ and applying Lemma~\ref{Benet}, we observe
$$
 \mathbb{P}\lr{\sum_{i=1}^n \delta_{i} \geq \frac{9n}{r}} \leq \exp(-(9/(rq) -1)qn) \leq 
 \exp(-(8n/r)).  
$$
This completes the proof.
\end{proof}

The following lemma was proved in \cite{litvak2019smallest} (see Claim 3.6 there).
\begin{lemma}\cite[Claim 3.6]{litvak2019smallest} \label{lemma: suppcol}
Let $1\leq d, k, \leq n$. 
Let $J \subset [n]$, $|J|=k$, and $A>1$. Let $M \in \cM_{n,d}$. Then
\[
\left|\Big\{i \leq n: |\text{Supp}(R_{i}(M^{T})) \cap J| \geq \frac{Akd}{n} \Big\}\right| \leq \frac{n}{A}.
\]
\end{lemma}

\begin{proof}
By the definition, the total number of ones in the submatrix of $M$ indexed by $J\times [n]$ is $kd.$ Therefore, 
\[
\left|\Big\{i \leq n: |\text{Supp}(R_{i}(M^{T})) \cap J| \geq \frac{Akd}{n} \Big\}\right| \cdot \frac{Akd}{n}
\leq kd,
\]
which yields the desired bound.
\end{proof}

\subsection{The triple norm}

Recall that $\1=(1,1,..., 1)$ and 
 denote $\e=\1/\sqrt{n}$.  Let $P_\e$ be the projection on $\e$, and let $P_{\e^\perp}$ be the projection on $\e^\perp$. Consider the triple norm on $\C^n$ defined by
\[
\tnrm{x}^2:=\|P_{\e^\perp} x\|_2^2+d\|P_\e x\|_2^2.
\]
The triple norm will be used later for net construction.

The following proposition is an adaptation of a result from \cite[Proposition 3.14]{litvak2022singularity}. As the proof is analogous to that in \cite{litvak2022singularity}, we omit it. 

\begin{prop}\label{prp: nrm}
    Let $n$ be large enough and $ \log n \le d\le n/2$. Denote
    \[\cE_{\ref{prp: nrm}}:=\set{M\in \cM_{n, d}: \|M-\E M\|\le  C_{\ref{thm: operator_bdd}}\sqrt{d}}.
    \]
    Then for every $x\in \C^n$  and every $M\in \cE_{\ref{prp: nrm}}$ one has $\|Mx\|\le \lr{C_{\ref{thm: operator_bdd}}+1}\sqrt{d}\, \tnrm{x}$.
\end{prop}

A similar result holds for $M^T$.

\begin{prop}\label{lem:nrmleft}
    Let $n$ be large enough and $\log n\le d\le n/2$. Denote
    \[
    \cE_{\ref{lem:nrmleft}}:=\set{M\in \cM_{n,d}: \|M^T-\E M^T\|\le C_{\ref{thm: operator_bdd}}\sqrt{d}
    \quad \mbox{ and }\quad  \|M^T\1\|_2\le (1+C_{\ref{thm: operator_bdd}})d\sqrt{n}  }.
    \]
    Then for every $x\in \C^n$  and every $M\in \cE_{\ref{lem:nrmleft}}$ one has $\|M^{T}x\|\le \lr{2C_{\ref{thm: operator_bdd}}+1}\sqrt{d}\, \tnrm{x}$.
\end{prop}

\subsection{Vertex expansion property}\label{sec:vectex_exp}

In this section we establish two vertex expansion properties for in-neighbors and out-neighbors 
of the random matrix $M$ uniformly drawn from $\cM_{n,d}$.  
The first property, \cref{Expansion1}, is a consequence of the independence between the rows of $M$, which allows 
for a direct application of the classical Chernoff bound for independent random variables. In the contrast, 
the second  property, \cref{Expansion}, is more subtle, as the corresponding indicator variables are not independent. 
To overcome this difficulty, we leverage the property of negative association (NA) introduced in \cref{ssna} and apply a Chernoff concentration for NA random variables.

Let $M\in \cM_{n, d}$. For $J\subset [n]$ denote 
$$S(J,M):=\{i\in [n]: \mbox{Supp} (R_i(M))\cap J\neq \emptyset\}.$$
Given $1\le k\le n$ and $\varepsilon\in (0,1/2]$, let 
    \[
  \Om_{k, \varepsilon}^{in}:=\set{M\in \cM_{n,d}: \forall J\subset [n], |J|=k \mbox{ one has } (1-\varepsilon)kd \le |S(J,M)|\le (1+\varepsilon)kd}
    \]
    and 
 \[
  \Om_{k, \varepsilon}^{out}:=\set{M\in \cM_{n,d}: \forall J\subset [n], |J|=k \mbox{ one has } |S(J,M^{T})|\ge (1-\varepsilon)kd}.
\]
We  also denote 
\begin{equation}\label{generalomega}
 \Om_{k, \varepsilon}:=   \Om_{k, \varepsilon}^{in}\cap   \Om_{k, \varepsilon}^{out}.
\end{equation}
Note that if we interpret a matrix $M \in \cM_{n,d}$ as the adjacency matrix of a directed graph $G$ with vertex set $[n]$, then $S(J, M)$ and $S(J, M^T)$ correspond to the sets of in-neighbors and out-neighbors, respectively, of the vertices in $J \subset [n]$. 
    Moreover, $|S(J,M^{T})| \leq d|J|$, since each column of $M^T$ contains exactly $d$ ones.


\begin{lemma}\label{Expansion1}
    Let $n\geq 3$, $192 \log n \le d \le n/2$,
     $\varepsilon\in \Big[\sqrt{\frac{48\log n}{d}}, 1\Big)$, and $1 \le k\le \varepsilon n/d$. 
    Then
    \[
    \P(\Om_{k, \varepsilon}^{in})\ge 1-2\binom{n}{k}\exp\lr{-\frac{\varepsilon^2 kd}{24}}.
    \]
 Moreover, 
\[
 \P\lr{\bigcap_{k \leq \varepsilon n /d} 
 \Om_{k, \varepsilon}^{in}}\ge 1-3\exp\lr{-\frac{\varepsilon^2 d}{48}}.
\]
\end{lemma}

\begin{proof}
Fix a set $J\subset [n]$ with $|J|=k$.  Let $\xi_i$, $i\leq n,$ be the indicator function on the event $\{i\in S(J, M)\}$. 
Then we have
\[
|S(J, M)|=\sum_{i=1}^n \xi_i.
\]
As rows of $M$ are independent, $\xi_i$'s are independent Bernoulli random variables with 
$$
 \P(\xi_i=0)=\P(\mbox{Supp} (R_i(M))\cap J=\emptyset)=\frac{\binom{n-k}{d}}{\binom{n}{d}} 
$$
and 
\[
q:=\P(\xi_i=1)=1-\frac{\binom{n-k}{d}}{\binom{n}{d}}=1-\prod_{i=1}^{k} \left(1-\frac{d}{n-i+1}\right),
\]


Let $\mu:=\E\lr{|S(J,M)|}$. Then 
\[
n \left(1-\left(1-\frac{d}{n}\right)^k \right) \leq \mu =\sum_{i=1}^{n} \E\xi_i=n q 
\leq n \left(1-\left(1-\frac{d}{n-k+1}\right)^k \right).
\]
For $1 \leq k \leq \frac{\ep n}{d}$ and $d\geq 4$, using the inequality $1-kx \leq (1-x)^k \leq 1-kx+\frac{k(k-1)x^2}{2}$, for $x\in (0,1)$, we obtain that 
\[
kd (1-\ep/2) \leq \mu \leq kd(1+\ep/3).
\]
  Chernoff bound (see, e.g., \cite[Theorem 1.1]{dubhashi2009concentration}), for every $\delta>0$,
\[
\P(|S(J,M)|-\mu| \geq \de \mu) \leq 2 e^{-\de^2 \mu /3}. 
\]
Set $\de=\ep/2$, and combining the estimation for $\mu$, we have 
\[
  \P \left( |S(J,M)|-kd| \geq \ep kd \right) \leq 2 \exp\left({-\frac{\ep^2(1-\ep/2)kd}{12}}\right)
  \leq 2 \exp\left({-\frac{\ep^2 kd}{24}}\right) :=p_k.
\]
Then the first result follows from the union bound over $J\subset [n]$, $|J|=k$. 
Using that $\ep^2 d \geq 48 \log n$, 
$$
   \frac{\binom{n}{k+1}\, p_{k+1}}{\binom{n}{k} \, p_k}\leq n \exp\left({-\ep^2  d/24}\right)
   \leq  \exp\left({-\ep^2  d/48}\right),
$$ 
and  the union bound over all $k$,
\[
\P\lr{\bigcup_{k \leq \varepsilon n /d} (\Om_{k,\varepsilon}^{in})^c}\leq 2\sum_{k\leq \varepsilon n /d} \binom{n}{k}\exp\lr{-\frac{\varepsilon^2 kd}{24}} \leq 3 \exp\lr{-\frac{\varepsilon^2 d}{48}} 
\]
for $n\geq 3$. This implies the moreover part.  
\end{proof}


\begin{lemma}\label{Expansion} Let $n\geq 2$ and 
     $128\log n \le d \le n/2$. Let $\varepsilon\in \Big[\sqrt{\frac{32\log n}{d}},1\Big)$. 
    Let $1 \le k\le \varepsilon n/d$. Then
    \[
    \P(\Om_{k, \varepsilon}^{out})\ge 1-\binom{n}{k}\exp\lr{-\frac{\varepsilon^2 kd}{16}}.
    \]
Moreover,
\[
\P\lr{\bigcap_{k \leq \varepsilon n /d} \Om_{k, \varepsilon}^{out}}\ge 1-2\exp\lr{-\frac{\varepsilon^2 d}{32}}.
\]
\end{lemma}

\begin{proof}
Fix $J\subset [n]$ with $|J|=k$.  Let $\xi_i$, $i\leq n,$ be the indicator function on the event $\{i\in S(J, M^{T})\}$. Then we have
\[
|S(J,M^T)|=\sum_{i=1}^n \xi_i.
\]

Let $\cC_j=\set{M_{ij}^T: i\in [n]}$, $j\leq n$. Applying 
Lemma~\ref{lemma: Negas} to the matrix consisting of rows of $M$ indexed by $j\in J$, 
we observe that 
$$\bigcup_{j\in J}\cC_j=\set{M_{ij}^T\,\, :\,\, i\in [n], j\in J}$$ 
is an NA family.
Note that $\xi_i=\max_{j\in J}M^T_{ij}$ for $i\in [n]$. For each row $i$, define $A_i:=\{(i, j): j\in J\}$. 
Clearly, $A_i\cap A_{i'}=\emptyset$ whenever $i\neq i'$. Then the family $\cA=\{A_1,\dots, A_n\}$ 
consists of $n$ disjoint subsets of $\{(i, j): i\in [n], j\in J\}$. For each $A_i$, consider the function 
\[
 h_{A_i}\lr{ (M^T_{ij})_{j\in J}}:=\max_{j\in J}M^T_{ij}
\]
which is coordinate-wise non-decreasing. 
By the disjoint monotone aggregation of NA in \cref{ssna}, the family 
$$\set{h_{A_i}\lr{M^T_{ij}, j\in J}: i\in [n]}=\{\xi_i: i\in [n]\}$$ is NA.

Then by Chernoff-Hoeffding bounds for NA random variables (see, e.g., \cite[Theorem 3.1]{dubhashi2009concentration}), for every $\delta\in (0,1)$
\[
\P\lr{|S(J, M^{T})|<(1-\delta)\mu}\le \exp\lr{-\frac{1}{2}\delta^2\mu},
\]
where $\mu:=\E\lr{|S(J, M^{T})|}$.

\smallskip

Let $\Tc_j=\Tc_j(M^T)$ be the $j$-th column of $M^T$. Note that for every $i\in [n]$
\[\E\xi_i=\P(\xi_i=1)=1-\P\lr{i\notin \mbox{Supp}\lr{\Tc_j}, \forall j\in J}=1-\prod_{j\in J}\P\lr{i\notin \mbox{Supp}\lr{\Tc_j}}=1-\lr{1-\frac{d}{n}}^k.
\]
Since $\mu=\E\sum_{i=1}^n\xi_i$, then for $1 \leq k \leq \varepsilon n/d$
\[
\mu=n\lr{1-\lr{1-\frac{d}{n}}^k}\ge dk-\frac{k(k-1)d^2}{2n} \ge kd\lr{1-\frac{\varepsilon}{2}},
\]
where we again used that for $x\in (0,1)$, $(1-x)^m\le 1-mx+\binom{m}{2}x^2$ with $x=d/n$.  
Take $\delta=\varepsilon/2$. Since $(1-\delta)\mu\ge (1-\varepsilon)kd$,
we have
\[
\P\lr{|S(J, M^{T})|<(1-\varepsilon)kd}\le \exp\lr{-\frac{\varepsilon^2(1-\varepsilon/2)}{8}kd}\le 
\exp\lr{-\frac{\varepsilon^2 kd}{16}}.
\]
Therefore, by the union bound,
\[
\P(\Om_{k, \varepsilon}^{out})\ge 1-\binom{n}{k}\exp\lr{-\frac{\varepsilon^2 kd}{16}}.
\]
The proof of the moreover part is similar to the  
end of the proof of Lemma~\ref{Expansion1}. Using $\ep^2 d\geq 32\ln n$, we observe
\[
\P\lr{\bigcup_{k \leq \varepsilon n /d} (\Om_{k,\varepsilon}^{out})^c}\leq \sum_{k\leq \varepsilon n /d} \binom{n}{k}\exp\lr{-\frac{\varepsilon^2 kd}{16}} \leq 2 \exp\lr{-\frac{\varepsilon^2 d}{32}} .
\]
for $n\geq 2$.  This completes the proof.
\end{proof}

We need  the following notation. Given two disjoint sets $J^{\ell}, J^{r} \subset [n]$ and a matrix $A$ with $\{0,1\}$ entries, denote 
\[
I^{\ell}(A)=I^{\ell} (A, J^{\ell},J^{r}):= \{ i \leq n: |\mbox{Supp}(R_{i}(A)) \cap J^{\ell}|=1 \quad \mbox{and} \quad \mbox{Supp}(R_{i}(A)) \cap J^{r}= \emptyset\}
\]
and 
\[
I^{r}(A)=I^{r} (A, J^{\ell},J^{r}):= \{ i \leq n: \mbox{Supp}(R_{i}(A)) \cap J^{\ell}= \emptyset \quad \mbox{and} \quad |\mbox{Supp}(R_{i}(A)) \cap J^{r}|=1 \}.
\]

We now combine the expansion properties established in Lemmas~\ref{Expansion1}  and \ref{Expansion} 
to derive a key combinatorial property of matrices in the high-probability event 
$\cE_{\ref{lem: rowsum}} \cap \Om_{s,\varepsilon}^{in} \cap \Om_{s,\varepsilon}^{out}$, 
where $s$ is some parameter and $\cE_{\ref{lem: rowsum}}$ is the event from Lemma~\ref{lem: rowsum} 
(in the case of square matrices, that is, when $m=n$ in  Lemma~\ref{lem: rowsum}). 
The following lemma is an adaptation of the argument in \cite[Lemma~2.7]{litvak2019smallest} to our setting. 
While their result was for $d$-regular matrices, one can show that the same combinatorial counting argument holds for our model, 
by using the expansion properties for both in-neighbors and out-neighbors. The proof follows the same line of proof as in \cite[Lemma~2.7]{litvak2019smallest}, and for the sake of brevity we omit it.

\begin{lemma}\label{lem: omega_ke}
Let $5000\log n \le d \le n/2$ and $\varepsilon \in \Big[\sqrt{\frac{48\log n}{d}},0.1\Big)$. Let $q \geq 2, m \geq 1$ be integers satisfying $qm \leq \varepsilon n /d$. Let $J^{\ell}, J^{r} \subset [n]$ be such that 
$$
 J^{\ell} \cap J^{r}= \emptyset, \,\, |J^{\ell}|=m, \,\,\,\,  \mbox{ and }\,\, \,\, |J^{r}|=(q-1)m.
$$
 Let $M \in \cE_{\ref{lem: rowsum}} \cap \Om_{qm,\varepsilon}$. Then
\[
 |I^{\ell}(M)| \geq (1-4\varepsilon q)d|J^{\ell}| \quad \mbox{and} \quad |I^{\ell}(M^{T})| \geq (1-2\varepsilon q)d|J^{\ell}|.
\]
In particular, if $|J^{r}|=|J^{\ell}|=m$ then
\[
(1-8\varepsilon ) dm  \leq \min \{|I^{\ell}(M)|,|I^{r}(M)|\} \leq \max \{|I^{\ell}(M)|,|I^{r}(M)|\} \leq (1+\ep)dm
\]
and
\[
(1-4\varepsilon ) dm  \leq \min \{|I^{\ell}(M^T)|,|I^{r}(M^T)|\} \leq \max \{|I^{\ell}(M^T)|,|I^{r}(M^T)|\} \leq dm.
\]
\end{lemma}

\section{Invertibility  on the almost constant vectors }
\label{sec: inverb_almost_const}

In this section, we establish the invertibility of the shifted matrix $M-z\I_n$ and its transpose when acting on vectors from the class of almost-constant vectors, $\mbox{Cons}(\delta, \rho)$. To this end, we provide separate lower bounds for both the right-multiplication $\|(M-z\I_n)x\|_2$ and the left-multiplication $\|(M^{T}-z\I_n)x\|_2$, which require different treatments due to asymmetry in the distribution of our model. We will split almost-constant vectors into two parts: almost constant steep vectors (those with ``jump", that is, $x_i^*\gg x_j^*$ for some $i\ll j$) and almost constant gradual vectors (those without jumps). Then we further classify our steep vectors as $\cT_1, \cT_2,\cT_3$  based on a careful selection of the magnitude and location of the first significant jump, see \cref{steepvectors} below for  precise definitions. We will proceed with our proof by using different techniques for almost constant gradual vectors and each subclass $\cT_1, \cT_2, \cT_3$ of steep vectors.

We first prove a bound for almost constant gradual vectors in \cref{thm: invnonsteep}, which is less involved as discussed in the introduction.

 For vectors in $\cT_1$, which have a large jump (of order $d$), our approach is to partition the initial block of coordinates (from $x_1^*$ to $x_{n_1}^*$) into segments whose lengths are powers of $\ell_0$, and then iteratively apply \cref{lem: omega_ke}.
This approach leverages the expansion properties to find a large set of ``good" rows where the inner product with a vector $x\in \cT_1$ is dominated by its unique large coordinates.  A similar idea was also used in \cite[Lemma 3.7]{litvak2019smallest} for random $d$-regular matrices. However, unlike random $d$-regular matrices, we must handle the left- and right-multiplication separately due to asymmetry in distribution (see \cref{sec:t1_step}). To implement this idea, the parameters $n_1, \ell_0$ must be carefully chosen. In subclass $\cT_1$, the parameter $\ep_0\lesssim \sqrt{\log n/d}$ is selected as the minimum value of $\ep$ to effectively apply the expansion properties in Lemmas~\ref{Expansion1} and \ref{Expansion}. This selection of $\ep_0$  in turn determines the choice of $n_1\le \epsilon_0n/d\lesssim (\log n/d)^{1/2}(n/d)$. The choice of  $\ell_0\lesssim 1/\ep_0$ comes from $\ell_0\ep_0 < 1/4 $  needed in \cref{lem: omega_ke}.

For vectors in $\cT_2$ and $\cT_3$, where the jumps are  moderate, the direct combinatorial argument is insufficient.  We therefore use the $\ep$-net argument, adapting the strategy from \cite[Lemma 6.7]{litvak2022singularity}. 
To construct the nets, the triple norm is utilized, and the fact that $\|M-\E M\|$ is bounded by $O(\sqrt{d})$ and our choice of $n_1$ leads to the choice of $\ep=1/d^{3/4}$, which satisfies the condition of \cref{lem: individ1}. This, in turn, constrains the scale of $n_2\lesssim n/d^{3/4}$ and the magnitude of the shift $|z|\lesssim \sqrt{d}\log\log d$ in the proof. The parameter $n_3\approx \delta n$ is chosen to align with the definition of almost-constant vectors. To optimize the final probability bound,  $\delta$ is chosen to be a positive absolute constant. 
To effectively apply the $\ep$-net argument, we need to balance the size of the nets and the individual probability bound. Here again, due to the asymmetry in the distribution of our model, to obtain individual probability bounds, we need to deal with left- and right-multiplication separately. For the right multiplication, the individual probability bound is attained by anti-concentration type techniques in companion with row independence. In contrast, for the left multiplication, due to the lack of independence, the individual probability bound is obtained by anti-concentration type techniques together with a switching argument. The delicate interplay between the expansion properties and the jump structures is essential for the argument. All technical details of this approach are presented in \cref{individual}.

\subsection{Almost-constant vectors}
As in \cite{litvak2017adjacency, litvak2019smallest}, for  
 a lower bound on the smallest singular value $s_n(M-zI_n)$, we consider a decomposition of the complex unit sphere $\S^{2n-1}$ into \textit{almost-constant} and \textit{non-almost-constant vectors}.

\begin{defn}\label{def:almost_constant}
    Let $\delta, \rho \in (0,1)$. We define the set of  almost-constant  $\mbox{Cons}(\delta, \rho)\subset \C^{n}$  as the set of vectors $v\in \C^{n}$ for which there exists a complex number $\lambda$ such that the number of coordinates $i\in [n]$ satisfying $|v_i-\lambda|\le\rho\|v\|_2/\sqrt{n}$ is strictly larger than $(1-\delta)n$. The vectors in the remaining set $\lr{\mbox{Cons}(\delta, \rho)}^c=\C^n\setminus \mbox{Cons}(\delta, \rho)$ are called non-almost-constant.  
\end{defn}


We recall the  definition of the L\'evy concentration function, which roughly measures the spread of a random variable. 
\begin{defn}
    For a random vector $X\in \R^n$ and $\varepsilon\ge 0$, the L\'evy concentration of $X$ of width $\varepsilon$ is 
    \begin{equation*}
        \Ls(X,\varepsilon):=\sup_{x\in \R^n} \P(\|X-x\|_2<\varepsilon).
    \end{equation*}
\end{defn}

We will use the following bound on the L\'evy concentration function of sums of independent random variables. 

\begin{lemma}\cite[Corollary 1 of Theorem 6.1]{esseen1968concentration}\label{rogozin1961increase}
    There exists an absolute positive constant $C_{\ref{rogozin1961increase}}$ such that the following holds. Let $n\ge 1$ and let $\xi_1, \dots, \xi_n$ be independent random vectors in $\R^2$.  Then for any $t\ge t_0>0$, one has
    \[
    \cL\lr{\sum_{i=1}^n\xi_i, t}\le \frac{C_{\ref{rogozin1961increase}} (t/t_0)^2}{\sqrt{n-\sum_{i=1}^n\Ls(\xi_i,t_0)}}.
    \]
    In particular, if $\alpha\ge  \max_{1\le i\le n}\Ls(\xi_i,t_0)$, then
    \[
     \cL\lr{\sum_{i=1}^n\xi_i, t}\le \frac{C_{\ref{rogozin1961increase}}}{\sqrt{n(1-\alpha)}}.
    \]
\end{lemma}

For non-almost-constant vectors, we have the following small ball estimate. 

\begin{lemma}\cite[Lemma 8.5]{sah2023limiting}\label{lem: weak_smallball_notalmost}
    There exists an absolute constant $C_{\ref{lem: weak_smallball_notalmost}}>0$ such that the following holds. Let $1\le d\le n/2$ and $\delta, \rho>0$. Sample $\xi\in \{0, 1\}^n$ uniformly at random with exactly $d$ ones. Let $v\in \lr{\mbox{Cons}(\delta, \rho)}^c$, 
that is
\[
\sup_{\lambda\in \C} \abs{\set{i\in [n]: |v_i-\lam|\le \rho /\sqrt{n}}}\le (1-\delta)n.
\]
Then 
\[
\cL\lr{\inn{\xi, v}, \rho /\sqrt{n}}\le \frac{C_{\ref{lem: weak_smallball_notalmost}}}{\sqrt{\delta d}}+C_{\ref{lem: weak_smallball_notalmost}}
\exp{(-\delta^2 d/C_{\ref{lem: weak_smallball_notalmost}})}.
\]
\end{lemma}

Next we introduce parameters which will be used in the remaining part of the paper. 
We always assume that $C_{\ref{thm: lsvlowertails}}$ is a large enough absolute constant, that $n$ is large enough, and that   $C_{\ref{thm: lsvlowertails}}\log n \leq d \leq n/2$. Fix a sufficiently small absolute positive constant $\delta$. In order to use Lemma~\ref{Expansion}, we fix 
\begin{equation*}
    \ep_0=\sqrt{\frac{48\log n}{d}} \quad  \quad \mbox{ and } \quad \quad\ell_0
    =\left\lfloor{1/(100\ep_0)}\right\rfloor
    =\left\lfloor{\frac{\sqrt{d}}{100 \sqrt{48\log n}}}\right\rfloor\geq 2.
\end{equation*}
Let $a_{1}, a_2, a_3$ be three positive, sufficiently small, absolute constants, satisfying
\[
a_3 \leq a_2/30 \leq a_1/900.
\]
Set
\[
n_{1}:=\left\lceil\frac{a_1\varepsilon_0 n }{d}\right\rceil,
\quad  \quad
n_{2}:=\left\lfloor \frac{a_2 n}{d^{3/4}} \right\rfloor, \quad \mbox{and} \quad n_{3}:=\left \lfloor a_3 n\right\rfloor.
\]
Finally, if $n_1>1$, fix the integer $r \geq 0$ satisfying
$$
  \ell_0^r < n_1 \leq \ell_0^{r+1}.
$$

\subsubsection{$\cT$-vectors}\label{steepvectors}
We describe the class of steep vectors, $\cT$, which is the union of three parts: $\cT_1, \cT_2$ and $\cT_3$. If $n_1=1$ set $\cT_{1}=\emptyset$. Otherwise, set 
$$\cT_{10}=\set{x\in \C^n: x_1^*>6d x^{\ast}_{\min\{\ell_0,n_1\}}},$$ 
and, for $1 \leq i \leq r-1$, 
\[
\cT_{1i}:=\set{x\in \C^n: x\notin  \bigcup_{j=0}^{i-1}\cT_{1j} \mbox{ and } x_{\ell_0^{i}}^*>6dx_{\ell_0^{i+1}}^*},
\]
for $i=r$ 
\[
\cT_{1r}:=\set{x\in \C^n: x\notin  \bigcup_{j=0}^{r-1}\cT_{1j} \mbox{ and } x_{\ell_0^r}^*>6dx_{n_{1}}^*} .
\]
Then 
$$
 \cT_{1}:=\bigcup_{i=0}^{r} \cT_{1i}.
$$
Finally set 
\[
\cT_2:=\set{x\in \C^n: x\notin \cT_1 \mbox{ and } x_{n_{1}}^*>d^{3/4}  x_{n_{2}}^*},
\]
\[
\cT_3:=\set{x\in \C^n: x\notin \cT_1\cup \cT_2 \mbox{ and } x_{n_{2}}^*>4  x_{n_{3}}^*}, \quad \mbox{ and } \quad \cT:=\bigcup_{i=1}^3 \cT_i.
\]
\cref{fig:steep_vector_classification} offers a visual guide to this classification scheme.


In the case $n_1>1$ we denote
\begin{equation}\label{bt}
\btt= \frac{(6d)^{r+1}\, d^{1/4}}{26 \log^{1/4} n}, \quad 
\bt3=\frac{(6d)^{r+1}\, d}{26 \log^{1/4} n} \quad \mbox{ and } \quad 
     B_\cT=\btc=\frac{(6d)^{r+2}}{36(\log n)^{1/4}}
\end{equation}
and for $n_1=1$,
\begin{equation}\label{btn}
\btt=\sqrt{n}, \qquad    \bt3  =B_\cT=\sqrt{2n}.
\end{equation}

\begin{figure}[htp]
    \centering
    \includegraphics[width=0.8\textwidth]{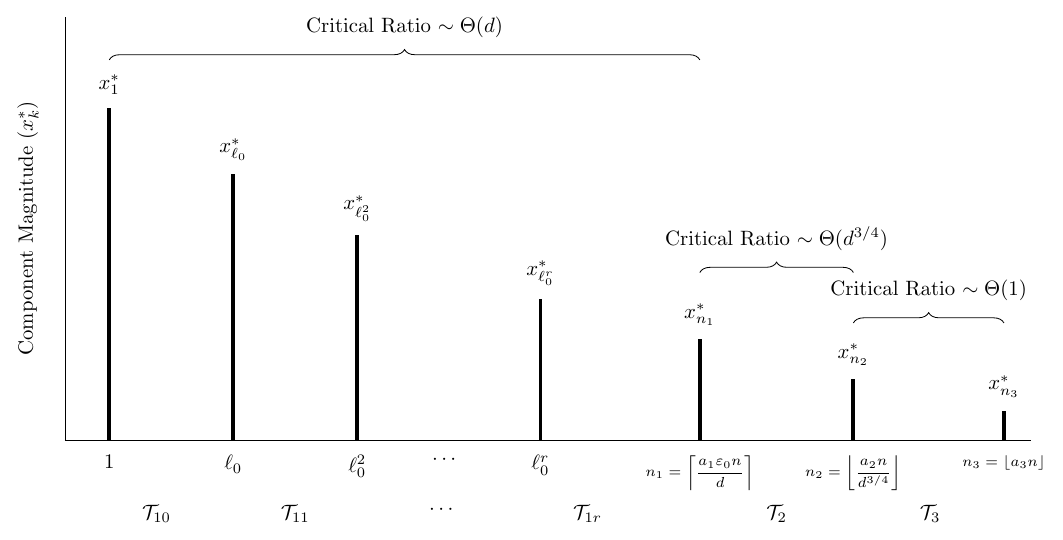}
    \caption{An illustration of the classification of steep vectors into the sets $\mathcal{T}_1$, $\mathcal{T}_2$, and $\mathcal{T}_3$. The vertical axis represents the magnitude of the sorted components ($x_k^*$) and the horizontal axis represents the component index $k$. The picture is drawn for the regime $\ell_0<n_1$, so that $\mathcal{T}_{10}$ is represented by the jump condition $x_1^*>6d\,x_{\ell_0}^*$. A vector is classified based on the location and the size of the first jump.}
    \label{fig:steep_vector_classification}
\end{figure}

\subsubsection{Bounds for $\ell_2$-norm}

The following lemma provides a bound on Euclidean norms of vectors in the class $\cT$ and its complement in terms of their order statistics. This result is similar to \cite[Lemma 6.4]{litvak2022singularity} (see also \cite[Lemma 3.5]{litvak2019smallest}, \cite[Lemma 4.3]{litvak2019structure}). Since our choice of parameters as well as the definition of steep vectors is slightly different, we provide the proof for the sake of completeness.

\begin{lemma}\label{lem: l2norm_linf} 
 There exists a positive absolute constant $C$ such that the following hold.   Let $n$ be large enough and let $C\log n\le d\le n/2$.  Let
       $y\in \cT_2$, $z\in \cT_3$, and $w\in \cT^c$. Then 
\[
      \|y\|_2\le \btt \, y_{n_{1}}^*,  \qquad 
    \|z\|_2\le \bt3\, z_{n_{2}}^*, \qquad \mbox{ and }  \qquad 
    \|w\|_2\le \, B_\cT \, w_{n_3}^*.
 \]
    Moreover, if $n_1>1$ (so $\cT_{1}\ne \emptyset$) and $x\in \cT_{1j}$ for some $0\le j\leq  r$, then 
 \[
    \|x\|_2\le \max{ \left\{   \frac{ d^{1/4} \, (6d)^{j}}{15\log^{1/4} n}, \sqrt{2 n}    \right\} }
    \, x_{\ell_0^j}^*. 
\]
\end{lemma}

\begin{proof}
 We first consider the case $n_1>1$. Note that in this case $(6d)^{2r+2}\geq 2 n$ for large enough $n$. 
Indeed, if $d^2\geq n$, then
$$(6d)^{2r+2}\geq (6d)^{2}\geq 36 n;$$ 
if $d^2\leq n$, then, using the definitions of $r$ and $n_1$, 
$$
 (6d)^{2r+2}\geq \exp\left( 2 (\log (6d) (\log n_1) /\log \ell_0 \right) 
 \geq n_1^4 \geq 2 n.
$$

  Let $z\in \cT_3$. Then 
    \[
    z_1^*\le (6d)z_{\ell_0}^*\le (6d)^2z_{\ell_0^2}^*\le \cdots \le (6d)^{r}z_{\ell_0^{r}}^* \le (6d)^{r+1}z_{n_1}^*\le (6d)^{r+1} \, d^{3/4}  \, z_{n_{2}}^*.
    \]
Recall that $\ell_0^{r} <  n_1\leq \ell_0^{r+1}$ and $\ell_0\leq \sqrt{d}/(690\sqrt{\log n})$. Therefore, for large enough $d$,  
\begin{align*}
    \|z\|_2^2&=\sum_{s=1}^{\ell_0}   (z_s^*)^2    +   \sum_{s=\ell_0+1}^{\ell_0^{2}}  (z_s^*)^2+
     \sum_{s=\ell_0^2+1}^{\ell_0^{3}}  (z_s^*)^2 
    + \cdots+  \sum_{s=\ell_0^{r}+1}^{n_1} (z_s^*)^2  +  \sum_{s=n_1+1}^n (z_s^*)^2 \\
    &\le \left(\ell_0(6d)^{2r+2}+\ell_0^2(6d)^{2r}+\ell_0^3(6d)^{2r-2}+\cdots+\ell_0^{r+1} (6d)^2+n\right) d^{3/2}\, (z_{n_{2}}^*)^2\\
    &\le \Big(\ell_0 (6d)^{2r+2}\sum_{s\ge 0}(6d)^{-2s} \ell_0^s +n\Big) d^{3/2}\, (z_{n_{2}}^*)^2 
    \le \Big(1.5 \ell_0 (6d)^{2r+2} +n\Big) d^{3/2}\, (z_{n_{2}}^*)^2 \\
    &\le 2 \ell_0 (6d)^{2r+2}  d^{3/2}\, (z_{n_{2}}^*)^2 \le  (6d)^{2r+2}  d^{2}\, (z_{n_{2}}^*)^2/(690\sqrt{\log n}).
\end{align*}
Thus, 
\begin{align*}
    \|z\|_2&\leq  (6d)^{r+1}\, d\, z_{n_{2}}^*/ (26 \log^{1/4} n).
\end{align*}
This completes the proof of the bound of $z$. 
The bounds for $y$ and $w$ are obtained similarly.

 \smallskip 

 Let $x\in \cT_{1j}$. If $j=0$ we clearly have $ \|x\|_2\leq \sqrt{n}  x_{1}^*$. If $1\le j\leq r$, then 
 \[
    x_1^*\le (6d) x_{\ell_0}^*\le (6d)^2 x_{\ell_0^2}^*\le \cdots \le (6d)^{j} x_{\ell_0^{j}}^* .
 \]
Similarly to the bound for $ \|z\|_2$, we obtain
\begin{align*}
    \|x\|_2^2
    &\le \left(\ell_0(6d)^{2j}+\ell_0^2(6d)^{2j-2}+\cdots+\ell_0^{j} (6d)^2+n\right) \, (x_{\ell_0^{j}}^*)^2\\
    &\le \Big(\ell_0 (6d)^{2j}\sum_{s\ge 0}(6d)^{-2s} \ell_0^s +n\Big)\, (x_{\ell_0^{j}}^*)^2 
    \le \Big(1.5 \ell_0 (6d)^{2j} +n\Big) \, (x_{\ell_0^{j}}^*)^2 \\
    &\le \Big(\sqrt{d} \, (6d)^{2j}/(460\sqrt{\log n}) +n\Big) \, (x_{\ell_0^{j}}^*)^2.
\end{align*}
This implies the desired bound.

The case $n_1=1$ is similar but simpler. Indeed, the bound for $y$ is trivial, for  $w\in \cT^c$ we have 
$w_j^*\leq d^{3/4}w^*_{n_2}\leq  4d^{3/4}w^*_{n_3}$ for $j< n_2$ and $w_j^*\leq  4w^*_{n_3}$ for 
$n_2<j\leq n_3$, hence 
$$
  \|w\|_2^2\leq (4d^{3/4} n_2 + 4 n_3 +n ) (w^*_{n_3})^2 \leq (4a_2 n + 4 a_3 n +n ) (w^*_{n_3})^2
  \leq 2 n (w^*_{n_3})^2. 
$$
The bound for $z$ is similar.  
\end{proof}

For the class of steep vectors $\cT$, we prove the following lower bound, in which we denote 
\[
    \cT':= \cT_1\cup\cT_2\cup \lr{\cT_3\cap \mbox{Cons}(a_3, \rho)}
\]
and 
$$ \theta(n, d) = 
    \min\left\{  \frac{\sqrt{d} }{4 \sqrt{n}} ,\,  
            \frac{5  \, \ell_0^{r/2}\, (d \log n)^{1/4}}{  (6d)^{r}}, \, 
     \, \frac{54 (\log n)^{1/4}\, \sqrt{n}}{(6d)^{r+2}}\right\}, 
$$ 
whenever $n_1>1$, 
$$ 
  \theta(n, d) = 
  c_1  \left((\log n)/d\right)^{1/4}
$$ 
for sufficiently small absolute constant $c_1>0$  in the case $n_1=1$.

\begin{thm}\label{thm: lower_bdd_T}
  Let $n$ be a sufficiently large  integer and  $C\log n\le d\le n/2$ for large enough absolute constant $C$. 
 Let $p=d/n$, $M\in \cM_{n,d}$, $0<\rho<\sqrt{n}/(B_\cT d^{3/4})$, and let $z\in \C$ be such that $|z|\le \sqrt{d}\log \log d$. Denote
    \[
    \cE_{\mbox{steep}}(M):=\set{\exists x\in \cT' \mbox{ such that } \|(M-z\I_n)x\|_2<\theta(n, d) \|x\|_2},
    \]
    and
    \[
    \cE_{\mbox{steep}}(M^T):=\set{\exists x\in \cT' \mbox{ such that } \|(M^{T}-z\I_n)x\|_2<\theta(n, d)\|x\|_2},
    \]
 Then
\[
\P\left ( \cE_{\mbox{steep}}(M) \bigcup  \cE_{\mbox{steep}}(M^T)\right )\le 22/n.
\]
\end{thm}

\subsubsection{Gradual vectors}

We now establish the lower bound for the almost-constant gradual vectors $\mbox{Cons}(a_3, \rho)\setminus \cT$, which by definition lack the significant jumps that characterize steep vectors. The proof of this theorem adapts the strategy from \cite[Theorem 3.1]{litvak2019smallest}. However, key modifications are necessary for our model. To deal with the right multiplication $\|(M-z\I_n)x\|_2$, we  use \cref{prop: supprow} to handle the non-fixed column sums of our matrix. Moreover, the argument is necessarily adapted to align with our specific selection of the jump size and its location for steep vectors. Recall that $B_\cT$ was defined in \eqref{bt}.

\begin{thm}\label{thm: invnonsteep}
   There exists a positive absolute constant $C$ such that the following hold. Let $n$ be large enough,   
   $C\log n\le d\le n/2$ and $M\in \cM_{n,d}$.  Let $0<\rho \le \sqrt{n}/(5 B_\cT)$ 
    and let $z\in \C$ be such that $|z|\le 4 d/25$. Then for all vectors $x\in (\mbox{Cons}(a_3, \rho)\setminus \cT)$ 
  with probability at least $1-4\exp(-(8n/81))-n\exp(-d/108)$, 
\[
    \|(M-z\I_n)x\|_2\ge \frac{3d\sqrt{n}}{25B_\cT} \|x\|_2 \ge  \theta(n, d) \|x\|_2,
\]
and deterministically, 
\[
    \|(M^{T}-z\I_n)x\|_2\ge  \frac{3d\sqrt{n}}{25B_\cT} \|x\|_2 \ge  \theta(n, d) \|x\|_2.
    \]
%
\end{thm}

\smallskip 

To combine Theorems~\ref{thm: lower_bdd_T} and \ref{thm: invnonsteep}, we introduce the following two events. 
Let $\rho>0$ and $z\in \C$  satisfy assumptions of both theorems. 
 Denote 
\[
\cE_{\ref{cor:inv_cons}}(M, \rho):=\set{M\in \cM_{n, d}: 
\exists x\in \mbox{Cons}(a_3, \rho)\,\, \mbox{ such that }\,\, \|(M-z\I_n)x\|_2\le \theta(n, d)\|x\|_2},
\]
and
\[
\cE_{\ref{cor:inv_cons}}(M^T, \rho):=\set{M\in \cM_{n, d}: \exists x\in \mbox{Cons}(a_3, \rho)\,\, \mbox{ such that } 
\,\, \|(M^{T}-z\I_n)x\|_2\le \theta(n, d)\|x\|_2}. 
\]
  The following is an immediate consequence of Theorems~\ref{thm: lower_bdd_T} and \ref{thm: invnonsteep}.
\begin{corollary}\label{cor:inv_cons}
     There exist  positive absolute constants $C_{\ref{cor:inv_cons}}, c_{\ref{cor:inv_cons}}>0$ such that the following holds. Let $n$ be large enough and $C_{\ref{cor:inv_cons}} \log n\le d\le n/2$. Let $0<\rho\le  \sqrt{n}/(B_\cT d^{3/4})$. Let $z\in \C$ be such that $|z|\le \sqrt{d}\, \log \log d$. Then
     \[
     \P\left (\cE_{\ref{cor:inv_cons}}(M, \rho)\bigcup \cE_{\ref{cor:inv_cons}}(M^T, \rho)\right )\le 23/n.
     \]
\end{corollary}

\subsection{{Proof of \cref{thm: invnonsteep}}}

Without loss of generality, we assume that $x\neq 0$. Fix a permutation $\sigma=\sigma_{x}$ of $[n]$ such that $|x_{\sigma{(i)}}|=x_{i}^{*}$. Since $x\notin \cT\cup \{0\}$, we have $x_{n_{3}}^*\neq 0$. 
Since $x\in \mbox{Cons}(a_3, \rho)$ there exists  
 $\lam_0\in \C$ such that the cardinality of
\[
J_0:=\set{i\in [n]: |x_i-\lam_0|\le \frac{\rho \|x\|_2}{\sqrt{n}}}
\]
is at least $n-n_{3}+1$. 
Therefore, there exists $k, \ell$ such that $k\le n_{3}<\ell$ and $\si(k), \si(\ell)\in J_0$. 
By \cref{lem: l2norm_linf}, $\|x\|_2\le B_\cT x_{n_{3}}^*,$
hence, using $\rho\le \sqrt{n}/(5B_\cT)$, 
$$
 \frac{\rho \|x\|_2}{\sqrt{n}}\le  \frac{ x_{n_{3}}^*}{5}. 
$$
Thus, 
\[
 x_{n_{3}}^*\le x_k^*\le |\lam_0|+\frac{\rho \|x\|_2}{\sqrt{n}}\le |\lam_0|+\frac{ x_{n_{3}}^*}{5}.
\]
Similarly, we have
\[
x_{n_{3}}^*\ge x_\ell^*\ge |\lam_0|-\frac{\rho \|x\|_2}{\sqrt{n}}\ge |\lam_0|-\frac{ x_{n_{3}}^*}{5}.
\]
Therefore
\begin{equation}\label{ineqforrho}
\frac{5}{6}|\lam_0|\le x_{n_{3}}^*\le \frac{5}{4}|\lam_0|\quad\quad \mbox{ and } \quad \quad
\frac{\rho \|x\|_2}{\sqrt{n}} \leq \frac{x_{n_{3}}^*}{5}   \leq \frac{1}{4}|\lam_0|.
\end{equation}

Next, we split $[n]\setminus J_0$ as follows. Set
\[J_1=\si([n_{1}])\setminus J_0,
 J_2=\si([n_{2}])\setminus (J_0 \cup J_1), \, J_3=\si([n_{3}])\setminus (J_0 \cup J_1\cup J_2), \, J_4=\si([n])\setminus \lr{\cup_{i=0}^3 J_i}.
\]
Then $|J_1|\le n_{1}, |J_2|\le n_{2}, |J_3|\le n_{3}$, $|J_4|\le n_{3}$, $[n]=\bigcup_{i=0}^4 J_i$. One has
\[
   \forall j\in J_2\, : \quad |x_j|\le x_{n_{1}}^*\le  d^{3/4}  x_{n_{2}}^*\le 4d^{3/4}x_{n_{3}}^* \leq 5 d^{3/4}|\lam_0|,
\]
\[
 \forall j\in J_3\, : \quad |x_j|\le x_{n_{2}}^*\le 4x_{n_{3}}^*\le 5|\lam_0|, 
\]
and
\[
 \forall j\in J_4\, : \quad |x_j|\le x_{n_{3}}^*\le \frac{5}{4}|\lam_0|. 
\]
Now, given a matrix $M\in \cM_{n, d}$, we define
\[
I_1= \set{i\in [n]: \text{Supp}(R_i(M^{T}))\cap J_1 \neq \emptyset}, \quad I_1'= \set{i\in [n]: \text{Supp}(R_i(M))\cap J_1 \neq \emptyset}
\]
and, denoting $n_4=n_3$, for $\ell=2, 3, 4$ define
\[
I_\ell=\set{i\in [n]: \abs{\sR_i(M^{T})\cap J_\ell}\ge \frac{9n_{\ell}d}{n}},\, \,\, \,  I_\ell'=\set{i\in [n]: 
\abs{\text{Supp}(R_i(M))\cap J_\ell}\ge \frac{81n_{\ell}d}{n}}.
\]

 Since columns of $M^T$ contain exactly $d$ ones, we first note that for $M^T$ one has  
 $$|I_1| \leq d n_1 \leq \frac{n}{9},$$ provided that $a_1$ is small enough. Then applying \cref{lemma: suppcol} to $J=J_2, J_3, J_4$, 
\[
\abs{\set{i \leq n: |\text{Supp}(R_{i}(M^T)) \cap J_\ell| \geq \frac{9 |J_\ell|d}{n} }} \leq \frac{n}{9}.
\]
Since $|J_\ell|\le n_{\ell}$ for $\ell=2, 3, 4$, we obtain $|I_\ell|\le n/9, \mbox{ for } 1\leq \ell\leq 4.$

\smallskip 

For $M$, applying \cref{prop: supprow}, we observe  that  with probability at least $1-4\exp(-(8n/81))$,
\[
\forall \ell \in \{1, 2, 3, 4\}  : \quad \abs{\set{i \leq n: |\text{Supp}(R_{i}(M))\cap J_\ell| \geq \frac{81 |J_\ell|d}{n} }} \leq \frac{n}{9}.
\]
 Note that $$\frac{81 |J_1|d}{n} \leq \frac{81 n_1d} {n} <1,$$ provided that $a_1$ is small enough.   Since  $|J_\ell|\le n_{\ell}$ for $\ell=2, 3, 4$, defining the event 
$$\cE:=\{\forall  \ell \in \{1, 2, 3, 4\}  \, :\quad |I_\ell'|\le n/9\},$$
 we obtain 
\[
\P(\cE)\ge 1-4\exp(-(8n/81)).
\]

Set $I:=[n]\setminus \lr{\cup_{i=1}^4 I_i \cup \si([n_{3}])}$ and $I':=[n]\setminus \lr{\cup_{i=1}^4 I'_i \cup \si([n_{3}])}$. Then, assuming $a_3\leq 0.001$, 
\[
|I|\ge n-\frac{4n}{9}-n_{3} \ge \frac{n}{2},
\]
and, assuming that $\cE$ occurs, 
\[
|I'|\ge n-\frac{4n}{9}-n_{3} \ge \frac{n}{2}.
\]
Note that
\[
 \forall i\in I \cup I': \qquad |x_i|\le x^*_{n_{3}}\le \frac{5}{4}|\lam_0|.
\]

In the case of $M^T$, for $ i\in I$,  denote 
$$
  J_{\ell}'=J_{\ell}'(i):= \text{Supp}( R_{i}(M^T)) \cap J_\ell, \quad \mbox{ for } \ell\in \{0, 1, 2, 3, 4\}
$$
and note that $J_{1}'=\emptyset$ since $i \notin I_1$. Then using the triangle inequality, we observe that for every $ i \in I$,
\[
|\langle R_{i}((M^{T}-z\I_n)), x \rangle| \ge |
\sum_{j \in J_{0}'}x_j|-\sum_{j \in J_{2}'}|x_j|-\sum_{j \in J_{3}'}|x_j| - \sum_{j \in J_{4}'}|x_j|-|zx_i|.
\]
We estimate each term separately.  First note that similarly to Lemma~\ref{lem: colsum} one observes that 
with probability at least $1-n\exp(-d/108)$ for every $i\leq n$ one has $|\text{Supp}( R_{i}(M^T))|\geq 5d/6$. 
Thus, for $i \notin I_1 \cup I_2 \cup I_3 \cup I_4$ (in particular, $J_{1}'=\emptyset$) one has
\[
|J_{0}'|\geq 5d/6-|J_2'|-|J_3'|-|J_4'| \geq 5d/6 - \frac{9n_{2}d}{n}-\frac{18n_{3}d}{n} \geq 5d/6- 27 a_3 d\geq  4d/5
\]
(recall,  $a_3 \leq 1/1000$). Therefore, using  the definition of $J_0$, we obtain
\[
\big|\sum_{j \in J_{0}'}x_j\big| \geq |\lambda_{0}||J_{0}'|-\sum_{j \in J_{0}'}|x_{j}-\lambda_{0}| \ge \frac{4d}{5}\lr{|\lambda_{0}|-\frac{\rho }{\sqrt{n}}\|x\|_2}.
\]
Meanwhile, one has
\begin{align*}
    \sum_{j \in J_{2}'}|x_j|+\sum_{j \in J_{3}'}|x_j| + \sum_{j \in J_{4}'}|x_j| 
    &\leq |J_{2}'|x_{n_{1}}^*+ |J_{3}'|x_{n_{2}}^* +|J_{4}'|x_{n_{3}}^*\\
    &\leq \left(\frac{45 n_{2}d^{7/4} }{n} + \frac{57 n_{3}d}{n}\right) |\lambda_0|\le \frac{ d\, |\lam_0|}{5},
\end{align*}
provided  $a_2 \leq 1/250, a_3\leq a_2/30$. Putting together the above estimates, and using that for $i\in I$, $i\notin \si([n_{3}])$, 
hence $|x_i|\le  x_{n_{3}}^*\le (5/4)|\lam_0|$, 
we obtain for $d\geq 7^4$
\begin{align*}
    |\langle R_{i}((M^{T}-z\I_n)), x \rangle| &\ge\frac{4d}{5}\lr{|\lambda_{0}|-\frac{\rho }{\sqrt{n}}\|x\|_2}-\frac{d}{5}|\lam_0|-\frac{5}{4}|\lam_0||z| \ge \frac{d}{5}|\lam_0|,
\end{align*}
where we used (\ref{ineqforrho}) and $|z|\le 4 d/25$. Since $|I|\geq n/2$, by (\ref{ineqforrho}), 
this implies that 
\[
\|(M^{T}-z\I_n)x\|_2 \geq \frac{|\lambda_{0}|d}{5} \sqrt{\frac{n}{2}} \geq \frac{4d\sqrt{n}}{25\sqrt{2}}  x_{n_{3}}^* \geq \frac{3d\sqrt{n}}{25B_\cT}\|x\|_2,
\]
where we again used $\|x\|_2\le B_\cT x_{n_{3}}^*$.

\smallskip 

In the case of $M$, on event $\cE$, for $k\in I'$, define $J_{\ell}{''}=J_{\ell}{''}(k):=\text{Supp} R_{k}(M)\cap J_\ell$ for 
$\ell\in \{0, 1, 2, 3, 4\}$. Repeating the above argument (note that $|\text{Supp} R_{k}(M)|=d$) yields, 
\[
|\langle R_{k}((M-z\I_n)), x \rangle|\ge \frac{d}{5}\, |\lam_0|.
\]
Hence, conditioned on $\cE$, \[
\|(M-z\I_n)x\|_2 \geq  \frac{3d\sqrt{n}}{25B_\cT}\, \|x\|_2.
\]
This completes the proof.

\subsection{Lower bounds for vectors from $\cT_1$}\label{sec:t1_step}

In this section, we provide lower bounds on $\|(M-z\I_n)x\|_2$ and $\|(M^{T}-z\I_n)x\|_2$ for vectors from $\cT_1$. 
As $\cT_1=\emptyset$ in the case $n_1=1$, the results of this section are relevant only in the case $n_1>1$. 
We essentially follow the proof in  \cite[Lemma 3.7]{litvak2019smallest} with necessary modifications. The primary difference lies in the analysis of the left--multiplication $\|\bar{x}(M-z\I_n)\|_2$, where fixed column sums of $M$ are not fixed as in the $d$-regular case. We apply a probabilistic argument involving \cref{lem: rowsum} to resolve this issue.

\begin{lemma}\label{lem:t1}
There exists an absolute constant $C>0$ such that the following holds. 
Let $n\geq 3$. 
    Let $C\log n\le d \le n/2$ and $z\in \C$ satisfy $|z|\le d$. Consider the events
    \[\cE_{\ref{lem:t1}}(M^T):=\set{\exists x\in \cT_1\,\,\, \mbox{ such that }\,\,\, \|(M^{T}-z\I_n)x\|_2\le 
    \omega (n, d) \|x\|_2},
    \]
    and
    \[
    \cE_{\ref{lem:t1}}(M):=\set{\exists x\in \cT_1\,\,\, \mbox{ such that }\,\,\, 
    \|(M-z\I_n)x\|_2\le \omega (n, d)\|x\|_2},
    \]
    where 
    $$
       \omega (n, d)= \min \left\{ \frac{\sqrt{d} }{4 \sqrt{n}} ,\,  
            \frac{5  \, \ell_0^{r/2}\, (d \log n)^{1/4}}{  (6d)^{r}}  \right\} .
    $$
    Then
    \[
      \P\lr{\cE_{\ref{lem:t1}}(M)\bigcup \cE_{\ref{lem:t1}}(M^T)}\le \frac{7e}{n}.
    \]
\end{lemma}

\begin{proof}
   We first estimate the probability of the event $\cE_{\ref{lem:t1}}(M^T)$.    Recall that the parameters $\ell_0, \varepsilon_0, n_1$ are defined  in the beginning of \cref{sec: inverb_almost_const} and the events $\Om_{k, \varepsilon_1}^{in}$ and $\Om_{k, \varepsilon_2}^{out}$ for $1\le k\le n,\, \varepsilon_1\in [\varepsilon_0,1/2], \, \varepsilon_2\in [\varepsilon_0,1)$ --- in \cref{sec:vectex_exp}. Without loss of generality assume $n_1>1$ 
   (otherwise, $\cT_1=\emptyset$). For each $0\le i\le r-1$, let
\begin{align*}
          \cE_{i}:=\Om_{\ell_0^i, \varepsilon_0}^{in}\cap\Om_{\ell_0^i, \varepsilon_0}^{out}, \,  \quad \cE_{r}:=\Om_{n_1, \varepsilon_0}^{in}\cap \Om_{n_1, \varepsilon_0}^{out}, \, \quad \mbox{ and } \quad 
         \cE:=\bigcap_{i=0}^{r}\cE_i.
    \end{align*}
We choose $\tau=1/2$ in \cref{lem: rowsum}, and set \[
\cE':=\cE\cap \cE_{\ref{lem: rowsum}}.
\]    
    By  \cref{Expansion1}, \cref{Expansion}, and since $\varepsilon_0<1/2$, for every $1\le k\le \varepsilon_0 n/d$ one has 
 \begin{align*}
\P\big((\Omega^{in}_{k,\varepsilon_0}\cap\Omega^{out}_{k,\varepsilon_0})^c\big)
&\le 2\binom{n}{k}\exp\Big(-\frac{\varepsilon_0^2}{24}kd\Big)
  +\binom{n}{k}\exp\Big(-\frac{\varepsilon_0^2}{16}kd\Big)\\
&\le 3\binom{n}{k}\exp\Big(-\frac{\varepsilon_0^2}{24}kd\Big)
\le 3\Big(\frac{en}{k}\Big)^k\exp\Big(-\frac{\varepsilon_0^2}{24}kd\Big).
\end{align*}
Using that $\varepsilon_0^2 d=48\log n$, we obtain
\[
\P\big((\Omega^{in}_{k,\varepsilon_0}\cap\Omega^{out}_{k,\varepsilon_0})^c\big)
\le 3\Big(\frac{en}{k}\Big)^k n^{-2k}= 3\Big(\frac{e}{kn}\Big)^k
\]
Using that $\ell_0\geq 2$ and  $n_1>\ell_0^r\geq 1$,  by the union bound, 
\begin{align*}
 \P(\cE^c)&\le \sum_{j=0}^{r}\P(\cE_j^c)
 \le 3\sum_{j=0}^{r-1}\Big(\frac{e}{n\ell_0^j}\Big)^{\ell_0^j}
  +3\Big(\frac{e}{nn_1}\Big)^{n_1}
  \le \frac{3e}{n}\Big(\sum_{j=0}^{r-1}\frac{1}{\ell_0^j}+\frac{1}{n_1}\Big)
  \le \frac{3e}{n}\sum_{j=0}^{r}\frac{1}{\ell_0^j}\le \frac{6e}{n}.
\end{align*}
Then, by \cref{lem: rowsum} applied with $m=n$, 
    \begin{equation*}
        \P((\cE')^c)\le \P(\cE^c)+\P(\cE_{\ref{lem: rowsum}}^c)\le  \frac{7e}{n}.
    \end{equation*}
Fix a realization $M=(M_{ij})\in \cE'$. For $0\le i\le r$, set
\[m_1= \ell_0^i, 
\qquad 
m_2= \begin{cases}
\ell_0^{i+1},\quad & \mbox{ if } 0\le  i\le r-1 \\
n_1\quad & \mbox{ if } i=r
\end{cases} 
\]
Recall that, given $x$, the permutation   $\sigma=\sigma_x$ denotes a permutation $[n]$ such that $x_i^*=|x_{\sigma(i)}|$ for $i\in [n]$. Define the index sets
\[
J^\ell=J^\ell(x)=\sigma([m_1]), \, \quad J^r=J^r(x)=\si([m_2]\setminus[m_1]), \, \quad J=J(x)=(J^\ell\cup J^r)^c.
\]
Fix $0\le i\le r$ and $x\in \cT_{1i}$. Then for all $j\in J^\ell$ we have
\[
\abs{x_{j}}\ge x_{m_1}^*\ge 6d x_{m_2}^*, \quad \mbox{ and } \quad  \max_{k\in J}|x_k|\le x_{m_2}^*\le x_{m_1}^*/(6d).
\]
Let $A$  be either $M$ or $M^T$. 
Recall that $I^\ell(A)=I^\ell(A, J^\ell, J^r)$ is the (random) set of rows of $A$  having exactly one $1$ in $J^\ell(x)$ and no $1$'s in $J^r(x)$. Next we apply  \cref{lem: omega_ke} with $q=\ell_0$, $m=m_1$ (with corresponding adjustment for $i=r$, where we have to take $qm=m_2$ and $q=m_2/m_1\leq \ell_0$ --- we may assume that it is integer) to obtain
\[
|I^\ell(M^T)|\ge (1-2\varepsilon_0\ell_0)d|J^\ell|\ge 3dm_1/5,
\]
and
\[
|I^\ell(M)|\ge (1-4\varepsilon_0\ell_0)d|J^\ell|\ge 3dm_1/5.
\]
Let $I_0^\ell(A):=I^\ell(A)\setminus(J^\ell\cup J^r)$. In both cases ($A=M$ or $A=M^T$), one has 
$$|I_0^\ell(A)|\ge 3dm_1/5-m_2\ge dm_1/2$$ 
provided that the absolute constant in $d\geq C \log n$ is large enough. 

 Now let  $B$ denote $(M-z\I_n)$ if $A=M$  or $(M^{T}-z\I_n)$ if $A=M^T$. 
 For any $s\in I_0^\ell(A)$   there exists  $j(s)\in J^\ell$ such that the $s$-th row of $A$ satisfies
\[
\mbox{Supp }R_s(A)\cap J^\ell=\{j(s)\} \qquad \mbox{ and } \qquad \mbox{Supp }R_s(A)\cap J^r=\emptyset.
\]
Then one has
\begin{align*}
\abs{\inn{R_s(B), x}} &= \abs{x_{j(s)} + \sum_{k\in J} A_{sk}x_k - zx_s}\ge \abs{x_{j(s)}} - \sum_{k\in J} A_{sk}|x_k| - |z||x_s|.
\end{align*}
If $A=M^T$, $R_s(A)$ is the $s$-th row of $M^T$ (i.e., $s$-th column of $M$). Since $M\in \cE_{\ref{lem: rowsum}}$, $$\sum_{k\in J}M^T_{sk} = \sum_{k\in J}M_{ks}\le \sum_{k=1}^nM_{ks}\le 3d/2.$$
If $A=M$, $R_s(A)$ is the $s$-th row of $M$, hence
$$\sum_{k\in J}M_{sk}\le \sum_{k=1}^nM_{sk}=d.$$ 
Using $|x_k|\le x_{m_2}^*$ for $k\in J$,  $x_s^*\le x_{m_2}^*$ (since $s\notin J^\ell\cup J^r$), $x_{m_1}^*\ge 6dx_{m_2}^*$, and $|z|\le d$, we obtain 
\begin{align*}
  \abs{\inn{R_s(B), x}}
 & \ge x_{m_1}^*-x_{m_2}^*\sum_{k\in J}A_{sk}-|z|x_{m_2}^*
\ge x_{m_1}^*-\frac{x_{m_1}^*}{6d}\cdot\frac{3d}{2}-\frac{x_{m_1}^*}{6}\ge \frac{1}{2} x_{m_1}^*.
\end{align*}
Since this holds for at least $\vert I_0^\ell(A)\vert \ge dm_1/2$ indices $s$, we have
\[\|Bx\|_2\ge \frac{1}{2}x^*_{m_1}\sqrt{\frac{dm_1}{2}}= \frac{\sqrt{d} \, \ell_0^{i/2}}{2\sqrt{2}}\, x^*_{m_1}.
\]
Applying \cref{lem: l2norm_linf} for $x\in \cT_{1i}$, we obtain that in both cases 
$B=(M-z\I_n)$ and $B=(M^{T}-z\I_n)$,
\[
   \|Bx\|_2\ge \min \left\{ \frac{\sqrt{d} \, \ell_0^{i/2}}{4 \sqrt{n}} ,\,  
   \frac{5 \sqrt{d} \, \ell_0^{i/2} \log^{1/4} n}{ d^{1/4} \, (6d)^{i}}  \right\} 
    \|x\|_2  \geq 
    \min \left\{ \frac{\sqrt{d} }{4 \sqrt{n}} ,\,  
   \frac{5  \, \ell_0^{r/2}\, (d \log n)^{1/4}}{  (6d)^{r}}  \right\} 
    \|x\|_2  .
\]
This completes the proof. 
\end{proof}

\subsection{Lower bounds for vectors from $\cT_2\cup (\cT_3\cap \mbox{Cons}(a_3, \rho))$} \label{individual}

In this section, we provide
lower bounds on $\|(M-z\I_n)x\|_2$ and $\|(M^{T}-z\I_n)x\|_2$ for vectors from 
$\cT_2\cup (\cT_3\cap \mbox{Cons}(a_3, \rho))$. We first estimate 
the cardinality of certain nets --- discretizations of the sets of steep vectors $\cT_2$ and $\cT_3$ introduced earlier. Then we obtain individual probability bounds for a fixed vector in a net. Finally we conclude with the standard $\ep$-net argument.

\subsubsection{Nets}\label{sec:nets}
Recall that $\1=(1,1, \ldots, 1)$,  $\e$ denotes the vector $\1/\sqrt{n}$, $P_{\e}$ denotes the projection on span of  $\e$, and $P_{\e^{\perp}}$ denotes the projection on $\e^{\perp}$. The triple norm on $\C^n$ is defined by 
\[ \tnrm{x}=\sqrt{\|P_{\e^{\perp}} x\|_{2}^{2} + d \|P_{\e} x\|_{2}^{2}}.
\]
 For convenience, we will consider the following normalization: 
\[
\cT_2':=\left \{x\in \cT_2: x_{n_{1}}^*=1\right\}=\left\{\frac{x}{x_{n_{1}}^*}: x\in \cT_2\right\}
\quad \mbox{ and } \quad 
\cT_3':=\left \{x\in \cT_3: x_{n_{2}}^*=1\right\}\cap \mbox{Cons}(a_3,\rho),
\]
where $0<\rho\le  \sqrt{n}/(B_\cT d^{3/4})$.

\begin{lemma}\label{lem: net_Tvectors}
    There exists an absolute constant $C>0$ such that the following holds. Let $n\ge 1$ and let $C\log n\le d\le n/2$. Let $0<\rho\le  \sqrt{n}/(B_\cT d^{3/4})$. Let $i\in \{2, 3\}$. There exists a set $\cN_i=\cN_i'+\cN_i''$, $\cN_i''\subset \mbox{span} \{\1\}$, with the following properties:
    \begin{enumerate}
        \item $|\cN_i|\le \exp(6n_{i}\log d)$.
        \item For every $u\in \cN_i'$ one has $u_j^*\le 1/4+d^{-3/4}$ for all $j\ge n_{i}$.
        \item For every $x\in \cT_i'$ there is $u\in \cN_i'$ and $w\in \cN_i''$ such that 
       \[\|x-u\|_\infty \le \frac{1}{ d^{3/4}}, \quad \quad \|w\|_\infty \le \frac{1}{ d^{3/4}}, 
        \quad\quad \mbox{ and }  \quad \quad \tnrm{x-u-w} \le \frac{\sqrt{2n}}{ d^{3/4}}.
        \]
    \end{enumerate}
\end{lemma}
As the proof of this lemma repeats the proofs from \cite[Lemma 3.8]{litvak2019smallest} and \cite[Lemma 6.7]{litvak2022singularity}, we  omit it.

\subsubsection{Individual probability}

We turn now to the individual probability bounds.
Following the proof in \cite{litvak2019structure, litvak2022singularity}, we consider any $n\times n$ complex matrix $W$ instead of the shift $z\I_n$.
Recall that $\cM_{n, d}$ denotes the class of  all $n\times n$ matrices with entries in $\{0,1\}$ for which each row sums to $d$.  Recall that the probability measure $\P$ on $\cM_{n,d}$ is the uniform probability measure, in particular rows of a random matrix from $\cM_{n,d}$ are independent.  For a fixed vector $x$ from the nets constructed earlier, we need to  bound from below both $\|(M+W)x\|_2$ and $\|(M^{T}+W)x\|_2$ for $M\in \cM_{n,d}$.
As $M$ has independent rows, whereas $M^T$ does not, we  will treat $A=M$ and $A=M^T$ separately. To obtain a lower bound, it suffices to obtain an anti-concentration bound for the inner products $\inn{R_{i}(A+W), x^\dagger}$ for indices $i\in [n]$, where $A$ stands  either for $M$ or for $M^T$, given a vector $z\in \C^n$, 
$$x^\dagger=\bar{x}=(\bar{x}_i)_{i=1}^n$$
for  $x\in \C^n$ with $\bar{x}_i$ is the complex conjugate of $x_i\in \C$. 
To estimate this inner product, we will apply Esseen inequalities (see \cref{rogozin1961increase}) for sums of independent random variables. To create independent random variables, we will use \cref{lem: omega_ke}  for  $2m$ columns of $A$ corresponding to $m$ largest and $m$ smallest (in modulus) coordinates of $x$, where in our application $m=n_1$ or $m=n_2$. The combinatorial input \cref{lem: omega_ke} applies only to the sets of columns of size at most $ \varepsilon_0 n/d$. This restriction is satisfied when $m=n_1$, but fails for $m=n_2$. To overcome this problem, following the idea from \cite[Section 4.3]{litvak2019structure}, we split the selected $2m$ columns into disjoint blocks of size at most $\varepsilon_0 n/d$ each. Conditioning on this splitting, we obtain independent random variables 
such that their sum is essentially $\inn{R_{i}(A+W), x^\dagger}$. This allows us to apply \cref{rogozin1961increase} for sums of independent random variables.

Recall that given $J\subset [n]$ and an $n\times n$ matrix $A$, by $I(J, A)$ we denote 
the set of row indices where the row has exactly one $1$ in the columns indexed by $J$, that is,
\[
I(J, A)=\set{i\in [n]: \abs{\mbox{Supp} (R_i(A)) \cap J}=1}.
\]
%
%
Given a subset $J\subset [n]$, we define two families of $n\times |J|$ matrices,
\[
\cM^\mathrm{col}_J:=\set{V\in \R^{n\times |J|}: \exists M\in \cM_{n, d} \mbox{ such that } V=M_{[:, J]}}
\]
and
\[
\cM^\mathrm{row}_J:=\set{V\in \R^{n\times |J|}: \exists M\in \cM_{n, d} \mbox{ such that } V=(M^T)_{[:, J]}},
\]
where for a matrix $A$, $A_{[:,J]}$ denotes the submatrix consisting of the columns of $A$ with indices in $J$.

Fix a positive integer $1\le q_0\le n$ and a disjoint partition $J_0, J_1, \dots, J_{q_0}$ of $[n]$. Given subsets $I_1, I_2, \dots, I_{q_0}$ of $[n]$, set $\cI=(I_1, I_2, \dots, I_{q_0})$ and consider the following two classes.
For $V\in \cM^\mathrm{col}_{J_0}$, denote
\[
\mathcal{F}(\mathcal{I}, V) := \left\{ M \in \mathcal{M}_{n,d}:  M_{[:,J_0]} = V \text{ and } I(J_k, M)=I_k , \forall k \in [q_0]
\right\}
\]
and for $V\in \cM^\mathrm{row}_{J_0}$, denote
\[
\mathcal{G}(\mathcal{I}, V) := \left\{ M \in \mathcal{M}_{n,d}:  (M^T)_{[:,J_0]} = V \text{ and } I(J_k, M^T)=I_k , \forall k \in [q_0]
\right\}.
\]
Note that these two classes depend on the choice of $\cI$ and they can be empty. In words, we start with a matrix $A$ such that either $A=M$ or $A=M^T$, for $M \in \cM_{n, d}$ and proceed  as follows. First, we fix the columns of $A$ indexed by $J_0$ to be $V$. Then for each $k\in [q_0]$, we also fix the row indices of $A$  that have exactly one $1$ in columns of $A$ indexed by $J_k$. 
 
In both cases, for any fixed partition $J_0, J_1, \dots, J_{q_0}$ of $[n]$, the families $\{\mathcal{F}(\mathcal{I}, V)\}_{\cI, V}$ and $\{\mathcal{G}(\mathcal{I}, V)\}_{\cI, V}$ form disjoint decompositions of $\mathcal{M}_{n,d}$:
 \[
 \mathcal{M}_{n,d}=\bigcup_{V\in \cM^\mathrm{col}_{J_0}}\ \bigcup_{\mathcal{I}\in (\cP([n]))^{q_0}}\mathcal{F}(\mathcal{I}, V) \qquad\text{and}\qquad \mathcal{M}_{n,d}=\bigcup_{V\in \cM^\mathrm{row}_{J_0}}\ \bigcup_{\mathcal{I}\in (\cP([n]))^{q_0}}\mathcal{G}(\mathcal{I}, V),
 \]
where $\cP([n])$ is the power set of $[n]$.

Moreover, given $\cI$ and $V$, we further split each class $\cF(\cI, V)$ (and similarly $\mathcal{G}(\mathcal{I}, V)$)  into smaller equivalence classes using an equivalence relation defined below. 

Equivalence for $\cF(\cI,V)$: 
Fix a row index $i_0\in [n]$ and a subset $K\subset [q_0]$. Let
\[
K_0:=\{0\}\cup([q_0]\setminus K),  \qquad J^*=J^*(K):=\bigcup_{q\in K_0}J_q=J_0\cup \bigcup_{q\in [q_0]\setminus
     K} J_q.
\]
We say that two matrices $M=(M_{ij}), M'=(M_{ij}')\in \cF(\cI, V)$ are equivalent if
\begin{enumerate}[label=(\roman*)]
\item For all $1\le s\le i_0-1$ and for all $j\in [n]$, $M_{sj}=M'_{sj}$.
\item  For all $s\in [n]$ and all $j\in J^*$, $M_{sj}=M'_{sj}$.
    \item For every row $s\in [n]$ and for every $ q\in  K$, the number of $1$'s in columns indexed by $J_q$ is the same:
    \[
    \sum_{j\in J_q}M_{sj}= \sum_{j\in J_q}M_{sj}'.
    \]
\end{enumerate}
Denote by $\cH:=\cH(\cF(\cI, V),i_0, K)$ the set of all equivalence classes corresponding to this relation. Note that 
\[
\cF(\cI, V)=\bigcup_{H\in\cH^M}H.
\]

Equivalence for $\cG(\cI,V)$: 
Fix a row index $i_0\in [n]$ and a subset $K\subset [q_0]$, and define $K_0$ and $J^*(K)$ as above.  For $M, M'\in \cG(\cI, V)$, we say that two matrices $M, M'$ are equivalent if their transposes  $M^T$ and $(M')^T$ satisfy the same three conditions above with $M^T, (M')^T$ in place of $M, M'$. Denote by $\cH^{M^T}:=\cH(\cG(\cI, V),i_0, K)$ the set of all equivalence classes corresponding to this relation, then
     \[
     \cG(\cI, V)=\bigcup_{H\in\cH^{M^T}}H.
     \]

Note that for matrices in a given class $H\in\cH^M$ (or $H\in\cH^{M^T}$), all columns in $J^*(K)$ and the first $i_0-1$ rows are fixed, thus the randomness comes only from the positions on the $1$'s inside the blocks $[n]\times J_q, q\in K$, subject to a prescribed sum in each row. Thus, in a sense, these blocks are independent of each other on $H$. 

Finally, we fix a vector $x\in \C^n$, a fixed constant vector $y=y_1\mathbf{1}\in \mbox{span}\{\1\}, y_1\in \C$, and a deterministic matrix $W=(w_{ij})$. Let $A=(A_{ij})$ denote either $M\in \cM_{n, d} $ (when conditioning on $\cF(\cI, V))$ or $M^T$ for $M\in \cM_{n, d} $ (when conditioning on $\cG(\cI, V)$).  Given a class $H\in \cH^M$ or $H\in \cH^{M^T}$ (in particular, $V, \cI, i_0, K$ are fixed), and for each $k\in K$, define the random variable
$\xi_k$ on $H$  by 
\[\xi_k=\xi_k(A):=\sum_{j\in J_k}A_{i_0j}x_j.
\]
In words, $\xi_k$ represents the inner product of $x$ with the restriction of the $i_0$-th row of $A$ to $J_k$. Later we will use this construction in the case $i_0\in I_k$ for all $k\in K$, so that $\abs{\mbox{Supp} (R_{i_0}(A)) \cap J_k}=1$ and each $\xi_k$ selects exactly one coordinate $x_j$ from the block $[n]\times J_k$. As we have already mentioned, by construction, for matrices $A\in H$, the rows $s<i_0$ are fixed, the entries in columns $J^* $ are fixed, and each row's sum over $J_k$ for $k\in K$ is prescribed. Thus,  the blocks $[n]\times J_k$ for $k\in K$ are independent on  $H$, and the random variables $\xi_k$, $k\in K$, are independent.

Define also
  \[
     \xi'=\xi'(A):=y_1\sum_{k\in K}\sum_{j\in J_k}A_{i_0j}
+\sum_{k\in K_0}\sum_{j\in J_k} A_{i_0j}(x_j+y_j)+\sum_{j=1}^n w_{i_0j}(x_j+y_j).
     \]
By definition of the equivalence classes, $\xi'$ is constant on $H$. Therefore,
 \[
     \inn{R_{i_0}(A+W), (x+y)^\dagger}=\sum_{k\in K} \xi_k+\xi'.
     \]
 Suppose there exists $\alpha>0$ such that $\cL(\xi_k, 1/3)\le \alpha$ for every $k\in K$. Then, applying \cref{rogozin1961increase} with $t_0=1/3$, we obtain
\begin{equation}\label{eq: innprds}
    \P\lr{\abs{\inn{R_{i_0}(A+W), (x+y)^\dagger}}\le 1/3}\le \frac{C_{\ref{rogozin1961increase}}}{\sqrt{(1-\alpha)|K|}}.
\end{equation}

Recall that $\Om_{k, \varepsilon}:=\Om^{in}_{k, \varepsilon}\cap \Om^{out}_{k, \varepsilon}$ was defined 
in (\ref{generalomega}) and the event  $\cE_{\ref{lem: rowsum}}(\tau)$ in Lemma~$\ref{lem: colsum}$. 

\begin{lemma}\label{lem: individ1}
    There exist absolute constants $C_1, C_2, C_3, c_{\ref{lem: individ1}}>0$ such that the following holds. Let $C_3 \log n \le d\le n/2$, and let $\varepsilon\in [\varepsilon_0, 0.01]$. Set $m_0=\max\{1, \lfloor  \varepsilon n/2d\rfloor\}$, and let $m_1, m_2$ be such that
    \[
    1\le m_1\le m_2\le n-m_1.
    \]
    Assume that $x\in \C^n$ satisfies
    \[
    \forall i>m_2:\, \, \,      x_{m_1}^*>2/3+x_i^*.\]
     Let $y\in \mbox{span}\{\1\}$ and $z\in \C$. Let $m=\min\{m_0, m_1\}$ and $\cE_{\ref{lem: rowsum}}=\cE_{\ref{lem: rowsum}}(\ep)$. Consider the event
     \[\cE_{\ref{lem: individ1}}'(x,y,A,k)=
     \lr{\cE_{\ref{lem: individ1}}(x,y,M)\cap \Om_{2k,\varepsilon} \cap \cE_{\ref{lem: rowsum}}} \cup \lr{\cE_{\ref{lem: individ1}}(x,y,M^T)\cap \Om_{2k,\varepsilon} \cap \cE_{\ref{lem: rowsum}}},
     \]
     where $k=m_0$ or $k=m_1$ and for $A=M$ or $A=M^T$,
    \[
     \cE_{\ref{lem: individ1}}(x,y,A):=\set{M\in \cM_{n, d}:\, \|(A-z\I_n)(x+y)\|_2\le \sqrt{c_{\ref{lem: individ1}}md}}.
    \]
       Then
    \begin{enumerate}
        \item in the case $m_1\le m_0$, one has $\quad \P\lr{\cE_{\ref{lem: individ1}}'(x,y,A, m_1)}\leq (5/6)^{m_1d/2};$
        \item in the case $m_1>C_1 m_0$ and $\varepsilon=0.01$, one has 
        $\quad \P\lr{ \cE_{\ref{lem: individ1}}'(x,y,A, m_0)}\leq  \lr{\frac{C_2 n}{m_1d}}^{m_0d/4}$.
    \end{enumerate}
\end{lemma}

\begin{remark}\label{rem:lem611}
    We apply this lemma for $\cT_i$ with the following choice of parameters.  For $i=2$, we set $m_1=n_1<m_0$, $m_2=n_2$, $\varepsilon=0.01$, obtaining 
    \[
     \P\lr{\lr{\cE_{\ref{lem: individ1}}(x,y,M)\cap \Om_{2n_1,0.01} \cap \cE_{\ref{lem: rowsum}}} \bigcup \lr{\cE_{\ref{lem: individ1}}(x,y,M^T)\cap \Om_{2n_1,0.01} \cap \cE_{\ref{lem: rowsum}}}}\le (5/6)^{n_1d/2}.
    \]
    For $i=3$, we set $m_1=n_2>m_0$, $m_2=n_3$, $\varepsilon=0.01$, obtaining
    \begin{align*}
    \P\Big(  (\cE_{\ref{lem: individ1}}(x,y,M)&\cap \Omega_{2m_0, 0.01}\cap \cE_{\ref{lem: rowsum}} 
     \bigcup 
    \left(\cE_{\ref{lem: individ1}}(x,y,M^T)\cap \Omega_{2m_0, 0.01} \cap \cE_{\ref{lem: rowsum}} \right)\Big)
    \\& \le \lr{\frac{C_2 n}{n_2d}}^{0.01n/8}
     \le  \lr{\frac{a_2d^{1/5}}{C_2}}^{-n/800}.
    \end{align*}
\end{remark}

Recall that given two disjoint subsets $J^\ell, J^r \subset [n]$ and a matrix $M\in \cM_{n, d}$,
\[
I^\ell=I^\ell(J^\ell, J^r, M)=\set{i\in [n]: \abs{\mbox{Supp}(R_i)\cap J^\ell}=1 \mbox{ and } \mbox{Supp}(R_i)\cap J^r=\emptyset}
\]
and
\[I^r=I^r(J^r, J^\ell,  M)=\set{i\in [n]: \abs{\mbox{Supp}(R_i)\cap J^r}=1 \mbox{ and } \mbox{Supp}(R_i)\cap J^\ell=\emptyset}.\]

\begin{proof}
    Fix $x\in \C^n $ and $y\in \text{span}\{\1\}$ satisfying the condition of the lemma.  Let $\sigma=\sigma_x$ be defined as above. Denote $q_0=m_1/m$ and without loss of generality, assume that either $q_0=1$ or that $q_0$ is a large enough integer.  Let $J_1^\ell, \dots, J_{q_0}^\ell$ be a partition of $\sigma{[m_1]}$ into sets of cardinality $m$ each, and let $J_1^r, \dots, J_{q_0}^r$ be a partition of $\sigma([n-m_1+1, n])$ into sets of cardinality $m$ each. Define for $q\in [q_0]$
    \[
    J_{q}:=J_q^\ell\cup J_q^r    \quad \mbox{ and } \quad  J_0:=[n]\setminus \bigcup_{q=1}^{q_0} J_q.
    \]
    Then $J_0, J_1, \dots, J_{q_0}$ forms a disjoint partition of $[n]$.

We deal simultaneously with the cases of $M$ and $M^T$, so let $A$ denote either $M$ or $M^T$. All definitions and notations involving $A$ apply to both cases. Let  
$$
  I_q^\ell(A)=I_q^\ell(J_q^\ell, J_q^r, A)\quad \mbox{ and }\quad  I_q^r(A)=I_q^r(J_q^r, J_q^\ell, A)
$$ 
be defined as above, and for every $1\le q \le q_0$ denote 
$$ 
 I_q(A)=I_q(J_q, A):=I_q^\ell(A)\cup I_q^r(A). 
$$ 
 Note that by definition, $I_q^\ell(A)\cap I_q^r(A) =\emptyset$ for each $1 \le q \le q_0$. 
 Since the cardinality of $J_q$ is  $2m\leq 2m_0\le \varepsilon n/d$, for 
 $M \in  \cE_{\ref{lem: rowsum}} \cap \Om_{2m, \varepsilon}$,  by \cref{lem: omega_ke}, we observe
    \begin{equation*}\label{eq: card_I_q_lemm}
            |I_ q^\ell(M)|,\, |I_ q^r(M)|\in [(1-8\varepsilon)dm, (1+\varepsilon)dm], 
    \end{equation*}
    and
    \begin{equation*}\label{eq: card_I_q_lemm1}
            |I_ q^\ell(M^T)|,\, |I_ q^r(M^T)|\in [(1-4\varepsilon)dm, dm].
    \end{equation*}
Therefore, as we deal with disjoint unions, 
    \begin{equation}\label{eq:Iqs}
        |I_q(M)|\in [2(1-8\varepsilon)dm, 2(1+\varepsilon)dm],
    \end{equation}
    and
    \begin{equation}\label{eq:Iqs1}
        |I_q(M^T)|\in [2(1-4\varepsilon)dm, 2dm].
    \end{equation}
    
Split $\cM_{n,d}$ into disjoint union of classes $\cF(\cI, V)$  defined at the beginning of this section with the fixed disjoint partition $J_0, J_1, \dots, J_{q_0}$ chosen earlier. 
Let 
$$
 \cF:=\{\cF(\cI, V): \Om_{2m,\varepsilon}\cap  \cE_{\ref{lem: rowsum}} \cap \cF(\cI, V)\neq \emptyset\}.
 $$
 Since $\cE_{\ref{lem: rowsum}} \cap \Om_{2m, \varepsilon}\cap \cF(\cI, V) \neq \emptyset$ implies that $|I_q|$ satisfies \eqref{eq:Iqs} for each $1 \le q \le q_0$, we have 
\begin{align*}
 \P( \cE_{\ref{lem: individ1}}(x,y,M) &\cap \Om_{2m,\varepsilon} \cap  \cE_{\ref{lem: rowsum}}   )\\
 &=\sum_{\cF(\cI, V) \in \cF} \P\lr{\cE_{\ref{lem: individ1}}(x,y,M)\cap \Om_{2m,\varepsilon} \cap  \cE_{\ref{lem: rowsum}} \,\, |\,\, \cF(\cI, V)} \P\lr{\cF(\cI, V)}\\
&\le \max_{\cF(\cI, V) }\P\lr{\cE_{\ref{lem: individ1}}(x,y,M)\cap \Om_{2m,\varepsilon} \cap  \cE_{\ref{lem: rowsum}} \,\, |\,\, \cF(\cI, V)}\\
&\le  \max_{\cF(\cI, V) }\P\lr{\cE_{\ref{lem: individ1}}(x,y,M)\,\,|\,\,\cF(\cI, V)},
\end{align*}
where the maximum is taking over $\cF(\cI, V)\in \cF$ with $I_q$'s satisfying \eqref{eq:Iqs} for $1 \le q \le q_0$.

\smallskip 

Similarly, we may also split $\cM_{n,d}$ into the disjoint union of classes $\cG(\cI, V)$  defined at the beginning of this section with the fixed disjoint partition $J_0, J_1, \dots, J_{q_0}$, and deduce that
\begin{equation} \label{eq: CPcG}
\P\lr{\cE_{\ref{lem: individ1}}(x,y,M^{T})\cap \Om_{2m,\varepsilon} \cap  \cE_{\ref{lem: rowsum}}} \leq  \max_{\cG(\cI, V) }\P\lr{\cE_{\ref{lem: individ1}}(x,y,M^{T})\,\,|\,\,\cG(\cI, V)},
\end{equation}
where the maximum is taking over $\cG(\cI, V)$ with $I_q$'s satisfying \eqref{eq:Iqs1} for $1 \le q \le q_0$.

Thus, to prove our lemma, it suffices to prove a uniform upper bound for the conditional probabilities. Since we partition $\cM_{n,d}$ using two different types of equivalent classes, to avoid potential confusion, we will treat those conditional probabilities separately, although the proofs are very similar. We start by fixing an admissible class $\cF(\cI_1, V_1) $
such that for $1 \leq q \leq q_0$, $I_q \in \cI_1$ satisfies \eqref{eq:Iqs}.  Denote the corresponding induced probability measure: 
\[
\P_\cF(\cdot):=\P(\cdot\,\,|\,\, \cF(\cI_1, V_1))
\]
Let \[
I:=\bigcup_{q=1}^{q_0} I_q.
\]
We first show that the set of $i$'s belonging to many $I_q$'s is large. More precisely, given $i\in [n]$, denote
\[
B_i:=\set{q\in [q_0]: i\in I_q}  \qquad \mbox{ and } \qquad I_0:=\set{i\in [n]: \abs{B_i}\ge \frac{3mdq_0}{5n}}.
\]
By \eqref{eq:Iqs} we have
\[
2(1-8\varepsilon)dmq_0\le \sum_{q=1}^{q_0}|I_q|=\sum_{i=1}^n |B_i| = 
\sum_{i\in I_0} |B_i|+\sum_{i\not\in I_0} |B_i|\le |I_0| q_0 + n \cdot\frac{3mdq_0}{5n}.
\]
Thus, one has $|I_0|\ge 2md/3$.
Without loss of generality, assume that $I_0=[I_0]=\{1,\dots, |I_0|\}$ and we only consider the first $k:=\lceil 2md/3\rceil$ indices from $I_0$. 
For every $M\in \cE_{\ref{lem: individ1}}(x,y,M)$ we have
\[
\|(M-z\I_n)(x+y)\|_2^2=\sum_{i=1}^n\abs{\inn{R_i(M-z\I_n), (x+y)^\dagger}}^2\le c_{\ref{lem: individ1}}md.
\]
Thus, there are at most $9c_{\ref{lem: individ1}}md$ rows with $\abs{\inn{R_i(M-z\I_n), (x+y)^\dagger}}\ge 1/3$.  Hence,
\[
\abs{\set{i\in [k]: \abs{\inn{R_i(M-z\I_n), (x+y)^\dagger}}< \frac{1}{3}}}\ge  \lr{\frac{2}{3}-9c_{\ref{lem: individ1}}}md.
\]
Let $k_0=\lceil \lr{\frac{2}{3}-9c_{\ref{lem: individ1}}}md\rceil$ ($k_0\geq md/2$ for small enough 
$c_{\ref{lem: individ1}}$).  For every $i\le k$, we denote
\begin{equation}\label{eq: event_subclasses_FIV}
\Om_i:=\set{ M \in \cF(\cI_1, V_1): \abs{\inn{R_i(M-z\I_n), (x+y)^\dagger}}<1/3}  \quad  \mbox{ and } \quad  \Om_0:=\cF(\cI_1, V_1).
\end{equation}
Then 
\begin{align*}
    \P_\cF(\cE_{\ref{lem: individ1}}(x,y,M))\le \sum_{\substack{B\subset [k]\\ |B|=k_0}}\P_\cF\lr{\bigcap_{i\in B}\Om_i}\le \binom{k}{k_0}\max_{\substack{B\subset [k]\\ |B|=k_0}}\P_\cF\lr{\bigcap_{i\in B}\Om_i}.
\end{align*}
Without loss of generality, we assume that the maximum above is attained at $B=[k_0]$. Then 
\begin{align*}
     \P_\cF(\cE_{\ref{lem: individ1}}(x,y,M))&\le \lr{\frac{ek}{k-k_0}}^{k-k_0}\P_\cF\lr{\bigcap_{i\in[k_0]}\Om_i}\le \lr{\frac{1}{c_{\ref{lem: individ1}}}}^{9c_{\ref{lem: individ1}}md}\P_\cF\lr{\bigcap_{i\in[k_0]}\Om_i}\\
     &=\lr{\frac{1}{c_{\ref{lem: individ1}}}}^{9c_{\ref{lem: individ1}}md}\prod_{i=1}^{k_0}\P_\cF\lr{\Om_i|\Om_1\cap \cdots \cap \Om_{i-1}}.
\end{align*}

Now we estimate the factors in the product. Fix $i\in [k_0]$ and $B_i$.  As $i\in I_0$, 
$|B_i|\ge\frac{3mdq_0}{5n}$. 
Consider the partition of $\cF(\cI_1,V_1)$ 
into equivalent classes $H\in \cH(\cF(\cI_1,V_1),i, B_i)$ (such classes were defined at the beginning of this section). Let $\P_H$ denote the conditional probability measure $\P(\cdot \vert H)$  on a class $H $.   Since all matrices in $H$ have their first $i-1$ rows fixed, for every $H$, 
the intersection $H_i:=H\cap \Om_1\cap \cdots \cap \Om_{i-1}$  is either $H$ or $\emptyset$. Thus,
\[
\P_\cF\lr{\Om_i|\Om_1\cap \cdots \cap \Om_{i-1}} \le \max_{H: H_i\neq \emptyset} \P_H(\Om_i).
\]
To bound the probability from above, we define a random variable $\xi_q(M, i), q\in B_i$ by
\[
\xi_q(M,i):=\sum_{j\in J_q}M_{ij}x_j.
\]
Fix $H$ such that $H_i\neq \emptyset$. For any matrix in the equivalent class $H$, by the definition, every block $[n]\times J_q$ has a prescribed sum in each row. Recall that 
$$
  J^*(B_i):=J_0\cup \bigcup_{k\in [q_0]\setminus B_i} J_k.
$$
Thus, conditional on $H$, random variables $\xi_q(M,i)$'s are independent for $q\in B_i$. 
Defining 
$$\alpha:=\max_{q\in B_i}\cL_H(\xi_q(M,i), 1/3),$$
by  \eqref{eq: innprds} we obtain
\begin{align*}
    \P_H(\Om_i)&=\P_H\lr{ \abs{ \left \langle R_i(M-z\I_n), (x+y)^\dagger\right \rangle}\le\frac{1}{3}}
    \le \frac{C_{\ref{rogozin1961increase}}  }{\sqrt{(1-\alpha)|B_i|}}\le \frac{C_{\ref{rogozin1961increase}}  }{\sqrt{(1-\alpha)3mdq_0/(5n)}}.
\end{align*}
Moreover, in the case $q_0=1$ we simply have
\[
\P_H(\Om_i)=\alpha=\cL_H(\xi_1(M,i), 1/3) .
\]

Now it remains to bound $\alpha$. Fix $q\in B_i$ such that $i\in I_q$. 
Since the intersection of the support of $R_i(M)$ with $J_q$ is a singleton, say indexed by 
$$j_r(q)=j_r(q, M, i) \in J_q^r\quad\quad \mbox{ or } \quad \quad j_{\ell}(q)=j_{\ell}(q, M, i) \in J_q^{\ell},$$
we have 
\[
|\xi_q(M,i)|=\begin{cases}
 |x_{j_r(q)}|,\\
|x_{j_\ell(q)}|.
\end{cases} 
\]
Since each row of $M$ is independently and uniformly sampled from the vectors in $\{0,1\}^n$ containing $d$ ones, we observe that conditioned on $H$, the single 1 appears in either $J_q^r$ or $J_q^{\ell}$ with equal probability, i.e., $\P_{H}\lr{j_r(q) \in J_q^r}=\P_{H}\lr{j_{\ell}(q) \in J_q^{\ell}}=1/2$.  Moreover, since $\abs{x_{j_\ell(q)}}-\abs{x_{j_r(q)}}>2/3$, we have $\alpha\le 1/2$. 
%
Therefore, we have 
in the case $q_0=1$,
\begin{equation}\label{eq: CPcF1}
   \P_\cF(\cE_{\ref{lem: individ1}}(x,y,M))\le \lr{\frac{1}{c_{\ref{lem: individ1}}}}^{9c_{\ref{lem: individ1}}md}\lr{\frac{1}{2}}^{(2/3-9c_{\ref{lem: individ1}})md}.
\end{equation}
In the case $q_0=m_1/m>C_1$ we obtain,
\begin{equation}\label{eq: CPcF2}
\P_\cF(\cE_{\ref{lem: individ1}}(x,y,M)) \le \lr{\frac{1}{c_{\ref{lem: individ1}}}}^{9c_{\ref{lem: individ1}}md}\lr{\frac{C_{\ref{rogozin1961increase}}\sqrt{10n}}{\sqrt{3m_1d}}}^{(2/3-9c_{\ref{lem: individ1}})md}.
\end{equation}

Now we turn to bounding the conditional probability in \eqref{eq: CPcG}. Fix an admissible class $\cG(\cI_2, V_2) $
such that for $1 \leq q \leq q_0$, $I_q \in \cI_2$ satisfies \eqref{eq:Iqs1}.  Denote the corresponding induced probability measure by
\[
\P_\cG(\cdot):=\P(\cdot\,\,|\,\, \cG(\cI_2, V_2)).
\]
Define the same sets $I$, $B_i$, and $I_0$ as before. 
By \eqref{eq:Iqs1}, we can again deduce that $|I_0|\ge 2md/3$. Now let
\begin{equation}\label{eq: event_subclasses_GIV}
 \Om_i':=\set{ M \in \cG(\cI_2, V_2): \abs{\inn{R_i(M^T-z\I_n), (x+y)^\dagger}}<1/3}  \quad  \mbox{ and } \quad  \Om_0':=\cG(\cI_2, V_2).
\end{equation}
Repeating the above proof, we obtain
\begin{align*}
 \P_\cG(\cE_{\ref{lem: individ1}}(x,y,M^T)) &\le \lr{\frac{1}{c_{\ref{lem: individ1}}}}^{9c_{\ref{lem: individ1}}md}\prod_{i=1}^{k_0}\P_\cG\lr{\Om'_i|\Om'_1\cap \cdots \cap \Om'_{i-1}} \\
 &\le \lr{\frac{1}{c_{\ref{lem: individ1}}}}^{9c_{\ref{lem: individ1}}md} \left(\max_{H: H_i\neq \emptyset} \P_H(\Om'_i) \right)^{k_0},
\end{align*}
where $H \in \cH(\cG(\cI_2,V_2),i, B_i)$ and $H_i:=H\cap \Om'_1\cap \cdots \cap \Om'_{i-1}$. To estimate the maximum probability, we again fix an $H$ such that $H_i \neq \emptyset$ and define 
\[
\xi_q(M^T,i):=\sum_{j\in J_q}M^T_{ij}x_j.
\]
Repeating the above proof, we have
\[
\P_{H}(\Om_i')=\P_{H}\lr{ \abs{ \left \langle R_i(M^T-z\I_n), (x+y)^\dagger\right \rangle}\le\frac{1}{3}}\\
    \le \frac{C_{\ref{rogozin1961increase}}  }{\sqrt{(1-\alpha')3mdq_0/(5n)}},
 \]
where $\alpha':=\max_{q\in B_i}\cL_{H}(\xi_q(M^{T},i), 1/3)$. And in the case $q_0=1$,
\[
\P_H(\Om'_i)=\alpha'=\cL_H(\xi_1(M^T,i), 1/3) .
\]
To  bound from above $\alpha'$, we follow our previous proof by indexing the single 1 by 
$$
  j_r(q)=j_r(q, M^T, i) \in J_q^r \quad\quad \mbox{ or } \quad \quad 
  j_{\ell}(q)=j_{\ell}(q, M^T, i) \in J_q^{\ell},
$$
 and defining 
\[
|\xi_q(M^T,i)|=\begin{cases}
 |x_{j_r(q)}|,\\
|x_{j_\ell(q)}|.
\end{cases} 
\]
Since $\abs{x_{j_\ell(q)}}-\abs{x_{j_r(q)}}>2/3$, we have
\[
 \cL_{H}(\xi_q(M^T, i), 1/3)\le \max(\P_{H}(H_l), \P_{H}(H_r))
 \]
 where $H_\ell:=\{M \in H: j_{\ell}(q)\in J_q^\ell\}$ and $H_r:=\{M \in H: j_r(q)\in J_q^r\}$.
 Since the rows of $M^{T}$ are not independent, to evaluate the probabilities, one needs to use a ``simple switching" argument, the details of which can be found in \cite[Claim 4.8]{litvak2019structure}. Applying the claim, we get for $i \in [k_0]$,
 \[
  \cL_{H}(\xi_q(M^T, i), 1/3)\le\max(\P_{H}(H_l), \P_{H}(H_r))\le 4/5,
\]
 which implies that $\alpha' \leq 4/5$.
 Hence, we have 
in the case $q_0=1$,
\begin{equation}\label{eq: CPcG1}
   \P_\cG(\cE_{\ref{lem: individ1}}(x,y,M^T))\le \lr{\frac{1}{c_{\ref{lem: individ1}}}}^{9c_{\ref{lem: individ1}}md}\lr{\frac{4}{5}}^{(2/3-9c_{\ref{lem: individ1}})md}.
\end{equation}
In the case $q_0=m_1/m>C_1$ we obtain,
\begin{equation}\label{eq: CPcG2}
\P_\cG(\cE_{\ref{lem: individ1}}(x,y,M^T)) \le \lr{\frac{1}{c_{\ref{lem: individ1}}}}^{9c_{\ref{lem: individ1}}md}\lr{\frac{5C_{\ref{rogozin1961increase}}\sqrt{n}}{\sqrt{3m_1d}}}^{(2/3-9c_{\ref{lem: individ1}})md}.
\end{equation}
Combining \eqref{eq: CPcF1} and \eqref{eq: CPcG1},  in the case $q_0=1$ we obtain that 
\begin{align*}
 \P_\cF(\cE_{\ref{lem: individ1}}(x,y,M))&+ \P_\cG(\cE_{\ref{lem: individ1}}(x,y,M^{T}))
  \\&\le \lr{\frac{1}{c_{\ref{lem: individ1}}}}^{9c_{\ref{lem: individ1}}md}
  \lr{\lr{\frac{4}{5}}^{(2/3-9c_{\ref{lem: individ1}})md}+\lr{\frac{1}{2}}^{(2/3-9c_{\ref{lem: individ1}})md}}
  \le \lr{\frac{5}{6}}^{md/2},
 \end{align*}
 provided that $c_{\ref{lem: individ1}}$ is small enough.
 In the case $q_0=m_1/m>C_1$, by combining \eqref{eq: CPcF2} and \eqref{eq: CPcG2}, we obtain
 \begin{align*}
     \max\{\P_\cF(\cE_{\ref{lem: individ1}}(x,y,M)), \P_\cG(\cE_{\ref{lem: individ1}}(x,y,M^{T}))\}
     &\le \lr{\frac{1}{c_{\ref{lem: individ1}}}}^{9c_{\ref{lem: individ1}}md}\lr{\frac{5C_{\ref{rogozin1961increase}}\sqrt{n}}{\sqrt{3m_1d}}}^{(2/3-9c_{\ref{lem: individ1}})md}
     \\& \le \frac{1}{2}\,\lr{\frac{C_2 n}{m_1d}}^{md/4},
\end{align*}
 provided $c_{\ref{lem: individ1}}$ is small enough and  $C_1, C_2$ are large enough. 
This completes the proof. 
\end{proof}

\subsection{{Proof of \cref{thm: lower_bdd_T}}}\label{sec:proof_lowt}
Set $\varepsilon:=0.01$ and  $m_0:=\max\{1, \lfloor  \varepsilon n/(2d)\rfloor\}=\max\{1, \lfloor  n/(200d)\rfloor\}$.
By  \cref{lem:t1}, we have
\begin{equation}\label{eq:probt1}
     \P\lr{\cE_{\ref{lem:t1}}(M)\cup \cE_{\ref{lem:t1}}(M^T)}\le \frac{7e}{n}.
\end{equation} 
Thus it is enough to consider vectors from $\cT_2\cup(\cT_3\cap \mbox{Cons}(a_3, \rho))$. 
Recall that $\btt$ and $\bt3$ were defined in (\ref{bt}) and (\ref{btn}). 
For  $j\in \{2, 3\}$ denote  
\[
\cE_j(M)=:\set{M\in \cM_{n, d}: \exists x\in \cT_j \mbox{ such that } \|(M-z\I_n)x\|_2\le 
\frac{\sqrt{c_{\ref{lem: individ1}}md}}{2B(\cT_j)}\|x\|_2}
\]
and 
\[
\cE_j(M^T)=:\set{M\in \cM_{n, d}: \exists x\in \cT_j \mbox{ such that } 
\|(M^T-z\I_n)x\|_2\le \frac{\sqrt{c_{\ref{lem: individ1}}md}}{2B(\cT_j)}\|x\|_2},
\]
where 
$$
   m= 
\begin{cases}
                       n_1, & \,\, \mbox{ when }\,\, j=2,\\
                       m_0, &\,\, \mbox{ when }\,\, j=3.
\end{cases}
$$
We prove \cref{thm: lower_bdd_T} by bounding from above the probability of the event 
$$
\cE:= \cE_{\ref{lem:t1}}(M)\cup \cE_{\ref{lem:t1}}(M^T)\cup \cE_2(M)\cup \cE_2(M^T)\cup \cE_3(M)\cup \cE_3(M^T). 
$$
Note that in the case $n_1>1$ bounds on the ratios of norms 
$$\|(M-z\I_n)x\|_2/\|x\|_2\qquad \mbox{ and }\qquad \|(M^T-z\I_n)x\|_2/\|x\|_2$$ 
coming from $\cE$ is the minimum between 
  $\omega(n, d),$
$$
  \frac{13 (\log n)^{1/4} \, \sqrt{c_{\ref{lem: individ1}} n_1 d}}{(6d)^{r+1}\, d^{1/4}}
  \ge \frac{c_0  \sqrt{n \log n} }{(6d)^{r+1}\, \sqrt{d}},
  \quad 
   \mbox{ and } \quad 
  \frac{13 (\log n)^{1/4}\, \sqrt{c_{\ref{lem: individ1}}m_0 \, d}}{(6d)^{r+1}\, d}
  \ge \frac{54 (\log n)^{1/4}\, \sqrt{n}}{(6d)^{r+2}},
$$
where $c_0>0$ is an absolute constant. 
In the case $n_1=1$ we have $\cT_1=\emptyset$, so we may assume 
$\cE_{\ref{lem:t1}}(M)= \cE_{\ref{lem:t1}}(M^T)=\emptyset$, therefore 
bounds on the ratios of norms above   
coming from $\cE$ is the minimum between 
$$
  \frac{\sqrt{c_{\ref{lem: individ1}} n_1 d}}{2\sqrt{2n }}
  \ge \frac{c_1  (\log n)^{1/4} }{d^{1/4}}
  \quad \quad 
   \mbox{ and } \quad \quad 
  \frac{\sqrt{c_{\ref{lem: individ1}}m_0 \, d}}{2\sqrt{2n }}
  \ge c_1,
$$
where $c_0>1$ is an absolute constant. This  leads to the bound announced in the Theorem~\ref{thm: lower_bdd_T}. 
Thus, it remains to estimate probabilities.

\smallskip 

Fix $j\in \{2,3\}$. Let $A$ be either $M$ or $M^T$. Define the corresponding ``norm control" event
\[
\cE_{\rm nrm}(A):=
\begin{cases}
\cE_{\ref{prp: nrm}}, & \,\, \mbox{ when }\,\, A=M,\\
\cE_{\ref{lem:nrmleft}}, &\,\, \mbox{ when }\,\, A=M^T.
\end{cases}
\]
Assume $M\in \cE_j(A)\cap \cE_{\rm nrm}(A)$. 
Then there exists $x=x(A)$ (which we fix now) such that 
\[
x\in 
\begin{cases}
\cT_2, & \,\, \mbox{ when }\,\, j=2,\\
\cT_3\cap \mbox{Cons}(a_3, \rho), &\,\, \mbox{ when }\,\, j=3 \qquad \mbox{ and } \qquad 
\|(A-z\I_n)x\|_2\le \frac{\sqrt{c_{\ref{lem: individ1}}md}}{2B(\cT_j)}\|x\|_2.
\end{cases}
\]
Without loss of generality we assume that $x^*_{n_{j-1}}=1$, which means that 
$x\in \cT_j'$ (recall that $\cT_j'$ is defined in \cref{sec:nets} for $j=2, 3$).
 By \cref{lem: l2norm_linf}, we have $\|x\|_2\le B(\cT_j)$ for $j=2, 3$. Thus,
\[
\|(A-z\I_n)x\|_2\le \frac{\sqrt{c_{\ref{lem: individ1}}md}}{2}.
\]
Let $\cN_j:=\cN_j'+\cN_j''$ be the net constructed in \cref{lem: net_Tvectors}. (Here and below, for fixed $j$, we write $u\in \cN_j'$ and $w\in \cN_j''$. When $j$ changes, we reuse the notation.) For the  $x=x(A)\in \cT_j'$ fixed above, there exist $u\in \cN_j'$ and $w\in \cN_j''\subset \mbox{span} \{\1\}$  such that
\begin{equation}\label{triplenormbound}
u_{n_{j-1}}^*\ge 1-\frac{1}{d^{3/4}}, \quad  u_{\ell}^*\le \frac{1}{4}+\frac{1}{d^{3/4}}  \quad  \forall \ell\ge n_j, \quad  \tnrm{x-(u+w)}\le \frac{\sqrt{2n}}{d^{3/4}}. 
\end{equation}
Moreover, for any $\ell\ge n_j$ and $d$ large enough,
\[
u^*_{n_{j-1}}-u_{\ell}^*\ge \Big(1-\frac{1}{d^{3/4}}\Big)-\Big(\frac{1}{4}+\frac{1}{d^{3/4}}\Big)=\frac{3}{4}-\frac{2}{d^{3/4}}>\frac{2}{3}.
\]
Hence, $u^*_{n_{j-1}}>2/3+u_{\ell}^*$ for all $\ell\ge n_j$.

Thus, for either $A=M$ or $A=M^T$, there exist $u=u(A)\in \cN_j'$ and $w=w(A)\in \cN_j''$ such that
\begin{align*}
    \|(A-z\I_n)(u+w)\|_2&\le \|(A-z\I)x\|_2+\|(A-z\I_n)(x-u-w)\|_2\\
    &\le \frac{\sqrt{c_{\ref{lem: individ1}}md}}{2}+\|A(x-u-w)\|_2+|z|\|x-u-w\|_2.
\end{align*}
Applying  \cref{prp: nrm} for $A=M$ and   \cref{lem:nrmleft} for $A=M^T$, together with (\ref{triplenormbound}), 
\[
\|A(x-u-w)\|_2\le \lr{2C_{\ref{thm: operator_bdd}}+1} \sqrt{d}\cdot\tnrm{x-u-w}\le \lr{2C_{\ref{thm: operator_bdd}}+1}\frac{\sqrt{2n}}{d^{1/4}}.
\]
Since $\|x-u-w\|_2\le \tnrm{x-u-w}\le \sqrt{2n}/d^{3/4}$ and $\abs{z}\le \sqrt{d}\log \log d$,
\[
|z|\|x-u-w\|_2\le \frac{\sqrt{2n} \log \log d}{d^{1/4}} .
\]
Note that  $c_{\ref{lem: individ1}}md\geq cn \sqrt{(\log n)/d}$ for small enough absolute constant $c>0$.  
Therefore, for large enough $n$,
\[
  \|(A-z\I_n)(u+w)\|_2\le \frac{\sqrt{c_{\ref{lem: individ1}}md}}{2}+\lr{2C_{\ref{thm: operator_bdd}}+1}\frac{\sqrt{2n}}{d^{1/4}}+\frac{\sqrt{2n}\log \log d }{d^{1/4}}\le \sqrt{c_{\ref{lem: individ1}}md}.
\]

Next we apply Lemma~\ref{lem: net_Tvectors} and  \cref{lem: individ1} twice with the  choice of parameters as in  \cref{rem:lem611}.

In the case $j=2$, by \cref{lem: net_Tvectors}, $\abs{\cN_2}\le \exp(6n_2\log d)$, and by  \cref{lem: individ1} (with $m_1=n_1<m_0, m_2=n_{2}$, $\varepsilon=0.01$ ---  see \cref{rem:lem611}), for each fixed $u\in \cN_2'$ and $w\in \cN_2''$,
\[
\P\Big( \big(\cE_{\ref{lem: individ1}}(u,w,M)\cap \Omega_{2n_1,0.01}\cap \cE_{\ref{lem: rowsum}}\big)
\cup
\big(\cE_{\ref{lem: individ1}}(u,w,M^T)\cap \Omega_{2n_1,0.01}\cap \cE_{\ref{lem: rowsum}}\big)\Big)
\le \Big(\frac{5}{6}\Big)^{n_1 d/2}.
\]
Hence, by the union bound over the net $\cN_2$, the probabilities $\P(\cE_2(M)\cap \Om_{2n_1, 0.01}\cap \cE_{\ref{prp: nrm}}\cap \cE_{\ref{lem: rowsum}})$ and $\P(\cE_2(M^T)\cap \Om_{2n_1, 0.01}\cap \cE_{\ref{lem:nrmleft}}\cap \cE_{\ref{lem: rowsum}})$ are both bounded by
\begin{align*}
   e^{6n_{2}\log d}\cdot (5/6)^{n_1d/2}\le e^{-n_1d/20}
\end{align*}
provided that $d$ is large enough and $a_2$ is small enough. Since $n_1=\left\lceil a_1\varepsilon_0 n /d\right\rceil$ and $d\le n/2$, one has
\[
e^{-n_1d/20}\le \exp\Big(-\frac{a_1}{20}\varepsilon_0 n\Big)=\exp\Big(-\frac{a_1n}{20}\sqrt{\frac{48\log n}{d}}\Big)\le \exp \left(-c\sqrt{n\log n}\right)
\]
for some absolute constant $c>0$.

For $j=3$, we argue similarly. By \cref{lem: net_Tvectors}, $\abs{\cN_3}\le \exp(6n_3\log d)$, and by  \cref{lem: individ1} (with $m_1=n_{2}>m_0$, $m_2=n_{3}$, $\varepsilon=0.01$ ---  see \cref{rem:lem611}), for each fixed $u\in \cN_3'$ and $w\in \cN_3''$, for some absolute constant $C_2>0$, 
\[
\P\Big( \big(\cE_{\ref{lem: individ1}}(u,w,M)\cap \Omega_{2m_0,0.01}\cap \cE_{\ref{lem: rowsum}}\big)
\bigcup
\big(\cE_{\ref{lem: individ1}}(u,w,M^T)\cap \Omega_{2m_0,0.01}\cap \cE_{\ref{lem: rowsum}}\big)\Big)
\le   \lr{\frac{a_2d^{1/5}}{C_2}}^{-n/800}.
\]
Thus, by the union bound over the net $\cN_3$, the probabilities $\P(\cE_3(M)\cap \Om_{2m_0, 0.01}\cap \cE_{\ref{prp: nrm}}\cap \cE_{\ref{lem: rowsum}})$ and $\P(\cE_3(M^T)\cap \Om_{2m_0, 0.01}\cap \cE_{\ref{lem:nrmleft}}\cap \cE_{\ref{lem: rowsum}})$  are both bounded by
\[
e^{6n_{3}\log d}\cdot \lr{\frac{a_2d^{1/5}}{C_2}}^{-n/800}\le e^{-c'n\log d},
\]
for $d$ is large enough and $a_3$ is sufficiently small, with  an absolute constant $c'>0$.

Finally, we bound the remaining probabilities. By \cref{lem: rowsum} (with $\tau=0.01, m=n$), 
\[
\P( \cE_{\ref{lem: rowsum}}^c)\le ne^{-d/30000}.
\]

Note that $\cE_{\ref{prp: nrm}} \subset \cE_{\ref{lem:nrmleft}}$. Indeed, if $\cE_{\ref{prp: nrm}} $ occurs, 
then by the triangle inequality
\begin{align*}
    \|M^T \1\|_2&\le \|(M^T-\E M^T) \1\|_2+\|\E M^T\1\|_2\le  \|(M^T-\E M^T)\|\cdot\| \1\|_2+d\|\1\|_2\\
    &\le (C_{\ref{thm: operator_bdd}}\sqrt{d}+d)\sqrt{n}\le (C_{\ref{thm: operator_bdd}}+1)d\sqrt{n}.
\end{align*}
Thus, by  \cref{thm: operator_bdd}, one has 
$\P(\cE_{\ref{prp: nrm}}^c)=\P((\cE_{\ref{prp: nrm}}\cup \cE_{\ref{prp: nrm}})^c)\le 2/n$.

\smallskip

By \cref{Expansion1} and \cref{Expansion} (with $\varepsilon=0.01$),
\[
\P(\Om_{2n_1, 0.01}^c)+\P(\Om_{2m_0, 0.01}^c)\le 10\,  e^{-d/480000}.
\]

Collecting the bounds for $j=2, 3$ and adding the probability in \eqref{eq:probt1}, we obtain that 
the total probability is at most $22/n$ for $d\ge C\log n$ with large enough absolute constant $C$ 
and large enough $n$. 
This completes the proof.
\qed

\section{Proof of \cref{thm: lsvlowertails}}
\label{sec:proof_lsv}
The infimum over the non-almost-constant vectors is tackled by associating it with the average distance of a row of the matrix $M-z\I_n$ from the subspace spanned by the rest of the rows. We adapt the result in \cite[Lemma 3.5]{rudelson2008littlewood}, where their original argument applies to incompressible vectors. 

\begin{lemma}[Invertibility via distance]\label{lem: inv_dis}
    Let $A$ be a $n\times n$ random matrix. Let $R_1,\dots, R_n$ be the row vectors of $A$, and let $H_k=\mbox{span}\{R_i: i\neq k\}$ be the span of all row vectors except the $k$-th row. Then for every $\delta, \rho\in (0,1)$ and every $\varepsilon\ge 0$, one has
    \[
    \P\lr{\inf_{x\in \S^{2n-1}\cap \lr{\mbox{\rm Cons}(\delta, \rho)}^c}\, \|\bar{x}A\|_2\le \frac{\varepsilon\, \rho}{\sqrt{n}}}\le \frac{1}{\delta n}\sum_{k=1}^n \P\lr{\mbox{\rm dist}\lr{R_k, H_k}\le \varepsilon}.
    \]
\end{lemma}
\begin{proof}
    Let $x\in \S^{2n-1}\cap \lr{\mbox{Cons}(\delta, \rho)}^c$. Let $X_k:=R_k^T$ be the column vector corresponding $R_k$ for $k\in [n]$. Note that for any vector $u$, and any linear subspace  $H$, $\mbox{dist}(u, H)\le \|u\|_2$. Therefore,
    \[
    \|\bar{x}A\|_2\ge \max_{k=1,\dots, n}\mbox{dist}(\bar{x}A, H_k)=\max_{k=1,\dots, n}\mbox{dist}\big(\sum_{i=1}^n\bar{x}_iX_i, H_k\big),
    \]
where $x_i$ is the $i$-th coordinate of the vector $x$.
Note that by the definition  $X_j\in H_k$ unless $j= k$. Therefore $\mbox{dist}(X_j, H_k)=0$ for every $j\neq k$. Thus,
\[
\|\bar{x}A\|_2\ge \max_{k=1,\dots, n}\mbox{dist}(\bar{x}_k X_k, H_k)=\max_{k=1,\dots, n}|x_k|\, \mbox{dist}(X_k, H_k).
\]
Let $p_k:=\P\lr{\mbox{dist}(X_k, H_k)\le\varepsilon}$. Let  $\sigma_1:=\set{k\in [n]: \mbox{dist}(X_k, H_k)>\varepsilon}$, and denote by $\cE$ the event that $\sigma_1$ contains more than $(1-\delta)n$ elements.   By Markov's inequality,
\begin{align*}
    \P(\cE^c)= \P(|\sigma_1^c|\ge \delta n)\le \frac{1}{\delta n}\E\lr{\sum_{k=1}^n\1_{\{\mbox{dist}(X_k, H_k)\le\varepsilon\}}}\le \frac{1}{\delta n}\sum_{k=1}^np_k.
\end{align*}
Let $\sigma_2(x):=\set{k\in [n]: |x_k|\ge \rho/\sqrt{n}}$. Note that $|\sigma_2^c(x)|\le (1-\delta) n$ (otherwise, $$|\{k\in [n]: |x_k-0|\le \rho/\sqrt{n}\}|\ge (1-\delta)n,$$
contradicts that $x$ is a non-almost-constant vector). Thus, $|\sigma_2(x)|\ge \delta n$.

Now, suppose that $\cE$ occurs. Then
\[
|\sigma_1|+|\sigma_2(x)|>(1-\delta)n+\delta n=n,
\]
which implies that $\sigma_1\cap\sigma_2(x)\neq \emptyset$. Let $\ell \in \sigma_1\cap \sigma_2(x)$. Then by the definitions of the sets $\sigma_1$ and $\sigma_2(x)$ one has
\[
\|\bar{x}A\|_2\ge |x_\ell |\, \mbox{dist}(X_\ell , H_\ell ) > \frac{\rho\, \ep}{\sqrt{n}}.
\]
Thus,
\[
  \P\lr{\inf_{x\in \S^{2n-1}\cap \lr{\mbox{Cons}(\delta, \rho)}^c}\|\bar{x}A\|_2\le \frac{\ep\, \rho}{\sqrt{n}}}\le \P(\cE^c)\le\frac{1}{\delta n}\sum_{k=1}^np_k.
\]
This completes the proof.
\end{proof}

To complete the picture, it remains to bound the probabilities appearing on the right-hand side in \cref{lem: inv_dis}.  
Recall that $B_\cT$ is defined in \eqref{bt} and $a_3$ is a small enough positive constant introduced in \cref{sec: inverb_almost_const}    to define $n_3$. 

\begin{lemma}\label{lem:dist_p}
   Let $C_{\ref{cor:inv_cons}}\log n\le d\le n/2$. Let $M\in \cM_{n,d}$. Let $z\in \C$. Let $R_1,\dots, R_n$ be the row vectors of $M-z\I_n$ and let $H_1=\text{span}\{R_2,\dots, R_n\}$. Let $R^*$ be a random unit vector orthogonal to $H_1$. Let $0<\rho\leq \sqrt{n}/(B_\cT d^{3/4})$. Then $R^* \in \S^{2n-1}\cap\lr{\mbox{Cons}(a_3, \rho)}^c$  with probability at least $1-23/n$ and 
    \[
    \P(\text{\rm dist}(R_1, H_1)\le \rho /\sqrt{n})  \le \frac{23}{n} +  \frac{C_{\ref{lem: weak_smallball_notalmost}}}{\sqrt{a_3 d}}        +
    C_{\ref{lem: weak_smallball_notalmost}}\exp\left({-a_3^2d/C_{\ref{lem: weak_smallball_notalmost}}}\right).
    \]
\end{lemma}

\begin{proof}
    It is clear that $R^*$ is a random vector that depends only on $R_2,\dots, R_n$ (hence independent of $R_1$) 
    and that 
    \[
    \text{dist}(R_1, H_1)\ge \abs{\inn{R^*, R_1}}.
    \]
    Let $$\cE:=\{R^*\in \S^{2n-1}\cap \lr{\mbox{Cons}(a_3, \rho)}^c\}. $$
   
    Then
    \begin{align*}
         \P\Big( \text{dist}(R_1, H_1)\le \frac{\rho}{\sqrt{n}}\Big) &\le \P\left(\abs{\inn{R^*, R_1}}\le  \frac{\rho}{\sqrt{n}}\right)
         \le  \P\left(\left \{ \abs{\inn{R^*, R_1}}  \le \frac{\rho}{\sqrt{n}} \right \} \cap \cE \right)+\P(\cE^c)\\
         &=   \P\Big( \abs{\inn{R^*, R_1}}  \le \frac{\rho}{\sqrt{n}}\,\,  \vert\,\, \cE \Big)\P(\cE)+\P(\cE^c). 
    \end{align*}
Since $R^*$ is independent of $R_1$, conditioning on $\cE$, 
we have 
$$ 
 \P\Big( \abs{\inn{R^*, R_1}}  \le \frac{\rho}{\sqrt{n}}  \,\,  \vert\,\,  \cE \Big)\le \Ls(\inn{R^*, \xi}, 
 \frac{\rho}{\sqrt{n}}),
$$
where $\xi:=R_1+ze_1$ is uniformly distributed on the set of all vectors in $\{0,1\}^n$ containing exactly $d$ ones.
Therefore,  by \cref{lem: weak_smallball_notalmost},
    \[
  \P\Big( \abs{\inn{R^*, R_1}}  \le \frac{\rho}{\sqrt{n}}  \,\,  \vert\,\,  \cE \Big)\le \frac{C_{\ref{lem: weak_smallball_notalmost}}}{\sqrt{a_3 d}}+C_{\ref{lem: weak_smallball_notalmost}}
   \exp{(-a_3^2 d/C_{\ref{lem: weak_smallball_notalmost}})}.
    \]
To bound $\P(\cE^c)$, let $M'$ be the $(n-1)\times n$ random matrix with row vectors $R_2,\dots, R_n$, i.e., the submatrix of $M-z\I_n$ obtained by removing the first row. By the definition of $R^*$, 
\[
M'R^*=0.
\]
Thus, on the event $\cE^c$, there exists $x\in \S^{2n-1}\cap \mbox{Cons}(a_3, \rho)$ such that $M'x=0$.
By replacing $n$ with $n-1$, one can  check that the proof of \cref{cor:inv_cons} remains valid for $M'$ instead of $M-z\I_n$. Thus,
\[
\P(\cE^c)\le 23/n.
\]
    This completes the proof.
\end{proof}

We are now ready to  prove the lower bound on the smallest singular value.

\begin{proof}[Proof of \cref{thm: lsvlowertails}] We choose parameters 
$$
  \rho =\sqrt{n}\, \min\left\{\frac{1}{B_\cT\, d^{3/4}}, \, \sqrt{\theta(n,d)}\right\}
  \quad \quad \mbox{ and } \quad \quad    \varepsilon=\frac{\rho}{\sqrt{n}} 
$$
and  set 
$$
A=M^T-\bar{z}\I_n,\,\,\,  S_1=\S^{2n-1}\cap  \mbox{Cons}(a_3, \rho),
 \,\,\, \mbox{ and } \,\,\, S_2=\S^{2n-1}\cap  (\mbox{Cons}(a_3, \rho))^c.
$$
  Consider the event
\[
\cE:=\set{\inf_{v\in \S^{2n-1}}\|Av\|_2\le \frac{\varepsilon\, \rho}{\sqrt{n}}} = \cE_1\cup \cE_2,
\]
where 
\[
\cE_1:=\set{\inf_{x\in S_1}\|A x\|_2\le \frac{\varepsilon\, \rho}{\sqrt{n}}}\quad \quad \mbox{ and } 
\quad \quad 
\cE_2:=\set{\inf_{x\in S_2}\|Ax\|_2\le \frac{\varepsilon\, \rho}{\sqrt{n}}}.
\]
Since $\rho\leq \sqrt{n}/(B_\cT d^{3/4})$, $\varepsilon \rho/\sqrt{n}=\rho^2/n\leq \theta(n, d)$, and $|\bar{z}|=|z| \leq \sqrt{d} \log \log d$, 
 \cref{cor:inv_cons}  implies that $$\P(\cE_1)\leq 23/n.$$ 
 Let $H_k=\mbox{span}\{R_i: i\neq k\}$, $k\leq n$. Using  that 
 $\mbox{dist}\lr{R_k(A), H_k}$, $k\leq n$, have the same distributions and   
 applying  \cref{lem: inv_dis} with $\delta=a_3$, we observe 
\begin{align*}
    \P(\cE_2)&\le  \frac{1}{a_3}\P\lr{\mbox{dist}\lr{R_1(A), H_1}\le \varepsilon}.
\end{align*}
Applying \cref{lem:dist_p} (note again that $\rho\leq \sqrt{n}/(B_\cT d^{3/4})$)), 
\begin{align*}
    \P(\cE)&\le  \P\lr{\cE_1}+ \P\lr{\cE_2}\le  \frac{23}{n}+\frac{1}{a_3}\lr{     \frac{23}{n}   +
           \frac{C_{\ref{lem: weak_smallball_notalmost}}}{\sqrt{a_3 d}}        +
      C_{\ref{lem: weak_smallball_notalmost}}\exp\left(-a_3^2d/C_{\ref{lem: weak_smallball_notalmost}}\right)}
      \leq \frac{C}{\sqrt{d}},
\end{align*}
where $C$ is a large enough absolute constant. This proves the bound on the singular value 
\begin{equation*}
    \frac{\varepsilon\, \rho}{\sqrt{n}}= \frac{\rho^2}{n} = \min\left\{\frac{1}{B_\cT^2\, d^{3/2}},\,  \theta(n,d)\right\}
\end{equation*}
Finally we analyze this minimum. 
 
First assume that $n_1=1$ (that is, $d\geq c\,  n^{2/3}(\log n)^{1/3}$). In this case, for large $n$, 
$$
 \min\left\{\frac{1}{B_\cT^2\, d^{3/2}},\,  \theta(n,d)\right\}=
 \min\left\{\frac{1}{2n \, d^{3/2}},\,  \frac{c_1 (\log n)^{1/4}}{d^{1/4}}\right\}=\frac{1}{2n \, d^{3/2}}\geq 
 \frac{1}{n^{5/2}}.
$$

Next assume that $n_1>1$ (that is, $d\leq c\,  n^{2/3}(\log n)^{1/3}$). It is not difficult to see that 
 \begin{equation}\label{eq:smsing}
     \min\left\{\frac{1}{B_\cT^2\, d^{3/2}},\,  \theta(n,d)\right\}
  = \min\left\{\frac{6^4 \sqrt{\log n}}{(6d)^{2r+4}\, d^{3/2}},\, \frac{\sqrt{d}}{4 \sqrt{n}}\right\}.
\end{equation}
Recall that $\ell_0^r< n_1\le \ell_0^{r+1}$. Denote 
\begin{equation*}
    \alpha:=\frac{\log(6d)}{\log \ell_0}\in (2,  C \log \log n]
\end{equation*}
for some absolute constant $C\geq 1$. Then 
$$
  (6d)^{2r+4}\, d^{3/2}=\ell_0^{2\alpha (r+1)} d^{7/2}\geq n_1^{2\alpha}d^{7/2}
  \ge n_1^{4}d^{7/2}\ge c_0 n^4 (\ln n)^2 /d^{5/2}  \ge c_0 n^{3/2},
$$
for some absolute constant $c_0>0$. 
This shows that for large enough $n$, the minimum in (\ref{eq:smsing}) attains at the first term. 
On the other hand, 
$$
  (6d)^{2r+4}\, d^{3/2}=\ell_0^{2\alpha r} d^{11/2}\leq n_1^{2\alpha} d^{11/2}\leq 
   \left( n^2 (\log n )/d^3 \right)^{\alpha } d^{11/2},
$$
Therefore the minimum in (\ref{eq:smsing})  is bounded below by 
$$
  \frac{ d^{3\alpha-5.5}}{n^{2 \alpha} (\log n)^{\alpha -1/2}}\geq n^{-2 \alpha} . 
$$
Finally note that \\
(1) if  
$d\geq  C_1 (\log n)^{21}$ then $\alpha \leq 2.1$, so the bound is larger than 
$$
  \frac{ d^{0.8}}{n^{4.2} (\log n)^{1.6}} ; 
$$
(2) if  
$d\geq   (\log n)^{2}$ then  $\alpha \leq 4.1$ (for large enough $n$), so the bound is larger than
$$
  \frac{ d^{6.8}}{n^{8.2} (\log n)^{3.6}} ; 
$$
   This 
  completes the proof of \cref{thm: lsvlowertails} and the bound in Remark~\ref{sing-comp-bb}.
\end{proof}

\section{Singularity threshold}\label{appx}

In \cite{ferber2022singularity}, Ferber, Kwan, and Sauermann established that a random matrix uniformly drawn from $\cM_{n,d}$ is asymptotically almost surely non-singular whenever $\min{(d,n-d)} \ge (1+\varepsilon)\log n$. In \cite{APsparseRCM}, Aigner--Horev and Person  mentioned that this threshold is sharp, moreover,  when $d\le (1-\varepsilon)\log n$, a random  matrix  asymptotically almost surely contains a zero column.
 However, it seems that a proof of this fact does not appear in the existing literature, so we provide 
 a direct  proof of this threshold.

\begin{prop} Let $d, n$ be positive integers and $d\le n/2$.  
    Let $M$ be an $n\times n$ random matrix where each row is independently and uniformly chosen from all $n$-dimensional binary vectors $\{0,1\}^n$ with exactly $d$ ones. Let $X$ denote the number of zero columns in $M$. Then the following holds.
    \begin{enumerate}
        \item Assume that $d=\log n - g(n)$ for some non-negative function $g$ satisfying $g(n) \longrightarrow \infty$ as $n \rightarrow \infty$. Then,
        \[
        \lim_{n\to \infty} \P(X\ge 1)=1.
        \]
        That is, asymptotically (in $n$) the matrix $M$ almost surely has at least one zero column. In particular, it holds for 
        $d\le (1-\varepsilon)\log n$ (for  any fixed $\varepsilon>0$). 
        \item For any fixed $d\geq 1$ and fixed $\theta \in (0,1)$,
        \[
         \liminf_{n\to \infty} \P(X > \theta \E X ) 
         \geq (1-\theta)^2.
        \]
        In particular, 
        \[
         \liminf_{n\to \infty}\P(X>\theta  e^{-d} n /2) 
         \geq (1-\theta)^2,
        \]
        that is, for large enough $n$, the nullity of the matrix $M$ is bounded from below by  a constant proportion of $n$ with a positive probability.
    \end{enumerate}
\end{prop}

\begin{proof}
 In both cases we will use the Paley-Zygmund inequality, which states that for any $\lam \in [0, 1]$,
\[
\P(X>\lam \,  \E X)\ge (1-\lam)^2\, \frac{(\E X)^2}{\E X^2}.
\]
We estimate $\E X$ and $\E X^2$ separately. 
   Note that $X=\sum_{i=1}^n I_i$ where $I_i$ is the indicator function of the event that $i$-th column is  zero. Thus,
   \[
   q:=\mathbb{E}X=\sum_{i=1}^nE(I_i)=n(1-d/n)^n
   \]
   and 
 \[
\E X^2=\sum_i \E I_i^2+\sum_{i\neq j}\E I_iI_j =q+\sum_{i\neq j}\P(I_i=1, I_j=1).
\]
Since $I_i=1, I_j=1$ means that $i$-th and $j$-th columns of $M$ are zero columns and since $\P$ is the uniform 
probability on $\cM_{n,d}$, 
\[
   \P(I_i=1, I_j=1)=
   \lr{\frac{\binom{n-2}{d}}{\binom{n}{d}}}^n=
   \lr{\frac{(n-d)(n-d-1)}{n(n-1)}}^n\le (1-d/n)^{2n} = (q/n)^2.
\]
Therefore, 
\[
\E X^2\le q+n(n-1)(q/n)^2=q+\frac{n-1}{n}q^2,
\]
and, consequently, 
$$
  \frac{(\E X)^2}{\E X^2}\ge  \frac{q^2}{q+\frac{n-1}{n}q^2}=\frac{1}{1/q+ \frac{n-1}{n}}. 
$$

\noindent 
    \textbf{Case 1:}  {\it $d=\log n - g(n)$, $g(n)\geq 0$, $g(n) \longrightarrow \infty$.}   
In this case, using  $\log(1-x)\ge -x-x^2$ on $[0, 1/2]$ and $d\leq n/2$, we have 
$$
 n\log(1-d/n)\ge -d-d^2/n \geq -\log n + g(n) - (\log^2 n)/n 
 \geq -\log n + g(n)- \log 2
$$
for large enough $n$. Thus,
\[
 q=\E X\ge  ne^{ -\log n + g(n)- \log 2}= e^{ g(n)}/2 \longrightarrow \infty
 \quad \mbox{ as } n\to \infty.
\]
Applying the Paley-Zygmund inequality with $\lam=0$, 
\[
 \P(X\geq 1)=\P(X>0)\ge \frac{(\E X)^2}{\E X^2}\ge 
  \frac{1}{1/q+ \frac{n-1}{n}}\longrightarrow 1  \quad \mbox{ as } n\to \infty.
\]

\noindent 
\textbf{Case 2: } {\it $d\geq 1$ and $\theta \in (0,1)$ are fixed.} 
Note that in this case $q/n=(1-d/n)^n\to  e^{-d}$ as $n\to \infty$, 
in particular, $q\to \infty$. 
Therefore, applying the Paley-Zygmund inequality with $\lam=\theta$, we have  
\[
 \P(X>\theta \E X)\ge (1-\theta)^2\frac{(\E X)^2}{\E X^2}\geq 
 \frac{(1-\theta)^2 }{1/q+ \frac{n-1}{n}}\longrightarrow (1-\theta)^2   \quad \mbox{ as } n\to \infty. 
\]
The ``in particular" part follows as $q/n \to e^{-d}$.  
\end{proof}

Finally, given the resemblance between random combinatorial matrices and i.i.d. Bernoulli random matrices, 
we would like to formulate the following conjecture. For corresponding results on 
Bernoulli($d/n$) random matrices we refer to 
\cite{basak2021sharp, litvak2022singularity, Huang-Bern, jain2020sharp, jain2020sharpsing}
and references therein.

 \begin{con}\label{con1}
   Let $\Omega_0$, $\Omega_r$,  and $\Omega_c$  denote the events that $M$ has a zero column, 
    that $M$ contains two identical rows, and  $M$ contains two identical columns respectively. Then 
\begin{equation}\label{eqcon}
     \P(\mbox{$M$ is singular})=(1+o(1))\left(\P(\Omega_0) + \P(\Omega_r) + \P(\Omega_c) \right).
\end{equation}
\end{con}

 Analyzing probabilities in Conjecture~\ref{con1}, note that $M$ is singular iff $E_{n}-M$ is singular, where 
 $E_n$ denotes the matrix with all entries equal to 1. Therefore, we assume that $d \leq n/2$. For $d=n/2$, Nguyen \cite{nguyen2013singularity} conjectured that a random combinatorial matrix $M$ is singular with probability $(\frac{1}{2}+o(1))^n$, which has recently been confirmed  by Jain, Sah, and Sawhney 
   \cite[Theorem~1.8] {jain2020sharpsing}. Note that for $d=n/2$, this bound is weaker than  (\ref{eqcon}). 
   Furthermore, for $d< c_0 n$ (for a suitable $c_0>0$) we have  $\P(\Omega_c)>\P(\Omega_r)$ and  if 
    $d\geq 2\log n$ then  $\P(\Omega_0)$ dominates $\P(\Omega_c)$. Moreover, for any 
   fixed positive constant $c < 1/2$,  if $d \leq c n$ then  $\P(\Omega_0)$  dominates $\P(\Omega_r)$. 
   Therefore, we may formulate a weaker  conjecture.
   
   \begin{con}
    Let $c \in (0,1/2)$ be a fixed constant, assume that $2\log n \leq d \leq c n$. Then
    \[
    \P(\mbox{$M$ is singular})=(1+o(1))\P(\Omega_0).
    \]    
\end{con}

\bibliographystyle{abbrv}
\bibliography{Singular_value}

\vspace{1cm}

\par\noindent
\textsc{Dongbin Li} 
\\
\textit{Email address:} \texttt{dongbinli@snnu.edu.cn}
\\
School of Mathematics and Statistics, Shaanxi Normal University, Xi'an, 710119, China.

\vspace{1cm}
\par\noindent
\textsc{Alexander E. Litvak} 
\\
\textit{Email address:} \texttt{alitvak@ualberta.ca}
\\
%
Dept. of Math. and Stat. Sciences, 
University of Alberta, Edmonton, AB, T6G 2R3, Canada.

\vspace{1cm}
\par\noindent
\textsc{Tingzhou Yu} 
\\
\textit{Email address:} \texttt{tingzho1@ualberta.ca}
\\
%
Dept. of Math. and Stat. Sciences, 
University of Alberta, Edmonton, AB, T6G 2R3, Canada.

\end{document}